\documentclass[hidelinks,onefignum,onetabnum]{siamart190516}

\usepackage{lipsum}
\usepackage{amsfonts}
\usepackage{graphicx}
\usepackage{subcaption}
\usepackage{epstopdf}
\usepackage{algorithmic}

\usepackage{etoolbox}   
\usepackage{cleveref}   

\ifpdf
  \DeclareGraphicsExtensions{.eps,.pdf,.png,.jpg}
\else
  \DeclareGraphicsExtensions{.eps}
\fi

\newsiamremark{remark}{Remark}
\newsiamremark{hypothesis}{Hypothesis}
\newsiamremark{example}{Example}
\crefname{hypothesis}{Hypothesis}{Hypotheses}
\newsiamthm{claim}{Claim}
\newsiamremark{fact}{Fact}
\crefname{fact}{Fact}{Facts}

\headers{Fully discrete Active Flux method for Euler}{E. Chudzik, C. Helzel, A. Porfetye}

\title{A fully discrete truly multidimensional
  Active Flux method for the two-dimensional
  Euler equations\thanks{Preprint version submitted to arXiv.} }

\author{Erik Chudzik\thanks{Heinrich-Heine-University D\"usseldorf, Germany
  (\email{erik.chudzik@hhu.de}).}
\and Christiane Helzel\thanks{Heinrich-Heine-University D\"usseldorf, Germany
  (\email{christiane.helzel@hhu.de}).}
\and Amelie Porfetye\thanks{Heinrich-Heine-University D\"usseldorf,
  Germany
(amelie.porfetye@hhu.de)\funding{This work was funded by the German Science Foundation (DFG)
	under project number 525800857.}}}

\usepackage{amsopn}

\newcommand{\ignore}[1]{}
\ifpdf
\hypersetup{
  pdftitle={AF Euler},
  pdfauthor={E.Chudzik, C.Helzel, and A.Porfetye}
}
\fi

\begin{document}

\maketitle

\begin{abstract}
The Active Flux method is a finite volume method for hyperbolic
conservation laws that uses both cell averages and point
values as degrees of freedom. Several versions of such
methods are currently under development. We focus on third order accurate,
fully discrete Active Flux methods with compact stencil in space and
time. These methods require
exact or approximate evolution operators to update the point
value degrees of freedom.  In our method, these operators are 
provided by the method of bicharacteristics, which solves a local linearisation
of the nonlinear hyperbolic problem.
We analyse the linearisation error and
propose a new version of a fully discrete Active Flux method, which is
third order accurate for smooth problems.
Furthermore, we propose limiting strategies for our fully discrete Active Flux method
applied to the Euler equations of gas dynamics that guarantee positivity of
density and pressure. Finally, we discuss the implementation of
boundary conditions for these fully discrete Active Flux methods.
Numerical results confirm that the
method provides accurate results even when using coarse grids.   
\end{abstract}



\section{Introduction}
The Active Flux method, originally proposed by Eymann and Roe
\cite{eymann2011active,eymann2013multidimensional}, is a relatively
new variant of a finite volume method which is attracting increasing interest.
In a series of papers
\cite{article:Roe2017,article:Roe2018,article:Roe2020,article:Roe2021,proc:Roe2025},
Roe described his motivation to reconsider fundamental concepts in computational fluid
dynamics and proposed methods that use a globally continuous,
Riemann solver free, truly multi-dim\-en\-sio\-nal approach with compact stencil. 

While classical finite volume methods use cell average values of the
conserved quantities as degrees of
freedom, Active Flux methods use in addition point values along
the grid cell boundaries. This allows to construct a globally continuous piecewise
quadratic reconstruction that can be used to obtain third order
accurate methods. In the original version of the Active Flux
method, see for example \cite{eymann2013multidimensional}, the point values for
advection and acoustics have been evolved in time using exact
evolution operators. Numerical fluxes of the finite volume method for
the evolution of the cell average values
of the conserved quantities are computed from point
values at the previous time level, an intermediate time level and the
new time level using  Simpson's rule.  The resulting fully discrete methods use the
conservative form of the partial differential equation to evolve the cell average values
and the characteristic form to evolve the point values, see also
\cite{article:Abgrall2023} for a related discussions. Fully discrete Active
Flux methods for acoustics have been studied extensively in
\cite{article:BHKR2019,eymann2013multidimensional,FR2015,article:SR2023}.
Barsukow \cite{article:Barsukow2023}
showed, that on Cartesian grids, the Active Flux method for acoustics, with
exact evolution operator, preserves all steady states. This
makes Active Flux methods interesting candidates for approximating flow
problems in the low Mach number regime as pointed out in \cite{article:BHKR2019}.  

For more complex nonlinear hyperbolic problems, and in particular the
Euler equations of gas dynamics, exact evolution
operators are not available. Therefore, various approaches to the
development of finite volume methods inspired by the original Active
Flux approach are currently being developed. 
Eymann and Roe
\cite{eymann2013multidimensional}
proposed splitting methods for the Euler equations,
which separately approximate acoustic
wave propagation and  nonlinear transport. Such an approach was also
considered in the PhD thesis of Fan \cite{PhD:Fan2017}. In
\cite{article:HKS2019}, an ADER approach was proposed for the update of
the point value degrees of freedom. Abgrall and Barsukow
\cite{article:AB2023} proposed a method of lines approach for
one-dimensional hyperbolic problems, which also
allows the method to be extended to arbitrary order of accuracy. Most
recently, the method of lines approach for the two dimensional Euler
equations is being developed very actively. The resulting methods are
named either generalised Active Flux methods or PAMPA schemes 
\cite{article:AB2023,article:ABK2025,preprint:ALB2025,article:DBK2025}.
The special choice of the degrees of freedom allows to construct
spatial discretisations with compact stencil. However, a method of
lines approach increases the stencil in each stage of the time
stepping method.  Roe et al. \cite{article:Roe2021,article:SR2023} conjectured that
fully-discrete methods with compact stencil in space and time lead to
more accurate approximations on coarse grids. This motivates our work
on fully discrete Active Flux methods.
To construct such methods for general hyperbolic problems, we recently explored the
method of bicharacteristics for the evolution of point values
\cite{article:CHL2024}.
This was based on earlier work by  Luk\'a\v{c}ov\'a et al., where
the method of bicharacteristics has been used extensively
to develop truly multi-dimensional
finite volume methods for hyperbolic conservation laws, see for
example \cite{article:LMW2000,article:LSW2002}. These earlier methods
used cell average values as degrees of freedom and piecewise constant
or limited piecewise linear  reconstructions as common for classical finite
volume methods.

Here we present a fully discrete third order accurate Active Flux method for the Euler
equations using the method of bicharacteristics for the evolution of
the point values.
Limiting concepts will be introduced which allow the approximation of
complex discontinuous solution structures.
Currently, limiting concepts for generalised Active
Flux and PAMPA methods are actively
being developed, see
\cite{article:ABK2025,preprint:ALB2025,article:DBK2025} for details.
These methods limit point values as well as numerical fluxes to
maintain positivity of density and pressure. Furthermore, additional
numerical viscosity is introduced near strong shocks to eliminate
unphysical oscillations.
For the fully discrete Active Flux method, a positivity preserving
flux limiter for
advective transport was recently developed and used
to approximate the Vlasov-Poisson problem \cite{article:KCH2025}.

The rest of the paper is organised as follows: In \cref{sec:2} we
introduce the Cartesian grid Active Flux method for the two-dimensional Euler
equations and discuss its accuracy.
In \cref{sec:3} we consider local linearisations for discontinuous
solutions and present our limiting approach for point
values as well as fluxes.   Numerical results
illustrating the accuracy of our  method even on coarse grids
are shown in \cref{sec:4}.


\section{A third order accurate Active Flux method for the Euler
  equations using the method of bicharacteristics}
\label{sec:2}
In this section we present our unlimited Active Flux method for the
two-dimen\-sio\-nal Euler equations of gas dynamics.
The method uses both the conservative and a
non-conservative form of the hyperbolic system of partial differential
equations.
In  conservative form we consider
\begin{equation}
\partial_t \mathbf{q} + \partial_x \mathbf{f}(\mathbf{q}) + \partial_y
\mathbf{g}(\mathbf{q}) = \mathbf{0},
\end{equation}
where $\mathbf{q}(x,y,t)$ is a vector of conserved quantities and
$\mathbf{f}(\mathbf{q})$ and $\mathbf{g}(\mathbf{q})$ are vector
valued flux functions.
For the Euler equations 
the vector of conserved quantities is given by
$\mathbf{q} := (\rho, \rho u, \rho v, E)^T$ and the fluxes have the form
\begin{equation}
  \begin{split}
\mathbf{f}(\mathbf{q}) & := \left(\rho u, \rho u^2 + p, \rho u v,
  u(E+p)\right)^T, \\
\mathbf{g}(\mathbf{q}) & := \left( \rho v, \rho u
  v, \rho v^2 + p, v(E+p)\right)^T,
\end{split}
\end{equation}
where $\rho$ denotes density, $(u,v)$ the two-dimensional velocity
field, $E$ the total energy and $p$ the pressure. We use
the ideal gas equation of state
$E = \frac{p}{\gamma-1} + \frac{1}{2} \rho (u^2 + v^2)$ with $\gamma =
1.4$ to close the system.

For smooth solutions, the Euler equations can equivalently be written
in various nonconservative forms. We will use  a representation 
using primitive variables that leads to a
quasilinear hyperbolic system of the general form
\begin{equation}\label{eqn:quasilinearEuler}
\partial_t \mathbf{u} + A(\mathbf{u}) \partial_x \mathbf{u} +
B(\mathbf{u}) \partial_y \mathbf{u} = 0.
\end{equation}
Here, $\mathbf{u}(x,y,t)$ is the vector of primitive variables given by
$\mathbf{u}:=\left(\rho,u,v,p\right)^T$ and the matrices
$A(\mathbf{u})$ and $B(\mathbf{u})$ are given by
\begin{equation}\label{eqn:AandB}
  A(\mathbf{u}) := \left( \begin{array}{cccc}
                           u & \rho & 0 & 0 \\
                           0 & u & 0 & 1/\rho \\
                           0 & 0 & u & 0 \\
                           0 & \gamma p & 0 & u\end{array}\right),
                       \quad
B(\mathbf{u}) := \left( \begin{array}{cccc}
                           v & 0 & \rho & 0 \\
                           0 & v & 0 & 0 \\
                           0 & 0 & v & 1/\rho \\
                           0 & 0 & \gamma p & v \end{array}\right).                       
\end{equation}

We will now give a brief description of the Active Flux method on
two-dimen\-sio\-nal Cartesian grids. More details can be found in
\cite{article:ABK2025,article:Barsukow2023,article:BHKR2019,article:CCH2023,article:CHL2024,article:HKS2019}.
Let $(x_{i-\frac{1}{2}},x_{i+\frac{1}{2}}) \times (y_{j-\frac{1}{2}},
y_{j+\frac{1}{2}})$ be the grid cell $(i,j)$. The cell average values
of the conserved quantities in grid cell $(i,j)$ at time $t_n$ are
denoted by $\bar{Q}_{i,j}^n$. In addition to cell averages, Active
Flux methods also use point values at all four corners $(x_{i\pm
  \frac{1}{2}},y_{j\pm \frac{1}{2}})$ of the grid cell and the four midpoints
$(x_{i\pm \frac{1}{2}},y_j)$ and $(x_i,y_{j\pm \frac{1}{2}})$ of the
grid cell interfaces. At time $t_n$, these point values in conservative
variables are denoted by
$Q_{i-\frac{1}{2},j}^n$, $Q_{i-\frac{1}{2},j-\frac{1}{2}}^n$,
$Q_{i,j-\frac{1}{2}}^n$, $Q_{i+\frac{1}{2},j-\frac{1}{2}}^n$,
$Q_{i+\frac{1}{2},j}^n$, $Q_{i+\frac{1}{2},j+\frac{1}{2}}^n$,
$Q_{i,j+\frac{1}{2}}^n$ and $Q_{i-\frac{1}{2},j+\frac{1}{2}}^n$.
Analogously, we denote point values and cell average values
in primitive variables with $U$. Point values in primitive variables
can straight forwardly be computed from point values in conservative
variables. In order to get third order accurate cell average values in
primitive variables we first compute the point value $Q_{i,j}$ in
conservative variables at the midpoint of the grid by inverting
Simpson's rule. From this point value we can again easily compute the
point value in primitive variables, denoted by $U_{i,j}$. Now we can use
Simpson's rule to compute cell average values in primitive variables
denoted by $\bar{U}_{i,j}$.

We assume that all cell average values and all point values at time
$t_n$ are known. From these values a globally continuous piecewise quadratic
reconstruction can be computed. Basis functions for this
reconstruction can be found in \cite[Table 3]{article:HKS2019}.

The cell average values of conserved variables are evolved in time
using a finite volume method, which can be written
in the classical form
\begin{equation}\label{eqn:2dfvm}
\bar{Q}_{i,j}^{n+1} = \bar{Q}_{i,j}^n - \frac{\Delta t}{\Delta x} \left(
  F_{i+\frac{1}{2},j} - F_{i-\frac{1}{2},j} \right) - \frac{\Delta
  t}{\Delta y} \left( G_{i,j+\frac{1}{2}}-G_{i,j-\frac{1}{2}}\right).
\end{equation}
The numerical fluxes are computed using
Simpson's rule, i.e.\ the flux has the form
\begin{equation}\label{eqn:F}
  \begin{aligned}
    F_{i+\frac{1}{2},j}  := &\frac{1}{36}\left( \phantom{4}\mathbf{f}(Q_{i+\frac{1}{2},j-\frac{1}{2}}^n)
       + \phantom{1}4 \mathbf{f}(Q_{i+\frac{1}{2},j}^n) +\phantom{4}
      \mathbf{f}(Q_{i+\frac{1}{2},j+\frac{1}{2}}^n) \right. \\
     & +  \phantom{1} \left. 4\mathbf{f}(Q_{i+\frac{1}{2},j-\frac{1}{2}}^{n+\frac{1}{2}})  + 16
    \mathbf{f}(Q_{i+\frac{1}{2},j}^{n+\frac{1}{2}}) + 4
    \mathbf{f}(Q_{i+\frac{1}{2},j+\frac{1}{2}}^{n+\frac{1}{2}})\right.  \\
    & + \phantom{1} \left.\phantom{4}\mathbf{f}(Q_{i+\frac{1}{2},j-\frac{1}{2}}^{n+1})
       + \phantom{1}4 \mathbf{f}(Q_{i+\frac{1}{2},j}^{n+1}) +\phantom{4}
      \mathbf{f}(Q_{i+\frac{1}{2},j+\frac{1}{2}}^{n+1}) \right) 
    \end{aligned}
\end{equation}
and analogously for $G_{i,j+\frac{1}{2}}$.
Calculating the point values at the intermediate time
$t_{n+\frac{1}{2}}$ and the final time $t_{n+1}$ is the crucial step
in a fully discrete Active Flux method. 
We compute these point
values  in primitive variables using the method of bicharacteristics for locally linearised Euler equations.

More precisely, we use the  EG2 evolution operator, introduced in \cite{article:LSW2002}, for the linearised Euler equations
\begin{equation}\label{eqn:linEuler}
\partial_t \mathbf{v} + \tilde{A} \partial_x \mathbf{v} + \tilde{B}
\partial_y \mathbf{v} = \mathbf{0},
\end{equation}
where $\mathbf{v}:=(\rho,u,v,p)$ and $\tilde{A}:=A(\mathbf{u}')$,
$\tilde{B}:=B(\mathbf{u}')$ are constant matrices, corresponding
to the matrices from (\ref{eqn:AandB}) evaluated at a temporally and
spatially constant state
$\mathbf{u}':=(\rho',u',v',p')^T$.

The EG2 evolution operator approximates the solution of the linearised
Euler equations 
at a point $P=(\bar{x},\bar{y},t_{n}+\tau)$, with $\tau \in (0,\Delta
t]$, provided that the
solution at time $t_n$ is known
in a sufficiently large neighbourhood of the point $(\bar{x},\bar{y})$, which
depends on $\tau$. It is given by
\begin{equation}
  \label{eqn:EG2}
  \begin{aligned}
\rho(P) & =\rho({P'})-2\frac{p({P'})}{c'^2}+
\frac{1}{\pi}\int_{0}^{2\pi}\frac{p(Q(\theta))}{c'^2}-\frac{\rho'}{c'}u(Q(\theta))
\cos(\theta)\\
& \hspace*{5.5cm} -\frac{\rho'}{c'}v(Q(\theta)) \sin(\theta)
\  {\rm d} \theta + {\cal O}(\tau^3),\\
u(P) & =\frac{1}{\pi}\int_{0}^{2\pi}
-\frac{p(Q(\theta))}{\rho'c'}\cos(\theta)+u(Q(\theta))
\left(2\cos^2(\theta)-\frac{1}{2}\right) \\
& \hspace*{4cm} +2v(Q(\theta)) \sin(\theta)\cos(\theta)
\,  {\rm d} \theta + {\cal O}(\tau^3),\\
v(P) & =\frac{1}{\pi}\int_{0}^{2\pi}
-\frac{p(Q(\theta))}{\rho'c'}\sin(\theta)+2u(Q(\theta)) \sin(\theta)
\cos(\theta)\\
& \hspace*{4cm}+v(Q(\theta)) \left(2\sin^2(\theta)-\frac{1}{2}\right)  \,
{\rm d} \theta + {\cal O}(\tau^3),\\
p(P) &=-p({P'})+\frac{1}{\pi}\int_{0}^{2\pi}
p(Q(\theta))  -\rho'c'u(Q(\theta)) \cos(\theta) \\
& \hspace*{4.2cm}
-\rho'c'v(Q(\theta)) \sin(\theta) \,  {\rm d} \theta + {\cal O}(\tau^3),
  \end{aligned}
\end{equation}
where $P' = (\bar{x}-u'\tau,\bar{y}-v'\tau,t_n)$, $Q(\theta) =
(\bar{x}-(u'-c' \cos(\theta))\tau, \bar{y}-(v'-c'\sin(\theta))\tau,t_n)$ and $c' =
\sqrt{\gamma \frac{p'}{\rho'}}$.
We introduce the notation $L_{EG2}(\mathbf{u}',\tau,\mbox{rec})$ to write the
evolution operator in a compact form, where $\mathbf{u}'$ describes the state
used for the linearisation and $\mbox{rec} \in \{\mbox{cpq},\mbox{pc}\}$ describes the
reconstruction. In this paper we consider the continuous
piecewise quadratic (cpq) Active Flux reconstruction as well as a piecewise
constant (pc) reconstruction.   

Note that this approach is truly multi-dimensional as all directions of wave
pro\-pa\-gation are taken into account by integrating over the base of the
characteristic cone. In our unlimited Active Flux method
the globally continuous piecewise
quadratic reconstruction  is used to describe the
functions at the right hand side of (\ref{eqn:EG2}). Thus, the
approximation of the point value for the linearised Euler equations
is third order accurate in space and
time. In \cite{article:CHL2024} we studied Active Flux methods for
linear acoustics and the linearised Euler equations using the EG2
operator. Our computational
studies on an equidistant grid with mesh width $h:=\Delta x = \Delta y$
confirmed third order accuracy and stability of the resulting Active Flux
methods for time steps satisfying
a CFL condition of the form
$$
\mbox{CFL} := \left(\max\left\{|u'|,|v'|\right\} + c' \right) \Delta t / h \le 0.279.
$$

In order to use this approach for the approximation of the nonlinear
Euler equations, we use a different linearisation for the update of
each point value degree of freedom. One natural choice would be to
linearise around the point values at the previous time $t_n$, i.e., 
linearise around $U_{i-\frac{1}{2},j-\frac{1}{2}}^n$ in order to
compute $U_{i-\frac{1}{2},j-\frac{1}{2}}^{n+\frac{1}{2}}$ and
$U_{i-\frac{1}{2},j-\frac{1}{2}}^{n+1}$ and analogously for all the
other point value degrees of freedom. Another obvious possibility
would be to linearise around $U_{i-\frac{1}{2},j-\frac{1}{2}}^n$ in
order to compute  $U_{i-\frac{1}{2},j-\frac{1}{2}}^{n+\frac{1}{2}}$
and subsequently linearise around
$U_{i-\frac{1}{2},j-\frac{1}{2}}^{n+\frac{1}{2}}$ to compute
$U_{i-\frac{1}{2},j-\frac{1}{2}}^{n+1}$. In any case, the solution
at the intermediate and new time level
will be computed via (\ref{eqn:EG2}) using the reconstruction at time
$t_n$. This is a fundamental difference to a MOL approach and leads to a
more compact stencil. 
It is not immediately clear that this approach
leads to a third order accurate method for the Euler equation as we
are only solving the linearised Euler equations to evolve the point
value degrees of freedom. 
A detailed study of the linearisation error led to the following result.
\begin{theorem}\label{th:1}
Let $\mathbf{u}$ be a sufficiently smooth solution of the nonlinear
hyperbolic system
(\ref{eqn:quasilinearEuler}) and $\mathbf{v}$ a solution of the
linearised system (\ref{eqn:linEuler}) with
$\tilde{A}:=A(\mathbf{u}(\bar{x},\bar{y},t_{n}+\frac{\tau}{2})$
and
$\tilde{B}:=B(\mathbf{u}(\bar{x},\bar{y},t_{n}+\frac{\tau}{2})$. Assuming that the initial values at time $t_n$ are
equal, i.e.\ $\mathbf{u}(x,y,t_n)=\mathbf{v}(x,y,t_n)$, the
difference of the 
solutions at $(\bar{x},\bar{y},t_{n}+\tau)$ 
is described by  
\begin{equation}\label{eqn:errorLinearisation}
  \begin{split}
 \mathbf{v}(\bar{x},\bar{y},t_n+\tau) & + \frac{1}{2} \tau^2\left( A(\mathbf{u}) \frac{\partial A(\mathbf{u})}{\partial
    \mathbf{u}} \cdot (\partial_x \mathbf{u},\partial_x \mathbf{u}) +
  A(\mathbf{u}) \frac{\partial B(\mathbf{u})}{\partial
    \mathbf{u}} \cdot (\partial_x \mathbf{u},\partial_y \mathbf{u}) \right.\\
&\qquad \quad + \left. \left. B(\mathbf{u}) \frac{\partial A(\mathbf{u})}{\partial
    \mathbf{u}} \cdot (\partial_y \mathbf{u},\partial_x \mathbf{u}) +
  B(\mathbf{u}) \frac{\partial B(\mathbf{u})}{\partial
    \mathbf{u}} \cdot (\partial_y \mathbf{u},\partial_y \mathbf{u}) \right) \right|_{(\bar{x},\bar{y},t_n)}
\\
& = \mathbf{u}(\bar{x},\bar{y},t_n+\tau) + {\cal O}(\tau^3).
  \end{split}
\end{equation}
\end{theorem}
\begin{proof}
Follows from Taylor series expansions of $\mathbf{u}$ and
$\mathbf{v}$.
Consider
\begin{equation}
  \label{eqn:Th1-1}  
\begin{split}
  \mathbf{u}(\bar{x},\bar{y},t_n+\tau) -
  \mathbf{v}(\bar{x},\bar{y},t_n + \tau) & =
  \mathbf{u}(\bar{x},\bar{y},t_n)-\mathbf{v}(\bar{x},\bar{y},t_n)\\
  & + \tau \, \partial_t \left(
    \mathbf{u}(\bar{x},\bar{y},t)-\mathbf{v}(\bar{x},\bar{y},t)
    \right) \Big|_{t=t_n}\\
  & + \frac{1}{2} \tau^2 \,  \partial_{tt} \left(
    \mathbf{u}(\bar{x},\bar{y},t)-\mathbf{v}(\bar{x},\bar{y},t)
  \right) \Big|_{t=t_n} + {\cal O}(\tau^3).
\end{split}
\end{equation}
The difference appearing in the first line of the right hand
side of (\ref{eqn:Th1-1}) vanishes by assumption. To shorten our notation,
we will often omit the argument
$(\bar{x},\bar{y},t_n)$.

We now consider the differences in the second line of
(\ref{eqn:Th1-1}). First observe
\begin{equation}\label{eqn:observe}
  \begin{split}
\tilde{A} & = A \left(\mathbf{u}(\bar{x},\bar{y},t_n+\frac{\tau}{2}) + {\cal
    O}(\tau^2)\right) \\
& = A\left( \mathbf{u} + \frac{\tau}{2} \partial_t \mathbf{u} + {\cal
    O}(\tau^2) \right) \\
& = A(\mathbf{u}) + \frac{\tau}{2} \frac{A(\mathbf{u})}{\partial {\mathbf{u}}} 
\left( - A(\mathbf{u}) \partial_x \mathbf{u} - B(\mathbf{u})
  \partial_y \mathbf{u} \right) + {\cal O}(\tau^2),\\
\tilde{B} & = B(\mathbf{u}) + \frac{\tau}{2}
\frac{B(\mathbf{u})}{\partial {\mathbf{u}}} 
 \left( - A(\mathbf{u}) \partial_x \mathbf{u} -
  B(\mathbf{u}) \partial_y \mathbf{u} \right) + {\cal O}(\tau^2).
  \end{split}
\end{equation}
Thus,  with (\ref{eqn:quasilinearEuler}) and (\ref{eqn:linEuler}) we obtain
\begin{equation}\label{eqn:utmvt}
  \begin{split}
\partial_t \left( \mathbf{u}(\bar{x},\bar{y},t_n) -
  \mathbf{v}(\bar{x},\bar{y},t_n)\right)
& = -A(\mathbf{u}) \partial_x \mathbf{u} -
B(\mathbf{u})\partial_y \mathbf{u} + \tilde{A} \partial_x
\mathbf{v} + \tilde{B} \partial_y \mathbf{v} \\
&=  -\frac{\tau}{2} \frac{\partial A(\mathbf{u})}{\partial
  {\mathbf{u}}} \cdot \left(
  A(\mathbf{u}) \partial_x \mathbf{u} + B(\mathbf{u}) \partial_y
  \mathbf{u},   \partial_x \mathbf{u}\right) \\
& \quad - \frac{\tau}{2} \frac{\partial B (\mathbf{u}) }{\partial
  {\mathbf{u}}} \cdot \left(
  A(\mathbf{u}) \partial_x \mathbf{u} + B(\mathbf{u}) \partial_y
  \mathbf{u},  \partial_y \mathbf{u} \right) + {\cal O}(\tau^2).
  \end{split}
\end{equation}
We again made use of the fact that at time $t_n$ the initial values in
$\mathbf{u}$ and $\mathbf{v}$ coincide and therefore  $\partial_x
\mathbf{u} = \partial_x \mathbf{v}$ and $\partial_y \mathbf{u} =
\partial_y \mathbf{v}$.

Now we consider the difference in the third line of
(\ref{eqn:Th1-1}), starting with the $\mathbf{u}$ term. We have
\begin{equation}\label{eqn:utt-1}
  \begin{split}
& \partial_{tt} \mathbf{u}(\bar{x},\bar{y},t) \Big|_{t=t_n}  =
-\partial_t \left( A(\mathbf{u}) \partial_x \mathbf{u}
\right)\Big|_{(\bar{x},\bar{y},t_n)} - \partial_t \left( B(\mathbf{u})
  \partial_y \mathbf{u} \right) \Big|_{(\bar{x},\bar{y},t_n)}\\
& \quad = -\left( \frac{\partial A(\mathbf{u})}{\partial {\mathbf{u}}} \cdot \left(\partial_t \mathbf{u},
  \partial_x \mathbf{u}\right) + A(\mathbf{u}) \partial_{tx} \mathbf{u} +
  \frac{\partial B(\mathbf{u})}{\partial {\mathbf{u}}} \cdot \left(  \partial_t \mathbf{u}, \partial_y
  \mathbf{u}\right) + B(\mathbf{u}) \partial_{ty} \mathbf{u} \right)
  \end{split}
\end{equation}
with
\begin{equation*}
  \begin{split}
& A(\mathbf{u}) \partial_{tx} \mathbf{u} = A(\mathbf{u}) \partial_x
\left( - A(\mathbf{u}) \partial_x \mathbf{u} - B(\mathbf{u})
  \partial_y \mathbf{u} \right) \\
& \quad = -A(\mathbf{u}) \left( \frac{\partial A(\mathbf{u})}{\partial {\mathbf{u}}} 
  \cdot \left(\partial_x \mathbf{u}, \partial_x \mathbf{u}\right) + A (\mathbf{u})
  \partial_{xx} \mathbf{u} + \frac{\partial B(\mathbf{u})}{\partial {\mathbf{u}}}
  \cdot \left(\partial_x \mathbf{u}, \partial_y \mathbf{u}\right) + B(\mathbf{u})
  \partial_{yx} \mathbf{u} \right)
  \end{split}
\end{equation*}
and analogously
\begin{equation*}
  \begin{split}
  &B(\mathbf{u}) \partial_{ty} \mathbf{u} = B(\mathbf{u}) \partial_y
  \left( -A(\mathbf{u}) \partial_x \mathbf{u} - B(\mathbf{u})
    \partial_y \mathbf{u} \right)\\
  &\quad  - B(\mathbf{u}) \left(
  \frac{\partial A(\mathbf{u})}{\partial \mathbf{u}} \cdot \left( \partial_y
    \mathbf{u},\partial_x \mathbf{u}\right) + A(\mathbf{u})
  \partial_{xy} \mathbf{u} + \frac{\partial B(\mathbf{u})}{\partial
    \mathbf{u}}\cdot \left( \partial_y \mathbf{u}, \partial_y
    \mathbf{u}\right) + B(\mathbf{u}) \partial_{yy} \mathbf{u} \right).
\end{split}
\end{equation*}
Thus, we get
\begin{equation}\label{eqn:utt-2}
  \begin{split}
& \partial_{tt}\mathbf{u} (\bar{x},\bar{y},t) \Big|_{t=t_n}  =
\frac{\partial A(\mathbf{u})}{\partial \mathbf{u}} \cdot \left(
A(\mathbf{u}) \partial_x \mathbf{u}  + B(\mathbf{u}) \partial_y
\mathbf{u}, \partial_x \mathbf{u} \right) \\
& \qquad \qquad \qquad \quad +
\frac{\partial B(\mathbf{u})}{\partial \mathbf{u}} \cdot \left(
  A(\mathbf{u}) \partial_x \mathbf{u} + B(\mathbf{u}) \partial_y
  \mathbf{u},\partial_y \mathbf{u}\right) \\
& \quad + A(\mathbf{u}) \left( \frac{\partial A(\mathbf{u})}{\partial
    \mathbf{u}} \cdot \left(\partial_x \mathbf{u},\partial_x \mathbf{u}
  \right) + A(\mathbf{u}) \partial_{xx} \mathbf{u} + \frac{\partial B(\mathbf{u})}{\partial
    \mathbf{u}} \cdot \left(\partial_x \mathbf{u},\partial_y \mathbf{u}
  \right) + B(\mathbf{u}) \partial_{yx} \mathbf{u} \right) \\
& \quad + B(\mathbf{u}) \left( \frac{\partial A(\mathbf{u})}{\partial
    \mathbf{u}} \cdot \left(\partial_y \mathbf{u},\partial_x \mathbf{u}
  \right) + A(\mathbf{u}) \partial_{xy} \mathbf{u} + \frac{\partial B(\mathbf{u})}{\partial
    \mathbf{u}} \cdot \left(\partial_y \mathbf{u},\partial_y \mathbf{u}
  \right) + B(\mathbf{u}) \partial_{yy} \mathbf{u} \right).
  \end{split}
\end{equation}
Note that the terms appearing in the first two lines of (\ref{eqn:utt-2})
also appear with different sign in
(\ref{eqn:utmvt}) and will therefore cancel when inserted into (\ref{eqn:Th1-1}).
This is the reason why we propose to linearise around the
solution at time $\tau/2$.

Finally, we consider the contribution of the solution to the
linearised problem:
\begin{equation*}
  \begin{split}
\partial_{tt} \mathbf{v}(\bar{x},\bar{y},t)\Big|_{t=t_n} & =
\left( \partial_t \left( -\tilde{A} \partial_x \mathbf{v} - \tilde{B}
  \partial_y \mathbf{v} \right) \right) \Big|_{(\bar{x},\bar{y},t_n)} \\
& = \tilde{A} \tilde{A} \partial_{xx} \mathbf{v} + \tilde{A} \tilde{B}
\partial_{yx} \mathbf{v} + \tilde{B} \tilde{A} \partial_{xy}
\mathbf{v} + \tilde{B} \tilde{B} \partial_{yy} \mathbf{v}.
  \end{split}
\end{equation*}
Since this term appears in (\ref{eqn:Th1-1}) in a product with
$\tau^2$, we only need to take first order approximations
into account. Furthermore, we use once more the fact that,  at time $t_n$,
the functions $\mathbf{u}$ and $\mathbf{v}$ are identical, and obtain
\begin{equation}\label{eqn:vtt}
  \begin{split}
\partial_{tt} \mathbf{v}(\bar{x},\bar{y},t) \Big|_{t=t_n} & =
\left( A(\mathbf{u}) A(\mathbf{u}) \partial_{xx} \mathbf{u} +
  A(\mathbf{u})B(\mathbf{u}) \partial_{yx} \mathbf{u} \right. \\
& \quad  \left. + B(\mathbf{u})
A(\mathbf{u}) \partial_{xy} \mathbf{u} + B(\mathbf{u}) B(\mathbf{u})
\partial_{yy} \mathbf{u}\right)\Big|_{(\bar{x},\bar{y},t_n)} + {\cal O}(\tau).
\end{split}
\end{equation}
Inserting (\ref{eqn:utmvt}), (\ref{eqn:utt-2}) and (\ref{eqn:vtt}) in
(\ref{eqn:Th1-1}) we obtain (\ref{eqn:errorLinearisation}).
\end{proof}

While the local linearisation introduces a linearisation error of order
$\tau^2$, we can easily correct this error 
and obtain a 
third order accurate approximation
of the nonlinear system (\ref{eqn:quasilinearEuler}) at the point
$(\bar{x},\bar{y},t_n+\tau)$ by adding a correction term of order
${\cal O}(\tau^2)$.
\begin{corollary}\label{cor:1}
A third order accurate  approximation of 
$\mathbf{u}(\bar{x},\bar{y},t_n+\tau)$ 
can be obtained by  linearising around an at least 
second order accurate approximation of
$\mathbf{u}(\bar{x},\bar{y},t_n+\tau/2)$, approximating the resulting linear
system using the EG2 operator and  adding a correction term
$C(\bar{x},\bar{y},t_n,\tau)$ of the form
$$
\frac{1}{2} \tau^2 \left( \begin{array}{c}
                           \rho(f_1+f_2)+u(\rho_x (2u_x+v_y)+\rho_y
                            v_x) + v(\rho_xu_y+\rho_y(u_x+2v_y)) -
                            h_1\\
                            u f_1 + v u_y (u_x+v_y) + g_1/\rho -
                            p_xh_2\\
                             v f_2 + u v_x(u_x+v_y) + g_2/\rho -
                            p_yh_2\\
                             u g_1 + v g_2 + \gamma p (f_1 + f_2 - h_1/\rho)
         \end{array}\right),
       $$
       with
       \begin{equation*}
         \begin{split}
           f_1 & = u_x^2+u_y v_x, \quad f_2  = u_y v_x+v_y^2,\\
           g_1 & = p_x (\gamma (u_x+v_y)+u_x)+p_y v_x, \quad g_2  = p_y (\gamma (u_x+v_y)+v_y)+p_x u_y,\\ 
           h_1 & = (\rho_x p_x+\rho_y p_y)/\rho, \quad h_2  = (\rho_x u+\rho_y v)/\rho^2,
         \end{split}
       \end{equation*}
       where all terms are evaluated at $(\bar{x},\bar{y},t_n)$. 
     \end{corollary}
     \begin{proof}
Follows from explicit computation of the second order term in (\ref{eqn:Th1-1}) for the
Euler equations, the accuracy of the EG2 evolution operator for the
linearised Euler equations and the accuracy of Simpson's rule for the
fluxes used to evolve the cell average values.
       \end{proof}
We summarise the Active Flux method for the Euler equations in
\cref{alg:AFEu}.
\begin{algorithm}
  \caption{Third order accurate Active Flux method for the Euler
    equations}
  \label{alg:AFEu}
  \begin{algorithmic}
    
\STATE{1.) For all grid cells: compute point values in primitive variables at the intermediate
  time using
  \begin{equation*}
    \begin{split}
U_{i-\frac{1}{2},j-\frac{1}{2}}^{n+\frac{1}{2}} & =
L_{EG2}(L_{EG2}(U_{i-\frac{1}{2},j-\frac{1}{2}}^n,\Delta
t/4,\mbox{cpq}),\Delta t/2, \mbox{cpq}) \\
& \quad +
C(x_{i-\frac{1}{2}},y_{j-\frac{1}{2}},t_n,\Delta t/2)
\end{split}
\end{equation*}
and analogously for all other point value degrees of freedom.}

\STATE{2.) For all grid cells: compute point values in primitive variables at the new time
  level using
  \begin{equation*}
U_{i-\frac{1}{2},j-\frac{1}{2}}^{n+1} = L_{EG2}
(U_{i-\frac{1}{2},j-\frac{1}{2}}^{n+\frac{1}{2}},\Delta t,\mbox{cpq}) +  C(x_{i-\frac{1}{2}},y_{j-\frac{1}{2}},t_n,\Delta t)
\end{equation*}
and analogously for all other point values.}

\STATE{3.) Compute numerical fluxes using Simpson's rule and evolve all
  cell average values using (\ref{eqn:2dfvm}).}  
\end{algorithmic}
\end{algorithm}

Note that the term $L_{EG2}(U_{i-\frac{1}{2},j-\frac{1}{2}}^n,\Delta
t/4,\mbox{cpq})$, used in the  first step of the algorithm,
represents a second order accurate approximation of
$\mathbf{u}(x_{i-\frac{1}{2}},y_{j-\frac{1}{2}},t_{n}+\Delta
t/4)$. Here this point value is
computed using the method of bicharacteristics applied
to the linearised Euler equations by linearising around the point
value at the old time level.
From the proof of \cref{th:1}, compare with (\ref{eqn:observe}),
it follows that a second order
accurate approximation of this point value  is
sufficient for the  computation of a third order accurate point value at
the intermediate time $t_{n+\frac{1}{2}}$.

Note furthermore, that a first order accurate approximation of $C$
would be appropriate. Here we instead used a second order accurate
approximation, where all derivatives were computed using centered
differences employing neighbouring point values or midpoints of the grid cell.
In the special case where the velocity field and the pressure are
constant in space, the exact solution of the Euler equations
consists in a simple transport of the initial density profile. In this case
the correction term $C$ is equal to zero and therefore not needed for
third order accuracy.

\begin{remark}
Our approach for Burgers' equation, described in \cite{article:CHK2021},
differs from the method described in \cref{alg:AFEu}.  
For Burgers' equation the speed of the characteristic used to compute
the solution $q$ at $(\bar{x},\bar{y},t_{n+1})$ is equal to the
solution $q(\bar{x},\bar{y},t_{n+1})$ we wish to compute.
Thus, an iterative approach was used to
get third order accurate results for smooth solutions.
\end{remark}

We finally confirm the order of convergence by numerical
simulations. Note that the first test problem was also considered in
\cite{article:HKS2019} and \cite{article:Barsukow2021} to test different
one-dimensional methods. Here we use our two-dimensional Active Flux
method with $8$ grid cells in the $y$-direction, different numbers
of grid cells in the $x$-direction and double-periodic boundary
conditions.
\begin{example}\label{ex:1}
We consider the Euler equations with initial values of the form
$$
\rho(x,y,0) = p(x,y,0) = 1 + \frac{1}{2} \exp(-80(x-\frac{1}{2})^2),
\quad u(x,y,0) = v(x,y,0) = 0
$$
and compute numerical solutions for $x \in [0,1]$
at time $t=0.25$. 
\end{example}
\cref{tab:ex1} shows results of our
numerical study which confirms the third order convergence
rate of \cref{alg:AFEu}, see column ``with correction''. We also show a
convergence study for the same method except that we do not add the
correction term. The results are shown in the ``without correction'' column.
On coarser grids both methods provide comparable results but on
the finest grid we observe the expected second order convergence
if we ignore the correction term.  
In the last column we show a convergence study for a simplified
method, where the first step of the algorithm is replaced by
\begin{equation} \label{eqn:1stStep}
U_{i-\frac{1}{2},j-\frac{1}{2}}^{n+\frac{1}{2}} =
L_{EG2}(U_{i-\frac{1}{2},j-\frac{1}{2}}^n,\Delta t/2,\mbox{cpq}) +
C(x_{i-\frac{1}{2}},y_{j-\frac{1}{2}},t_n,\Delta t/2)
\end{equation}
and analogously for all the other point value degrees of freedom. This
means that we compute the intermediate point values by linearising
around the old point value. These point values are used for the
computation of fluxes only and although they are now formally only second
order accurate the overall accuracy of the method is not
degraded for this test problem. For our numerical convergence study we
used the approach from \cite[Section A.6.3]{book:LeV2007}, which
allows to compute a convergence rate from numerical solutions on three
different grids. The values of the $L_1$-error correspond to
differences with approximations that use
twice as many cells in both directions.

\begin{table}[htbp]
\footnotesize
\caption{Numerical convergence study for the one-dimensional solution
  structure of \cref{ex:1}.}\label{tab:ex1}
\begin{center}
  \begin{tabular}{|r|cc|cc|cc|} \hline
 cells & \multicolumn{2}{c|}{with correction} &
      \multicolumn{2}{c|}{without correction} &
                                                \multicolumn{2}{c|}{simplified method}  \\
 in $x$ & $L_1$-error in $\rho$ & EOC & $L_1$-error in $\rho$ &
                                                                      EOC & $L_1$-error in $\rho$ & EOC\\ \hline
    $32$ & $3.112504\cdot10^{-4}$ & & $3.124039 \cdot 10^{-4}$ & &
                                                                 $3.098132
                                                                 \cdot
    10^{-4}$ & \\
    $64$ & $4.383598 \cdot 10^{-5}$ & $2.82$ & $4.409654 \cdot
                                               10^{-5}$ & $2.82$ &
                                                                   $4.321119
                                                                   \cdot
    10^{-5}$ & $2.84$\\ 
    $128$ & $5.676151 \cdot 10^{-6}$ & $2.94$ & $5.876851 \cdot
                                                10^{-6}$ & $2.89$ &
                                                                    $5.516531
                                                                    \cdot
    10^{-6}$ & $2.96$\\ 
    $256$ & $7.170790 \cdot 10^{-7}$ & $2.98$ & $7.892507 \cdot
                                                10^{-7}$ & $2.89$ &
                                                                    $6.796768
                                                                    \cdot
    10^{-7}$ & $3.02$\\
    $512$ & $9.022719 \cdot 10^{-8}$ & $2.99$ & $1.220382 \cdot
                                                10^{-7}$ & $2.69$ &
                                                                    $8.197030
                                                                    \cdot
    10^{-8}$ & $3.05$ \\
    $1024$ & $1.129548 \cdot 10^{-8}$ & $3.00$ & $2.892100 \cdot 10^{-8}$
                                & $2.07$ & $1.088839 \cdot 10^{-8}$ & $2.91$ \\ \hline
  \end{tabular}
\end{center}
\end{table}
In \cref{fig:ex1} we show numerical results for the three different
methods on a coarse grid using only 32 grid cells in the $x$
direction. All three versions of the Active Flux method produce very accurate results.  
\begin{figure}[tbh]
  \includegraphics[width=0.3\linewidth]{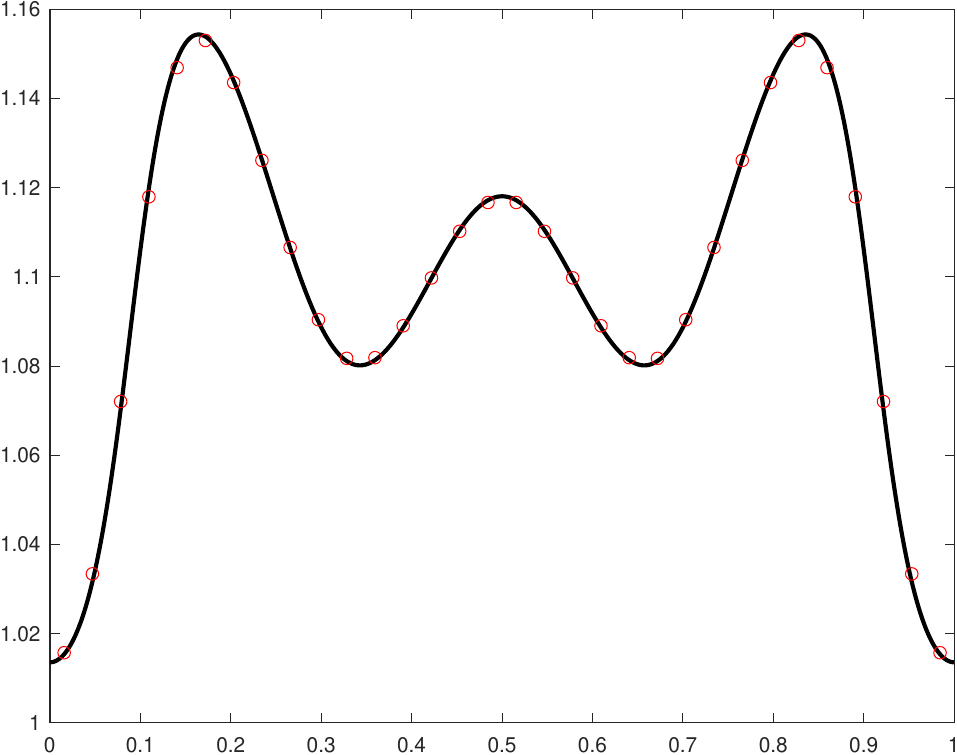}\hfill
  \includegraphics[width=0.3\linewidth]{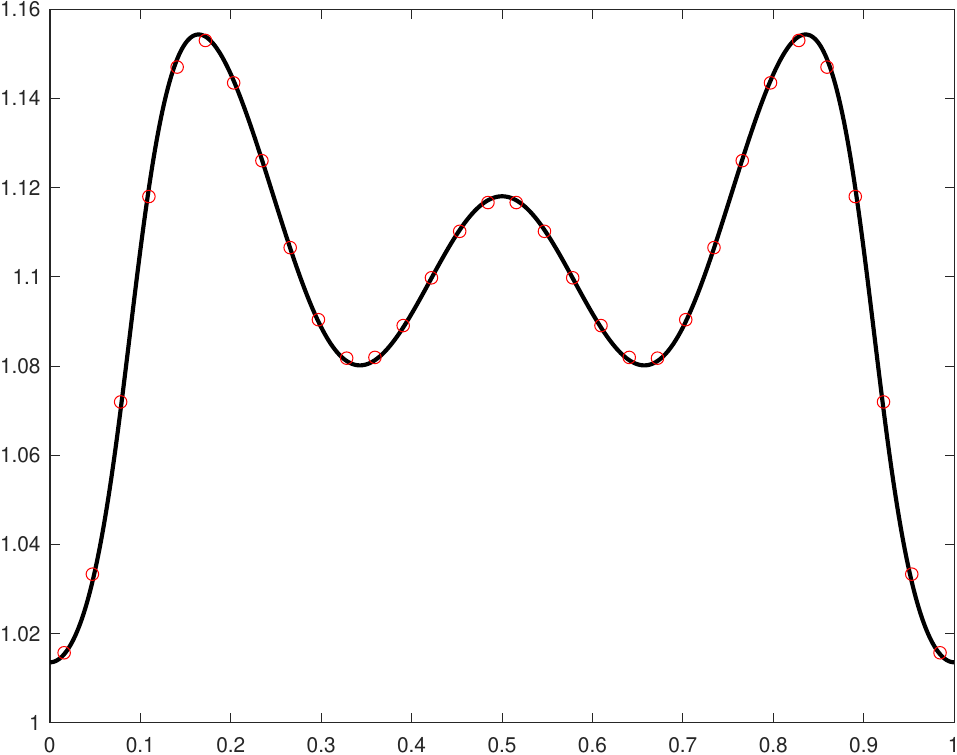}\hfill
  \includegraphics[width=0.3\linewidth]{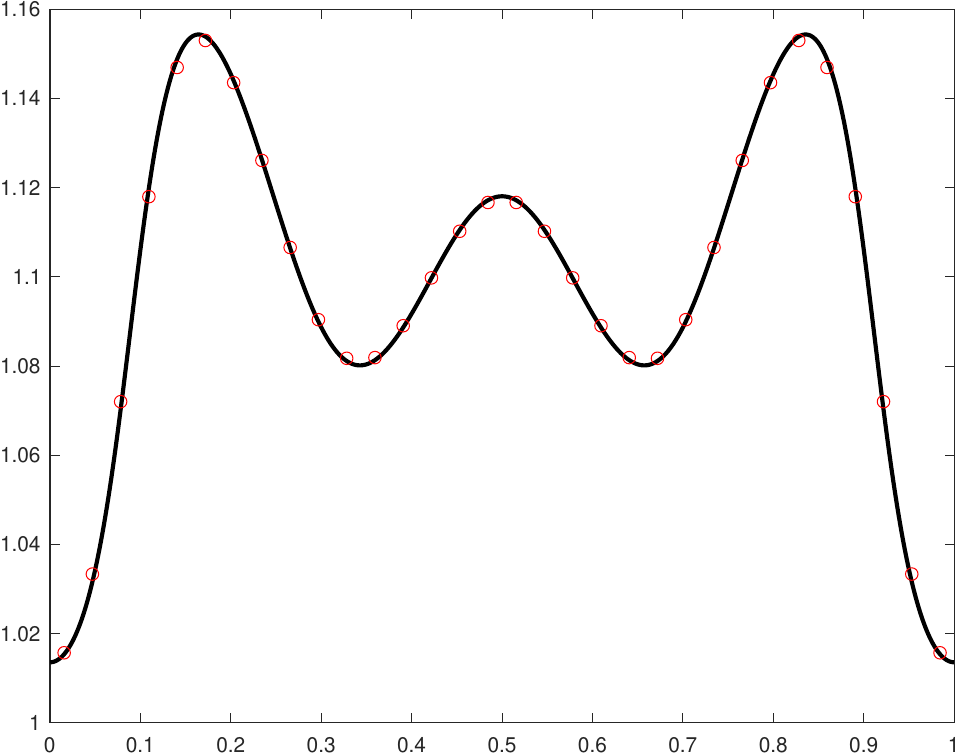}
\caption{\label{fig:ex1} Numerical results of density for \cref{ex:1} using a grid with 32 grid
  cells in $x$-direction. The solid line represents a highly resolved reference
  solution using $2048$ grid cells. Third order accurate
  Active Flux method described in \cref{alg:AFEu} (left), method
  without correction (middle) and simplified method (right). }  
\end{figure}  

In our second accuracy test we consider a two-dimensional smooth
traveling vortex introduced in \cite{article:KKM2008}.
\begin{example}\label{ex:2}
  We consider the Euler equations with initial values of the form
\begin{equation*}
\begin{split}
      \rho(x,y,0) & = \left\{ \begin{array}{lcl}
                                \rho_c + \frac{1}{2}(1-r^2)^6 & : &
                                                                    r<1,\\
                                \rho_c \phantom{ - 1024 \sin(\theta)
                                (1-r^6) r^6 } \, \, \, & : &
                                             \mbox{otherwise},\end{array}\right.\\
   u(x,y,0) & = \left\{ \begin{array}{lcl}
                          u_c - 1024 \sin(\theta) (1-r^6) r^6 & : & r<1, \\
                          u_c & : & \mbox{otherwise},\end{array}\right. \\
   v(x,y,0) & = \left\{ \begin{array}{lcl}
                          v_c + 1024 \cos(\theta) (1-r^6) r^6 & : &
                                                                    r<1,\\
                          v_c & : & \mbox{otherwise},\end{array}\right.\\
    p(x,y,0) & = \left\{ \begin{array}{lcl} p_c + (p(r)-p(1)) & : &
                                                                  r<1,\\
                         p_c \phantom{ - 1024 \sin(\theta) (1-r^6) r^6
                           } \, \, \, & : & \mbox{otherwise},\end{array} \right.
\end{split}
\end{equation*}
with $(\rho_c,u_c,v_c,p_c) = (0.5,1,1,0.1)$, $r =
\sqrt{(x-x_0)^2+(y-y_0)^2}/R$, \newline
$(x_0,y_0)=(0.5,0.5)$, $\theta = \arctan\left( \frac{y-y_o}{x-x_0}\right)$, $R=0.4$ and
\begin{equation*}
  \begin{split}
p(r) & = 1024^2 \left( \frac{1}{72}r^{36}-\frac{6}{35}r^{35} +
  \frac{15}{17}r^{34} -
  \frac{74}{33}r^{33}+\frac{57}{32}r^{32}+\frac{174}{31}r^{31}
  -\frac{269}{15}r^{30}\right.\\
  & +\frac{450}{29}r^{29}+\frac{153}{8}r^{28} -
    \frac{1564}{27}r^{27} 
    +\frac{510}{13}r^{26}+\frac{204}{5}r^{25}-\frac{1473}{16}r^{24}+\frac{1014}{23}r^{23}\\
  &
  +\frac{1053}{22}r^{22}-\frac{558}{7}r^{21}+\frac{783}{20}r^{20}+\frac{54}{19}r^{19}
  -\frac{38}{9}r^{18}-\frac{222}{17}r^{17}+\frac{609}{32}r^{16}  \\
  & \left. -\frac{184}{15}r^{15} + \frac{9}{2} r^{14} - \frac{12}{13}r^{13} + \frac{1}{12}r^{12}\right).
  \end{split}
\end{equation*}
We compute the solution on the domain $[0,1]\times [0,1]$ using
double-periodic boundary conditions. At times $t=n$, ($n \in \mathbb{N}$), the
exact solution agrees with the initial values.
\end{example}  
\cref{tab:ex2} shows results of a convergence study for approximations
at time $t=1$, which again
confirm third order accuracy of the method described in \cref{alg:AFEu}
and comparable results for the simplified method.
\begin{table}[htbp]
\footnotesize
\caption{Numerical convergence study for the two-dimensional smooth
  vortex problem described in \cref{ex:2}.}\label{tab:ex2}
\begin{center}
  \begin{tabular}{|r|cc|cc|cc|} \hline
 cells & \multicolumn{2}{c|}{with correction} &
      \multicolumn{2}{c|}{without correction} &
                                                \multicolumn{2}{c|}{simplified method}  \\
 in $x$ & $L_1$-error in $\rho$ & EOC & $L_1$-error in $\rho$ &
                                                                      EOC & $L_1$-error in $\rho$ & EOC\\ \hline
    $32$ & $5.825428\cdot10^{-4}$ & & $ 5.812933\cdot 10^{-4}$ & &
                                                                 $
                                                                 5.829871\cdot
    10^{-4}$ & \\
    $64$ & $ 9.548670\cdot 10^{-5}$ & $2.60$ & $ 9.685133\cdot
                                               10^{-5}$ & $2.58$ &
                                                                   $9.567499
                                                                   \cdot
    10^{-5}$ & $2.60$\\ 
    $128$ & $ 1.296321\cdot 10^{-5}$ & $2.88$ & $ 1.352539\cdot
                                                10^{-5}$ & $2.84$ &
                                                                    $1.301396
                                                                    \cdot
    10^{-5}$ & $2.87$\\ 
    $256$ & $ 1.646819\cdot 10^{-6}$ & $2.97$ & $ 1.870932\cdot
                                                10^{-6}$ & $2.85$ &
                                                                    $1.661371
                                                                    \cdot
    10^{-6}$ & $2.96$\\
    $512$ & $ 2.065112\cdot 10^{-7}$ & $2.99$ & $ 2.953343\cdot
                                                10^{-7}$ & $2.66$ &
                                                                    $2.111573
                                                                    \cdot
    10^{-7}$ & $2.97$ \\
    $1024$ & $ 2.606372\cdot 10^{-8}$ & $2.99$ & $ 6.084135\cdot 10^{-8}$
                                & $2.28$ & $ 2.777238\cdot 10^{-8}$ & $2.92$ \\ \hline
  \end{tabular}
\end{center}
\end{table}
Without the correction term introduced in \cref{cor:1} the
experimentally observed convergence rate decreases on finer grids and
approaches two, as expected from the theoretical results.
In \cref{fig:ex2} we show numerical results on a coarse grid 
using the third order accurate method with correction term. The
solution structure is well preserved for several rotations.
\begin{figure}[tbh]
\includegraphics[width=0.32\linewidth]{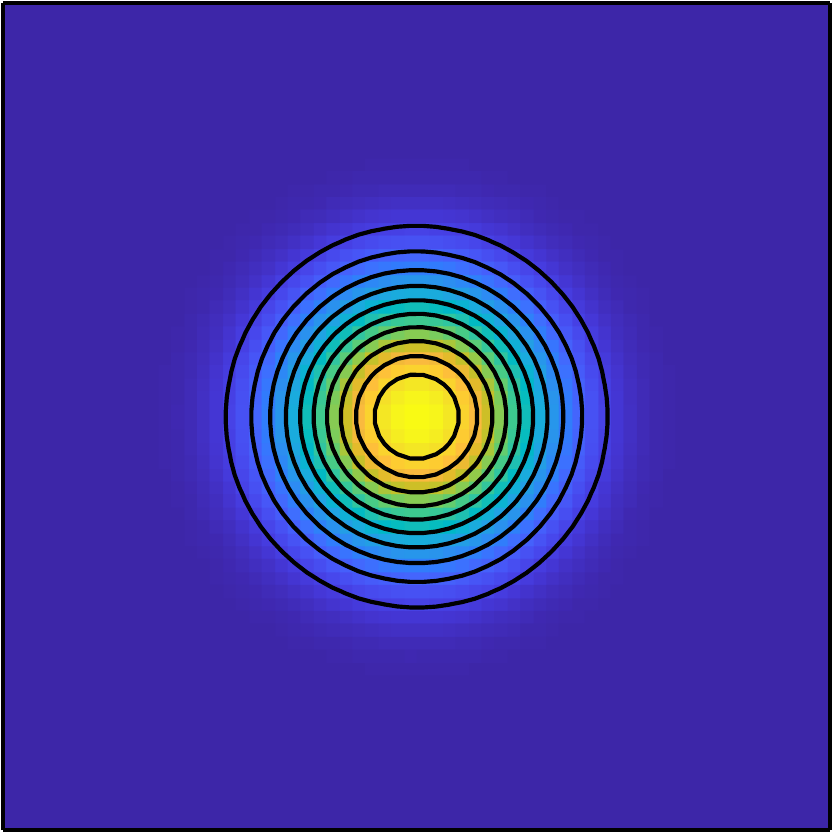}\hfill
\includegraphics[width=0.32\linewidth]{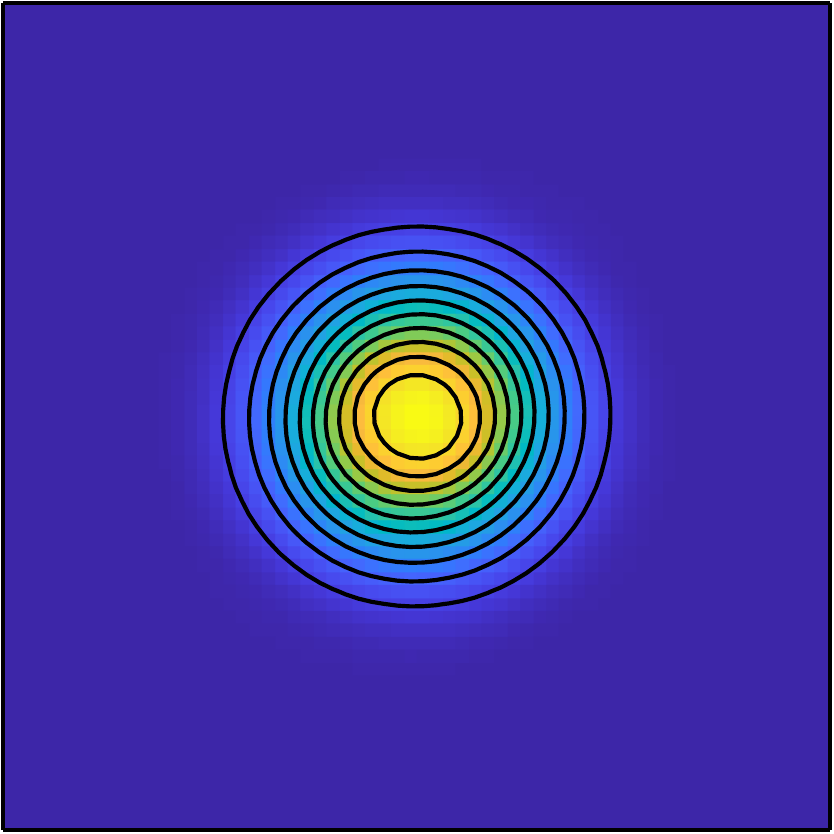}\hfill
\includegraphics[width=0.32\linewidth]{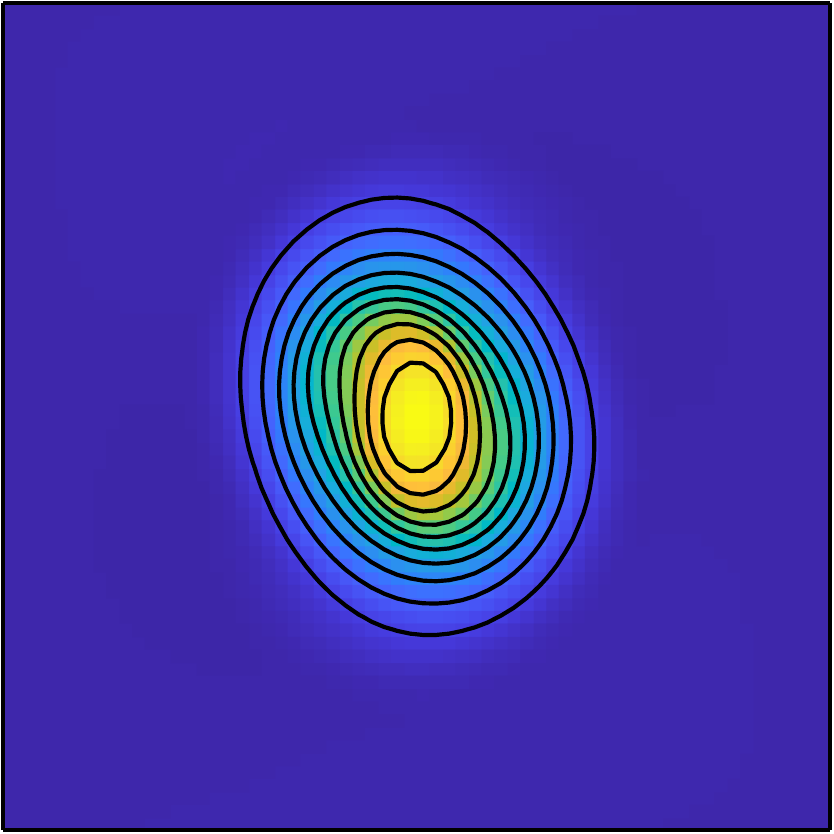}
\caption{\label{fig:ex2} Numerical results of density for \cref{ex:2}
  using a grid with $64^2$ grid cells.
  From left to right we show the solution at times $t=1$, $t=5$ and $t=10$. }  
\end{figure} 

In the reminder of this paper we
study the approximation of more challenging problems and in particular
of problems with discontinuities, where we can not
expect high order convergence rates. Therefore, we use versions
of the simplified method in the rest of the paper. We will see that
the compact stencil of our fully discrete Active Flux method provides
accurate results also for these more challenging problems. 

\section{Approximation of discontinuous solutions}
\label{sec:3}
Now we discuss modifications of \cref{alg:AFEu}, which are needed
for the simulation of discontinuous solutions.

\subsection{Local linearisation near transonic shock waves}
It is known, see \cite{article:HKS2019}, that Active Flux
methods for Burgers' equation can lead to unstable approximations of
shock waves  when the characteristic speed changes sign across the
shock while the local linearisation is computed from
the point value degrees of freedom.
In \cite{article:Barsukow2021,article:CHK2021}
it was shown that a stable approximation can
be obtained if the linearisation is instead
computed from averages of neighbouring grid cell values.
For the Euler equations an analogous problem arises for the
approximation of transonic shock waves and is illustrated in \cref{fig:transonicShock}.
\begin{figure}
(a)\includegraphics[scale=0.35]{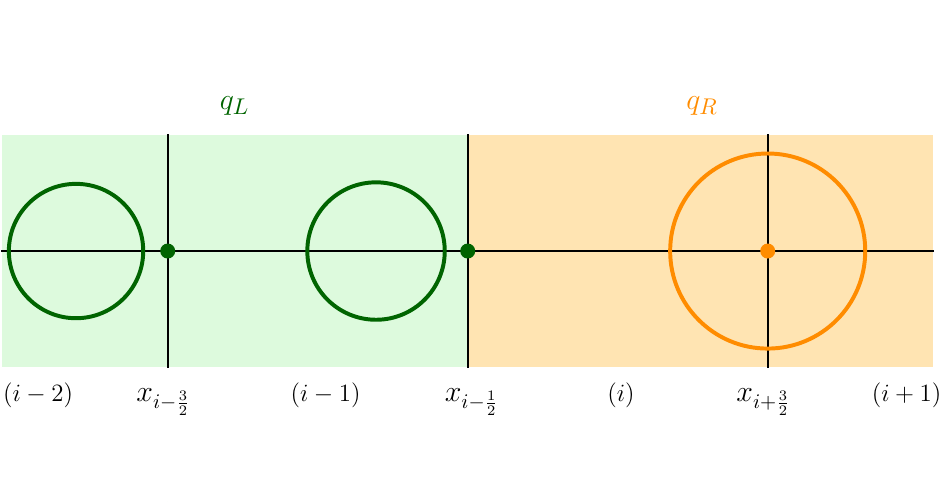}\hfill
(b)\includegraphics[scale=0.35]{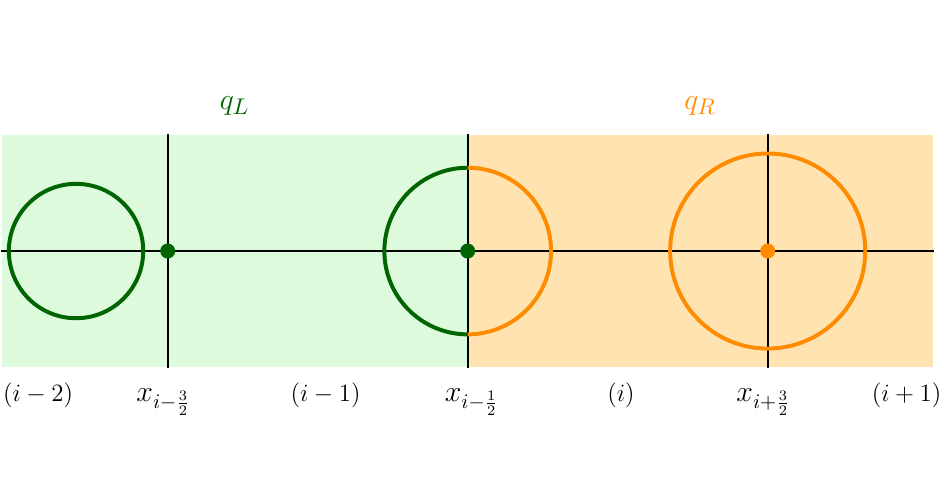}
\caption{\label{fig:transonicShock}Illustration of the update of the
  point values for a transonic left moving shock wave. In (a) the
  point values at the corners are used for linearisation, in (b) the
  neighbouring cell averages.}
\end{figure}
In (a) we illustrate the case of a local linearisation around
the point value degrees of freedom.  Initially, we assume that the point value at the
	interface $x_{i-\frac{1}{2}}$ between cell $(i-1)$ and $(i)$ is equal to the left
	state. Our evolution operator will integrate over a circle which lies
completely inside the left region and thus the point value at
this interface remains constant for all times. As a result, the shock wave can't
move to the left hand side and an instability might grow starting from
the third grid cell.
In (b) we illustrate a linearisation around the average
of the two neighbouring cell averages. The point values are no
longer constant in time which allows the shock to move. 

Since our accuracy analysis reveals that a second order accurate
approximation of the state used in the local linearisation is
sufficient, averaging in space does not reduce the accuracy for
smooth flow problems. 
We therefore adapt (\ref{eqn:1stStep}) as follows:
For point value degrees of freedom located at vertices we use the four neighbouring grid cells to
compute the linearisation, i.e., we compute
\begin{equation}\label{eqn:Un+1/2}
  \begin{split}
\tilde{U}_{i-\frac{1}{2},j-\frac{1}{2}}^n & = \frac{1}{4}\left( \bar{U}_{i-1,j-1}^n+\bar{U}_{i,j-1}^n+\bar{U}_{i-1,j}^n+\bar{U}_{i,j}^n\right),\\    
U_{i-\frac{1}{2},j-\frac{1}{2}}^{n+\frac{1}{2}} & =
L_{EG2}(\tilde{U}_{i-\frac{1}{2},j-\frac{1}{2}}^n,\Delta t/2,\mbox{cpq}) +
C(x_{i-\frac{1}{2}},y_{j-\frac{1}{2}},t_n,\Delta t/2).
  \end{split}
\end{equation}
At edge midpoints we linearise around the two neighbouring averages
in primitive variables, i.e.\
\begin{equation}
  \begin{split}
\tilde{U}_{i-\frac{1}{2},j}^n & = \frac{1}{2}\left( \bar{U}_{i-1,j}^n+\bar{U}_{i,j}^n\right),\\    
U_{i-\frac{1}{2},j}^{n+\frac{1}{2}} & =
L_{EG2}(\tilde{U}_{i-\frac{1}{2},j}^n,\Delta t/2,\mbox{cpq}) +
C(x_{i-\frac{1}{2}},y_{j},t_n,\Delta t/2).
  \end{split}
\end{equation}
and analogously for all other point values. Convergence studies for
the smooth test problems described in 
\cref{ex:1} and \cref{ex:2} using this linearisation lead to almost
identical results as those shown in \cref{tab:ex1,tab:ex2} and are
therefore omitted.

In addition, we will alter the second step of \cref{alg:AFEu} near
transonic shock waves. Let $c_{ij}^n = \sqrt{\gamma
  p_{ij}^n/\rho_{ij}^n}$ denote the speed of sound in grid cell
$(i,j)$ computed from the cell average values of the primitive
variables at time $t_n$. 
The modification of the second step is described in
\cref{alg:AFEuTransonicShocks}. There we introduced a variable
${\cal I}^{tsw}_{i-\frac{1}{2},j-\frac{1}{2}}$, which is an indicator
  for a transonic shock wave in the neighbourhood of the point value
  with index $(i-\frac{1}{2},j-\frac{1}{2})$.
  We check for all corner point values whether they
    lie in the neighbourhood of a transonic shock and use the modified
  linearisation if needed. For midpoints of an edge we adapt the
  linearisation if at least one of the endpoints of the edge have been
detected to lie in the vicinity of a transonic shock.

\begin{algorithm}
	\caption{Modification of second step of \cref{alg:AFEu} for
		transonic shock waves}
	\label{alg:AFEuTransonicShocks}
	\begin{algorithmic}

\vspace*{0.2cm}
\STATE{$\mathcal{I}^{tsw}_{i-\frac{1}{2},j-\frac{1}{2}} = 0,\,
  \mathcal{I}^{tsw}_{i+\frac{1}{2},j-\frac{1}{2}} = 0,\, \mathcal{I}^{tsw}_{i-\frac{1}{2},j+\frac{1}{2}} = 0$ }
\FOR{$\left(k,l\right) \text{in} \left\{\left(0,0\right),\left(1,0\right),\left(0,1\right)\right\}$}
\IF{($(u_{i-1+k,j+l}^n\pm c_{i-1+k,j+l}^n)>0$ and $(u_{i+k,j+l}^n \pm
	c_{i+k,j+l}^n)<0$) or \\
	\quad ($(u_{i-1+k,j-1+l}^n\pm c_{i-1+k,j-1+l}^n)>0$ and $(u_{i+k,j-1+l}^n \pm
	c_{i+k,j-1+l}^n)<0$) or \\
	\quad ($(v_{i-1+k,j-1+l}^n \pm c_{i-1+k,j-1+l}^n)>0$ and $(v_{i-1+k,j+l}^n\pm
	c_{i-1+k,j+l}^n)<0$) or \\
	\quad ($(v_{i+k,j-1+l}^n \pm c_{i+k,j-1+l}^n)>0$ and $(v_{i+k,j+l}^n\pm
	c_{i+k,j+l}^n)<0$)\\ }
\STATE{$\mathcal{I}^{tsw}_{i+k-\frac{1}{2},j+l-\frac{1}{2}}=1$}
\ENDIF
\ENDFOR
\vspace*{0.2cm}
\STATE{$\tilde{U}_{i-\frac{1}{2},j}=U_{i-\frac{1}{2},j}^{n+\frac{1}{2}}$}
\IF{$\max \left\{\mathcal{I}^{tsw}_{i-\frac{1}{2},j-\frac{1}{2}}, \mathcal{I}^{tsw}_{i-\frac{1}{2},j+\frac{1}{2}}\right\}=1$}
\STATE{
	$\tilde{U}_{i-\frac{1}{2},j}  = \frac{1}{2}
	\left(\bar{U}_{i-1,j}^n+\bar{U}_{i,j}^n
	\right)$\\
}
\ENDIF
\STATE{
	$U_{i-\frac{1}{2},j}^{n+1}  =
	L_{EG2}(\tilde{U}_{i-\frac{1}{2},j},\Delta t,cpq) +
	C(x_{i-\frac{1}{2}},y_j,\Delta t)$  
}
\vspace*{0.2cm}
\STATE{$\tilde{U}_{i,j-\frac{1}{2}}=U_{i,j-\frac{1}{2}}^{n+\frac{1}{2}}$}
\IF{$\max \left\{\mathcal{I}^{tsw}_{i-\frac{1}{2},j-\frac{1}{2}}, \mathcal{I}^{tsw}_{i+\frac{1}{2},j-\frac{1}{2}}\right\}=1$}
\STATE{
	$\tilde{U}_{i,j-\frac{1}{2}}  = \frac{1}{2}
	\left(\bar{U}_{i,j-1}^n+\bar{U}_{i,j}^n
	\right)$\\
}
\ENDIF
\STATE{
	$U_{i,j-\frac{1}{2}}^{n+1}  =
	L_{EG2}(\tilde{U}_{i,j-\frac{1}{2}},\Delta t,cpq) +
	C(x_{i},y_{j-\frac{1}{2}},\Delta t)$  
}
\vspace*{0.2cm}
\STATE{$\tilde{U}_{i-\frac{1}{2},j-\frac{1}{2}}=U_{i-\frac{1}{2},j-\frac{1}{2}}^{n+\frac{1}{2}}$}
\IF{$\mathcal{I}^{tsw}_{i-\frac{1}{2},j-\frac{1}{2}}=1$ }
\STATE{$\tilde{U}_{i-\frac{1}{2},j-\frac{1}{2}}= \frac{1}{4}\left(\bar{U}_{i-1,j-1}^n+\bar{U}_{i,j-1}^n+\bar{U}_{i-1,j}^n+\bar{U}_{i,j}^n \right)$}
\ENDIF
\STATE{
	$U_{i-\frac{1}{2},j-\frac{1}{2}}^{n+1}  =
	L_{EG2}(\tilde{U}_{i-\frac{1}{2},j-\frac{1}{2}},\Delta t,cpq) +
	C(x_{i-\frac{1}{2}},y_{j-\frac{1}{2}},\Delta t)$  
}
\vspace*{0.5cm}
\STATE{{\small Note: Within each conjunction $\pm$ has the same sign
		and both cases need to be considered.}}
	\end{algorithmic}
\end{algorithm}

\subsection{Approximation of rarefaction waves}
We did not observe problems for the approximation of
transonic rare\-faction waves using
our third order accurate Active Flux method with
continuous piecewise quadratic reconstruction as described in \cref{alg:AFEu}.
This is in contrast to related finite volume methods
that use a piecewise constant reconstruction and require a special treatment
near transonic rarefaction waves (see, for example,
\cite{proc:LT2009}). We illustrate the performance of the method in a test computation.
\begin{example}\label{ex:transonicRarefaction}
  We consider the Euler equations with initial values of the form
  $$
  (\rho_0,u_0,v_0,p_0) = \left\{ \begin{array}{ccc}
                                   (0.1,-2,0,0.1) & : & x<0, \\
                                   (0.55,-1.5,0,0.55) & : & x=0,\\
                                   (1,-1,0,1) & : & x>0.
                                 \end{array}\right.
                               $$
The solution at time $t=0.4$ is computed on the intervall $[-1.5,0.5]$.                              
\end{example}
This is again a one-dimen\-sio\-nal problem which we approximate with our
two-dimen\-sio\-nal method using only few grid cells in the
$y$-direction and periodic boundary conditions.
\begin{figure}[htb]
	(a) \includegraphics[width=0.45\textwidth]{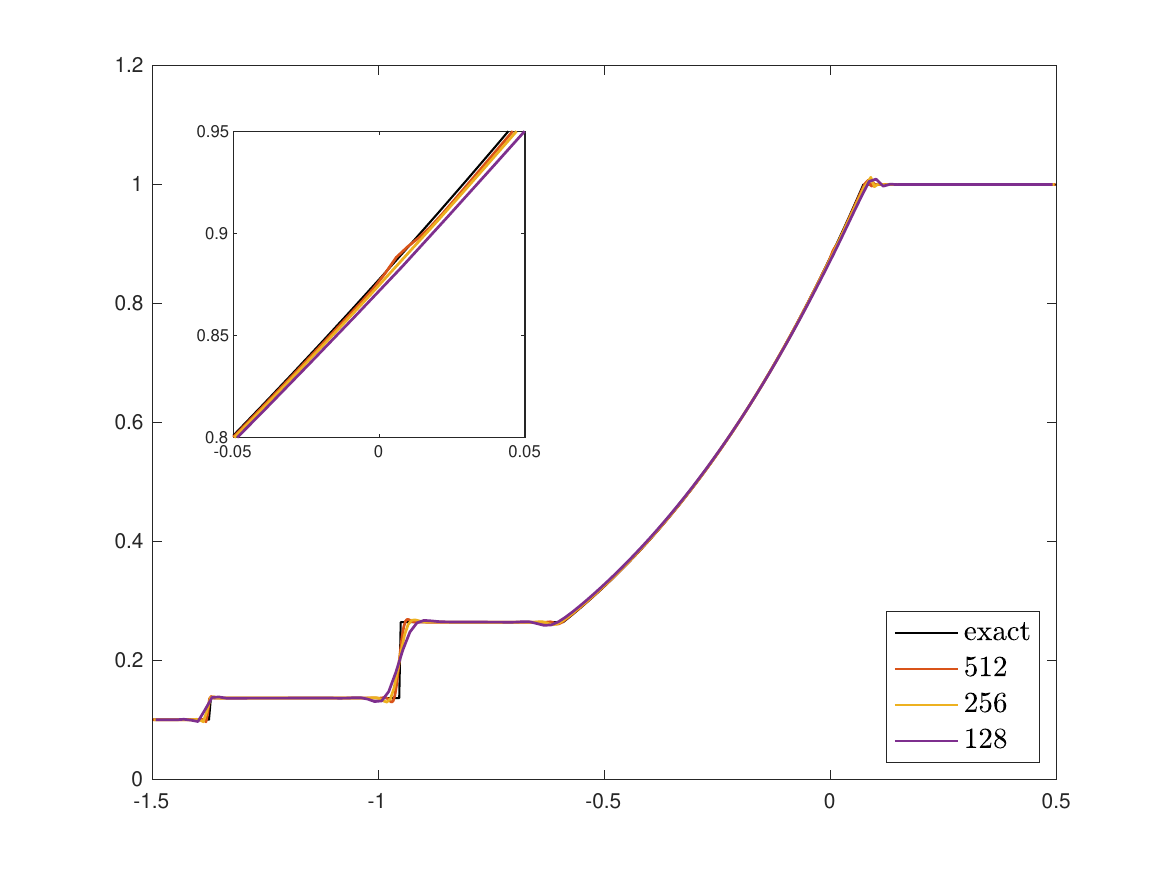}
	(b) \includegraphics[width=0.45\textwidth]{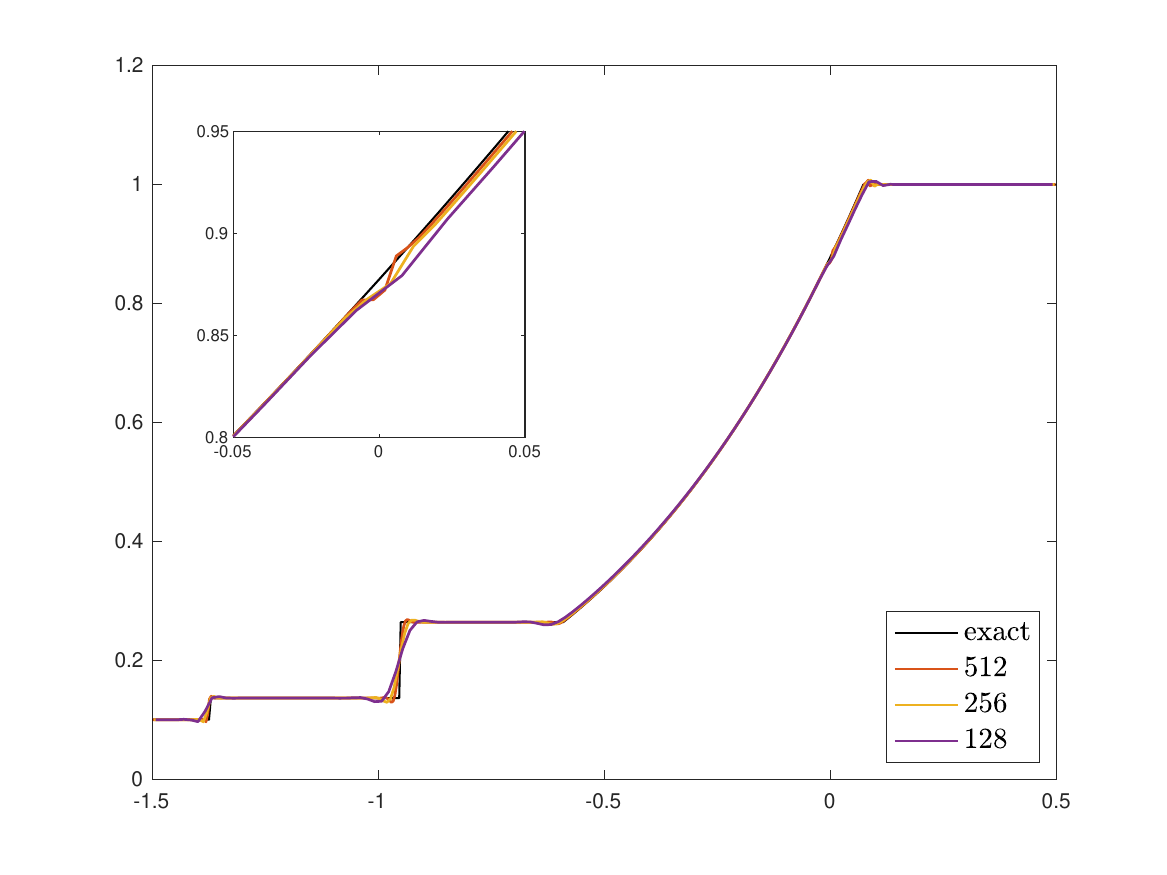}
	\caption{\label{fig:transonic} Numerical results for
          \cref{ex:transonicRarefaction} using (a) the third order accurate Active Flux method without limiting and (b) a version of the Active Flux method which uses the linearisation proposed for transonic shock waves.
	}
\end{figure}
In \cref{fig:transonic} (a) we present results obtained using our third order accurate Active Flux method without any limiters. We observe an accurate approximation of the transonic rarefaction wave without any entropy-fix.
In (b) we changed the local linearisation for the computation of
the full time step $t_{n+1}$ in such a way,
that we linearised around average values at time $t_n$, i.e., we used
the same linearisation for the half and the full time step.
While such a linearisation was needed for transonic shock waves, it  
deteriorates the accuracy of transonic rarefaction waves.

Another challenging one-dimensional test is the double
rarefaction wave problem from \cite{LindeRoe}.
The exact solution structure contains a vacuum and numerical methods
typically require some form of bound preserving limiting. 
\begin{example}\label{ex:double}
  We consider the Euler equations with initial values of the form
  $$
  (\rho_0,u_0,v_0,p_0) = \left\{ \begin{array}{ccc}
                                   (7,-1,0,0.2) & : & x<0.5, \\
                                   (7,0,0,0.2) & : & x=0.5,\\
                                   (7,1,0,0.2) & : & x>0.5.
                                                     \end{array}\right.
  $$
Solutions at time $t=0.3$ are computed on the intervall $[0,1]$.                                      
\end{example}
In \cref{fig:double} we show numerical results for \cref{ex:double} 
using three different grid resolutions. Our unlimited third order accurate
Active Flux method, see \cref{fig:double} (a), performs well for this
problem but shows some spurious
oscillations at the onset of the waves. Using the limiting of point
values, described below in
\cref{sec:shockIndicator}, these oscillations are reduced as shown in
\cref{fig:double} (b).
\begin{figure}[htb]
(a) \includegraphics[width=0.45\textwidth]{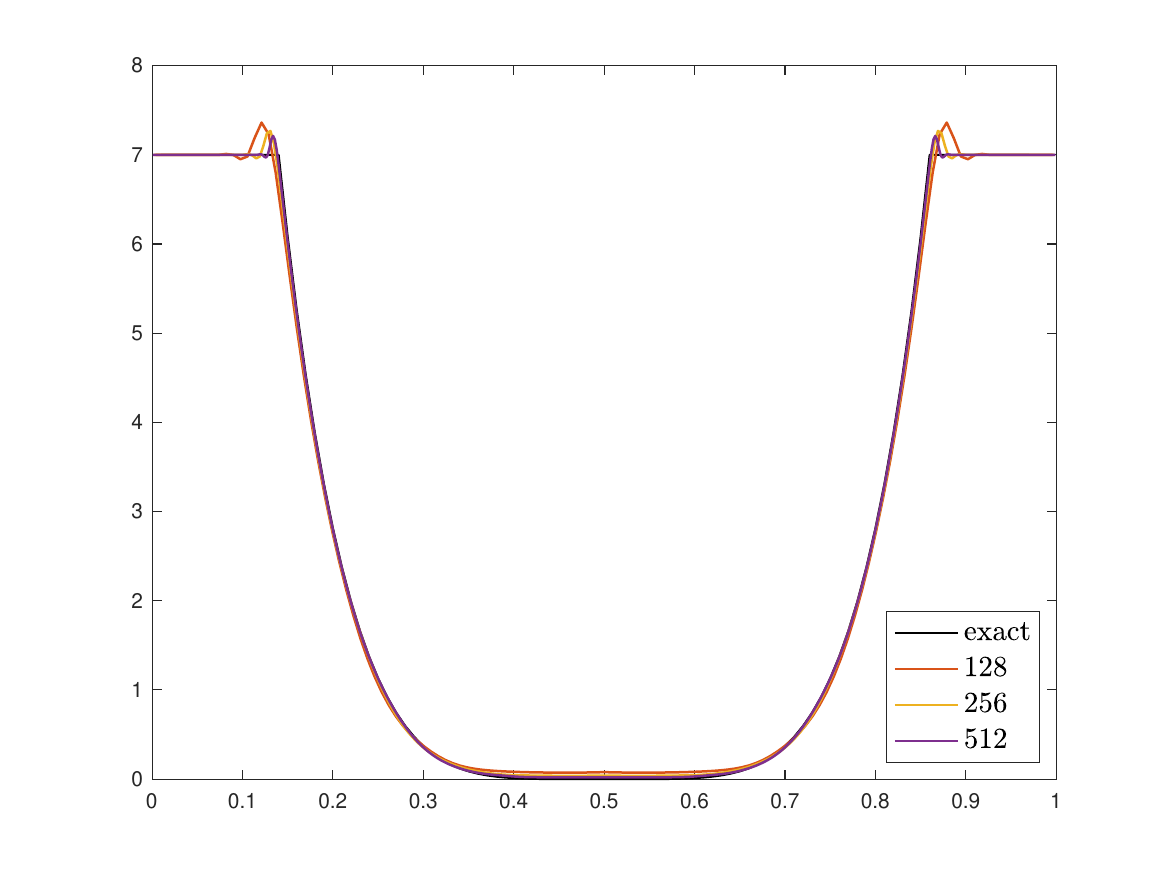}
(b) \includegraphics[width=0.45\textwidth]{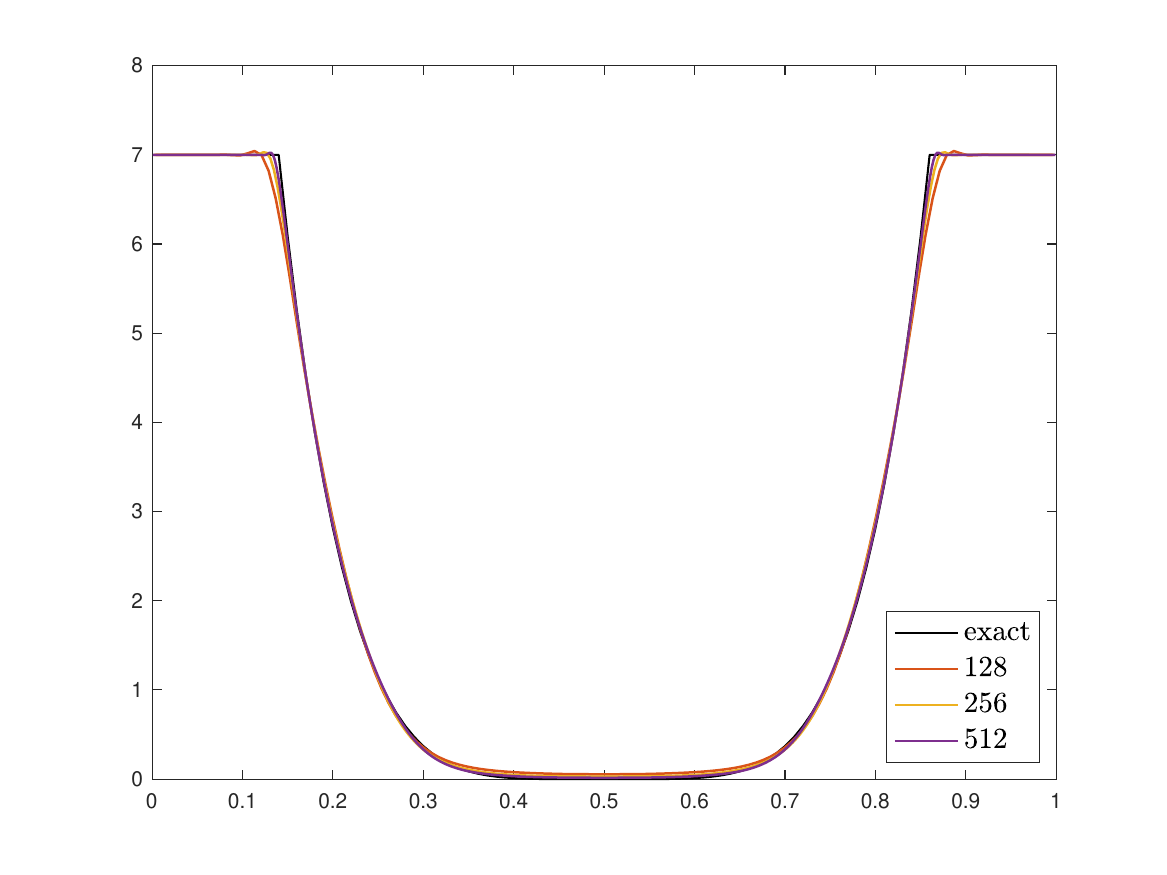}
\caption{\label{fig:double}
Plots of density for the double rarefaction problem given in
Example \ref{ex:double}. The solutions were computed on grids with
$128, 256$ and $512$ cells using
(a) no limiting and (b) limiters to reduce oscillations as described below.  
}
\end{figure}
In \cite[Supplementary Materials]{article:DBK2025}, the authors show
results of their one-dimensional generalised Active Flux methods, and
report that all tested versions of their method 
required bound preserving limiting. With our two-dimensional fully
discrete approach we obtain a stable approximation as long as the
linearisation for the computation of the full time step is performed
as predicted by our accuracy study and described in
\cref{alg:AFEu}. By instead linearising around the old time level (for
both the half and the new time step) or
the new time level (using an iterative approach), we observed
unphysical approximations.

\subsection{Bound-preserving approximation of point value degrees of
  freedom}
\label{sec:boundPreservingPointValues}
The EG2 evolution operator applied to the piecewise quadratic Active
Flux reconstruction does not guarantee to provide point values with
positive density and positive pressure. The same problem arises for
other high-order methods, in particular for
the method of lines approach studied in
\cite{preprint:ALB2025,article:DBK2025}, which uses the same degrees
of freedom. Those authors suggested to limit
the point values by using a linear combination of the
high order update and a Lax-Friedrichs update, that is guaranteed to be
bound preserving. 
In order to use a truly multi-dimensional approach we instead propose to replace
point values with unphysical density or pressure by  point values
computed using the method of bicharacteristics for piecewise
constant data.
Corresponding approximative  evolution operators have been proposed,
see \cite{article:LSW2002}, under the name EG1 and EG3. Here we use the EG1 method, almost identical
results have been obtained by using EG3.
\begin{algorithm}
  \caption{Bound-preserving evolution of point values located at 
    vertices}
  \label{alg:limPointValues}
  \begin{algorithmic}
\STATE{Compute high order accurate approximation
  $U_{i-\frac{1}{2},j-\frac{1}{2}}^{n+\frac{1}{2},(ho)}$ using (\ref{eqn:Un+1/2}).}\\

\STATE{Set $U_{i-\frac{1}{2},j-\frac{1}{2}}^{n+\frac{1}{2}} =
U_{i-\frac{1}{2},j-\frac{1}{2}}^{n+\frac{1}{2},(ho)}$}

\IF{$\rho_{i-\frac{1}{2},j-\frac{1}{2}}^{n+\frac{1}{2},(ho)}< 0$ or
  $p_{i-\frac{1}{2},j-\frac{1}{2}}^{n+\frac{1}{2},(ho)}< 0$}
\STATE{compute low order approximation
  \begin{align*}
U_{i-\frac{1}{2},j-\frac{1}{2}}^{n+\frac{1}{2},(lo)}  & :=
L_{EG1}(\tilde{U}_{i-\frac{1}{2},j-\frac{1}{2}}^n,\Delta t/2,pc) \hspace*{5cm}\\
\mbox{with} \quad \tilde{U}_{i-\frac{1}{2},j-\frac{1}{2}}^n & :=
\frac{1}{4}\left(\bar{U}_{i-1,j-1}^n+\bar{U}_{i,j-1}^n+\bar{U}_{i-1,j}^n+\bar{U}_{i,j}^n\right).
  \end{align*} }

\STATE{set $U_{i-\frac{1}{2},j-\frac{1}{2}}^{n+\frac{1}{2}} = U_{i-\frac{1}{2},j-\frac{1}{2}}^{n+\frac{1}{2},(lo)}$}\\[0.2cm]

\IF{$\rho_{i-\frac{1}{2},j-\frac{1}{2}}^{n+\frac{1}{2},(lo)}< 0$ or
  $p_{i-\frac{1}{2},j-\frac{1}{2}}^{n+\frac{1}{2},(lo)}< 0$}

\STATE{compute  
  \begin{align*}
    Q_{i-\frac{1}{2},j-\frac{1}{2}}^{n+\frac{1}{2},(LF)}  =
    Q_{i-\frac{1}{2},j-\frac{1}{2}}^{n}  & - \frac{\Delta t}{2\Delta
                                                                 x}(F_{i,j-\frac{1}{2}}^{LF}-F_{i-1,j-\frac{1}{2}}^{LF})
    \hspace*{3cm}\\
        & - \frac{\Delta t}{2\Delta y}(G_{i-\frac{1}{2},j}^{LF}-G_{i-\frac{1}{2},j-1}^{LF})                                                        
    \end{align*}}
\STATE{compute $U_{i-\frac{1}{2},j-\frac{1}{2}}^{n+\frac{1}{2}}$ from
  $Q_{i-\frac{1}{2},j-\frac{1}{2}}^{n+\frac{1}{2},(LF)}$}   
\ENDIF
\ENDIF
\vspace*{0.2cm}
\STATE{Point values at midpoints of edges are treated analogously.}
\vspace*{0.2cm}

\STATE{Point values at time $t_{n+1}$ are limited analogously.}  
\end{algorithmic}
\end{algorithm}
Unfortunately, there is no proof that these first order accurate
truly multi-dimensional evolution operators are guaranteed bound
preserving. Therefore, we suggest to evolve the point values using a
Lax-Friedrichs evolution operator if both the third and the first order  
multi-dimensional evolution operator computes an unphysical density or
pressure. Now the bound preservation of the point value
update follows from \cite[Lemma 4.6]{article:DBK2025}.
In practical computations we rarely observe unphysical results when
using the multidimensional,  first-order accurate point value update.

The update of point values at vertices is described in
\cref{alg:limPointValues}. There we used the superscript $(ho)$ and
$(lo)$ to indicate high and low order, respectively.
Midpoints of edges are limited analogously.

Using the definition of $P$, $P'$ and $Q(\theta)$
as introduced above, the EG1 operator can be found in \cite[Equations (2.17),(2.20),(2.22) and (2.23)]{article:LSW2002}.
\ignore{
\begin{equation}
	\label{eqn:EG1}
	\begin{aligned}
		\rho(P) & =\rho({P'})-\frac{p({P'})}{c'^2}+
		\frac{1}{2\pi}\int_{0}^{2\pi}\frac{p(Q(\theta))}{c'^2}-2\frac{\rho'}{c'}u(Q(\theta))
		\cos(\theta)\\
		& \hspace*{5.65cm} -2\frac{\rho'}{c'}v(Q(\theta)) \sin(\theta)
		\  {\rm d} \theta + {\cal O}(\tau^2)\\
		u(P) & =\frac{1}{2}u(P')+\frac{1}{2\pi}\int_{0}^{2\pi}
		-\frac{2p(Q(\theta))}{\rho'c'}\cos(\theta)+u(Q(\theta))
		\left(3\cos^2(\theta)-1\right) \\
		& \hspace*{4cm} +3v(Q(\theta)) \sin(\theta)\cos(\theta)
		\,  {\rm d} \theta + {\cal O}(\tau^2)\\
		v(P) & =\frac{1}{2}v(P')+\frac{1}{2\pi}\int_{0}^{2\pi}
		-\frac{2p(Q(\theta))}{\rho'c'}\sin(\theta)+3u(Q(\theta)) \sin(\theta)
		\cos(\theta)\\
		& \hspace*{4cm}+v(Q(\theta)) \left(3\sin^2(\theta)-1\right)  \,
		{\rm d} \theta + {\cal O}(\tau^2)\\
		p(P) &=\frac{1}{2\pi}\int_{0}^{2\pi}
		p(Q(\theta))  -2\rho'c'u(Q(\theta)) \cos(\theta) \\
		& \hspace*{4.2cm}
		-2\rho'c'v(Q(\theta)) \sin(\theta) \,  {\rm d} \theta + {\cal O}(\tau^2).
	\end{aligned}
      \end{equation}
}
We apply this operator to piecewise constant data using
the same local linearisation as for the high order method.

A Lax-Friedrichs update of point values in conservative variables
using only point value degrees of freedom can be found in
\cite[Equations (4.15)-(4.17)]{article:DBK2025}.
These point values can easily be converted to bound
  preserving point values in primitive variables.

\subsection{Additional limiting of point values based on a shock indicator}
\label{sec:shockIndicator}
Although the method described in \cref{alg:limPointValues} ensures
positivity of density and pressure for all point value degrees of freedom,
numerical simulations will typically exhibit  unphysical oscillations
in the vicinity of shock waves. These oscillations can be limited by
adding numerical viscosity.   Here we will explore limiting by using
a lower order point value
approximation. A shock indicator shows whether limiting should
be applied. In \cite{article:DBK2025} a similar limiting procedure for
fluxes was used.
Applying the limiter to point
  values in primitive variables allows us a more precise
  positioning of the additional numerical viscosity. 
We use
\begin{align*}
U_{i-\frac{1}{2},j}^{n+k,(lim)} & := \theta_{i-\frac{1}{2},j}
                                  U_{i-\frac{1}{2},j}^{n+k,(ho)} +
                                  \left(1-\theta_{i-\frac{1}{2},j}\right)U_{i-\frac{1}{2},j}^{n+k,(lo)},\\
U_{i-\frac{1}{2},j-\frac{1}{2}}^{n+k,(lim)} & := \theta_{i-\frac{1}{2},j-\frac{1}{2}}
                                  U_{i-\frac{1}{2},j-\frac{1}{2}}^{n+k,(ho)}
                                              +
                                              \left(1-\theta_{i-\frac{1}{2},j-\frac{1}{2}}\right)U_{i-\frac{1}{2},j-\frac{1}{2}}^{n+k,(lo)},
  \qquad
                                              k \in \left\{ \frac{1}{2},1 \right\},\\
\end{align*}
and analogously for $U_{i,j-\frac{1}{2}}^{n+k,(lim)}$.
Here $(lo)$ again indicates the solution obtained with the EG1 method
applied to piecewise constant data.
To avoid unnecessary limiting at contact discontinuities, the shock
indicator looks for changes in pressure as proposed in
\cite{article:DBK2025,JST1981}. Furthermore, we also take the wave
speed into account and add more limiting in regions of faster flow.
Thus, the parameter $\theta \in [0,1]$ is computed from two quantities
$\phi^{(1)}$ and $\phi^{(2)}$  via
$$
\theta =  \exp(-\phi^{(1)} \phi^{(2)}).
$$
The precise form of $\phi^{(1)}$ and $\phi^{(2)}$ depends on the
location of the point value that needs to be limited. We use
\begin{align*}
\phi_{i-\frac{1}{2},j}^{(1)} & = \max \left( \Big|
                               \frac{\bar{p}_{i+1,j}-2\bar{p}_{i,j}+\bar{p}_{i-1,j}}{\bar{p}_{i+1,j}+2\bar{p}_{i,j}+\bar{p}_{i-1,j}}\Big|,
                               \Big| \frac{\bar{p}_{i,j}-2\bar{p}_{i-1,j}+\bar{p}_{i-2,j}}{\bar{p}_{i,j}+2\bar{p}_{i-1,j}+\bar{p}_{i-2,j}}\Big|
                               \right),\\
\phi_{i,j-\frac{1}{2}}^{(1)} & =  \max \left( \Big|
                               \frac{\bar{p}_{i,j+1}-2\bar{p}_{i,j}+\bar{p}_{i,j-1}}{\bar{p}_{i,j+1}+2\bar{p}_{i,j}+\bar{p}_{i,j-1}}\Big|,
                               \Big| \frac{\bar{p}_{i,j}-2\bar{p}_{i,j-1}+\bar{p}_{i,j-2}}{\bar{p}_{i,j}+2\bar{p}_{i,j-1}+\bar{p}_{i,j-2}}\Big|
                               \right),\\
\ignore{\phi_{i-\frac{1}{2},j-\frac{1}{2}}^{(1)} & = \max\left( \phi_{i,j-\frac{1}{2}}^{(1)},\phi_{i-\frac{1}{2},j}^{(1)},\phi_{i-1,j-\frac{1}{2}}^{(1)},\phi_{i-\frac{1}{2},j-1}^{(1)}\right)  }
\ignore{  
\phi_{i-\frac{1}{2},j-\frac{1}{2}}^{(1)} & =  \max \left(
                               \Big|\frac{\bar{p}_{i+1,j}-2\bar{p}_{i,j}+\bar{p}_{i-1,j}}{\bar{p}_{i+1,j}+2\bar{p}_{i,j}+\bar{p}_{i-1,j}}\Big|,
                                           \Big|\frac{\bar{p}_{i,j}-2\bar{p}_{i-1,j}+\bar{p}_{i-2,j}}{\bar{p}_{i,j}+2\bar{p}_{i-1,j}+\bar{p}_{i-2,j}}\Big|,                                           
                                           \right.\\
  &  \left. \Big| \frac{\bar{p}_{i+1,j-1}-2\bar{p}_{i,j-1}+\bar{p}_{i-1,j-1}}{\bar{p}_{i+1,j-1}+2\bar{p}_{i,j-1}+\bar{p}_{i-1,j-1}}\Big|,
                               \Big|\frac{\bar{p}_{i,j-1}-2\bar{p}_{i-1,j-1}+\bar{p}_{i-2,j-1}}{\bar{p}_{i,j-1}+2\bar{p}_{i-1,j-1}+\bar{p}_{i-2,j-1}}\Big|
                                           \right)}
\end{align*}
and
\begin{align*}
\phi_{i-\frac{1}{2},j}^{(2)} & = 2^{\max( |\bar{u}_{i,j}| +
                               \bar{c}_{i,j},|\bar{u}_{i-1,j}|+\bar{c}_{i-1,j})},\\
\phi_{i,j-\frac{1}{2}}^{(2)} & = 2^{\max(|\bar{v}_{i,j}| +
                               \bar{c}_{i,j},|\bar{v}_{i,j-1}|+\bar{c}_{i,j-1})}.\\
\ignore{\phi_{i-\frac{1}{2},j-\frac{1}{2}}^{(2)} & = \max\left( \phi_{i,j-\frac{1}{2}}^{(2)},\phi_{i-\frac{1}{2},j}^{(2)},\phi_{i-1,j-\frac{1}{2}}^{(2)},\phi_{i-\frac{1}{2},j-1}^{(2)}\right).  }
\ignore{  
\phi_{i-\frac{1}{2},j-\frac{1}{2}}^{(2)} & = 2^{\max( w_{i,j} +
                               \bar{c}_{i,j},w_{i-1,j}+\bar{c}_{i-1,j},
                                           w_{i,j-1} +
                               \bar{c}_{i,j-1},w_{i-1,j-1}+\bar{c}_{i-1,j-1})}, } 
\end{align*}
\ignore{
with $w:= \max(|\bar{u}|,|\bar{v}|)$.
Analogously we compute $\phi_{i,j-\frac{1}{2}}^{(1)}$ and $\phi_{i,j-\frac{1}{2}}^{(2)}$.}
For the corners we use
\begin{align*}
	\theta_{i-\frac{1}{2},j-\frac{1}{2}} = \min \left\{ {\theta_{i-\frac{1}{2},j},\theta_{i-\frac{1}{2},j-1},\theta_{i,j-\frac{1}{2}},\theta_{i-1,j-\frac{1}{2}}}\right\}.
\end{align*}

 While our choice of $\phi^{(1)}$ uses an established
  approach from the literature, our choice of $\phi^{(2)}$ is
  experimentally. As shown in \cref{sec:4}, we obtain accurate results
  for a wide variety of classical test problems without using an additional
  tuning parameter. However, for flow problems in the low Mach number regime, our
  choice of $\phi^{(2)}$ might indicate the need for limiting although
  no limiting is needed and limiting might deteriorate the
  quality of the numerical approximations. It would be easy to include additional
  criteria or other definitions of $\phi^{(2)}$ to make the limiter
  applicable for a wider range of flow problems. The crucial step of
  our approach is that limiting is applied to point values in primitive
  variables based on criteria that can easily be adapted if needed. 

\ignore{
\textcolor{red}{As shown in \cref{sec:4}, the limiter yields good results for problems with discontinuous data. For other cases, such as low-Mach-number flows, the $\phi^{(2)}$ part may lead to unintended limiting in smooth regions, even though such problems do not require any limiting.}}

\subsection{Additional limiting of fluxes}
The use of limited point values does not guarantee that cell average
values of pressure and
density, computed from cell average values of the conserved
quantities, are positive. 
Computing the local linearisation from unphysical values of density
or pressure would dismantle the hyperbolic structure.
Therefore, the finite volume method for the evolution of the conserved
quantities might require an additional flux limiting.

Flux limiting procedures for related finite volume methods, which use
point values and cell average values as degrees of freedom, have
recently been proposed by Duan et al. \cite{article:DBK2025} and
Abgrall at al. \cite{preprint:ALB2025}. These methods built on earlier
work by Zhang and Shu \cite{article:ZS2010,article:ZS2010b}, Wu and
Shu \cite{article:WS2023} as well as Kuzmin 
\cite{article:Kuzmin2020}.

The general idea, in analogy to the limiting of point values, is
to replace the fluxes in the finite volume method
(\ref{eqn:2dfvm}) with limited fluxes of the form
\begin{align*}
F_{i-\frac{1}{2},j}^{(lim)} & := \gamma_{i-\frac{1}{2},j}
  F_{i-\frac{1}{2},j}^{(ho)} + (1-\gamma_{i-\frac{1}{2},j})
                            F_{i-\frac{1}{2},j}^{(lo)},\\
G_{i,j-\frac{1}{2}}^{(lim)} & :=
                            \gamma_{i,j-\frac{1}{2}}G_{i,j-\frac{1}{2}}^{(ho)}
                            + (1-\gamma_{i,j-\frac{1}{2}})G_{i,j-\frac{1}{2}}^{(lo)}.  
\end{align*}
Here the high order fluxes are computed using Simpson's rule with
possibly limited point values as described above and the low order
flux guaranties the bound preserving properties of the resulting cell average values.
The parameter $\gamma_{i-\frac{1}{2},j} \in [0,1]$ needs to be chosen as large as possible
under the restriction that the cell average values are bound
preserving. The difficulty is that the pressure depends on the
density and that $ \gamma_{i-\frac{1}{2},j}$ must be selected so that
the cell averages of both variables are positive.

The approach from \cite[Section 4.1.2]{article:DBK2025} could be applied to our
fully discrete finite volume method to provide bound preservation of
cell average values. 
However, for all test problems studied in this paper, a carefully chosen local
linearisation and the limiting of point values
as described in \cref{sec:boundPreservingPointValues} and
\cref{sec:shockIndicator} was sufficient to get admissible cell
average values. Therefore, we did not use additional flux limiting.

\section{Discussion of numerical results}
\label{sec:4}
We now present computational results which illustrate the
performance of our method for two-dimensional problems
with discontinuous solution structures. 
\subsection{Two-dimensional Riemann problems}
Now we present computational results for several classical two-dimensional
Riemann problems described in \cite{article:RCG1993}. All test
problems are computed on the domain $[0,1]\times [0,1]$ with outflow
boundary conditions. The initial values consist of
four constant states in the four quadrants. We start with
Configuration F, which contains  two shock waves and two stationary
contact discontinuities. This relatively simple configuration can be
computed without any limiting. Numerical results for such a
computation are shown in  \cref{fig:RP12} (top). This unlimited solution shows
small oscillations near the shock waves, which are only visible in the
plot of density  for constant $y$ shown in \cref{fig:Config12_yfix}.
By using the parameter free limiting
described in \cref{sec:shockIndicator} these oscillations can be
elimited as shown in \cref{fig:RP12} (bottom) and \cref{fig:Config12_yfix}.  
\begin{figure}
  \includegraphics[width=0.32\linewidth]{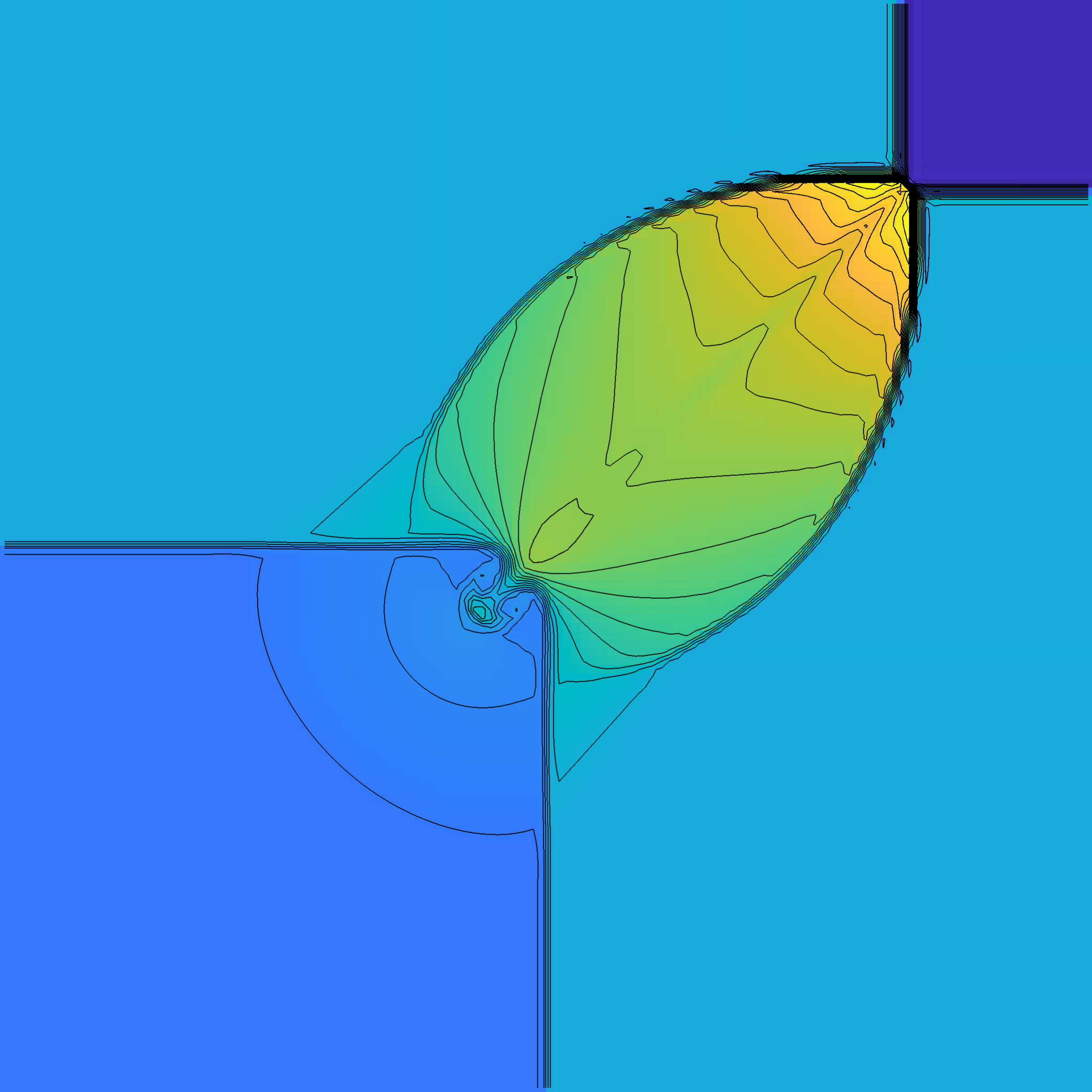}
  \includegraphics[width=0.32\linewidth]{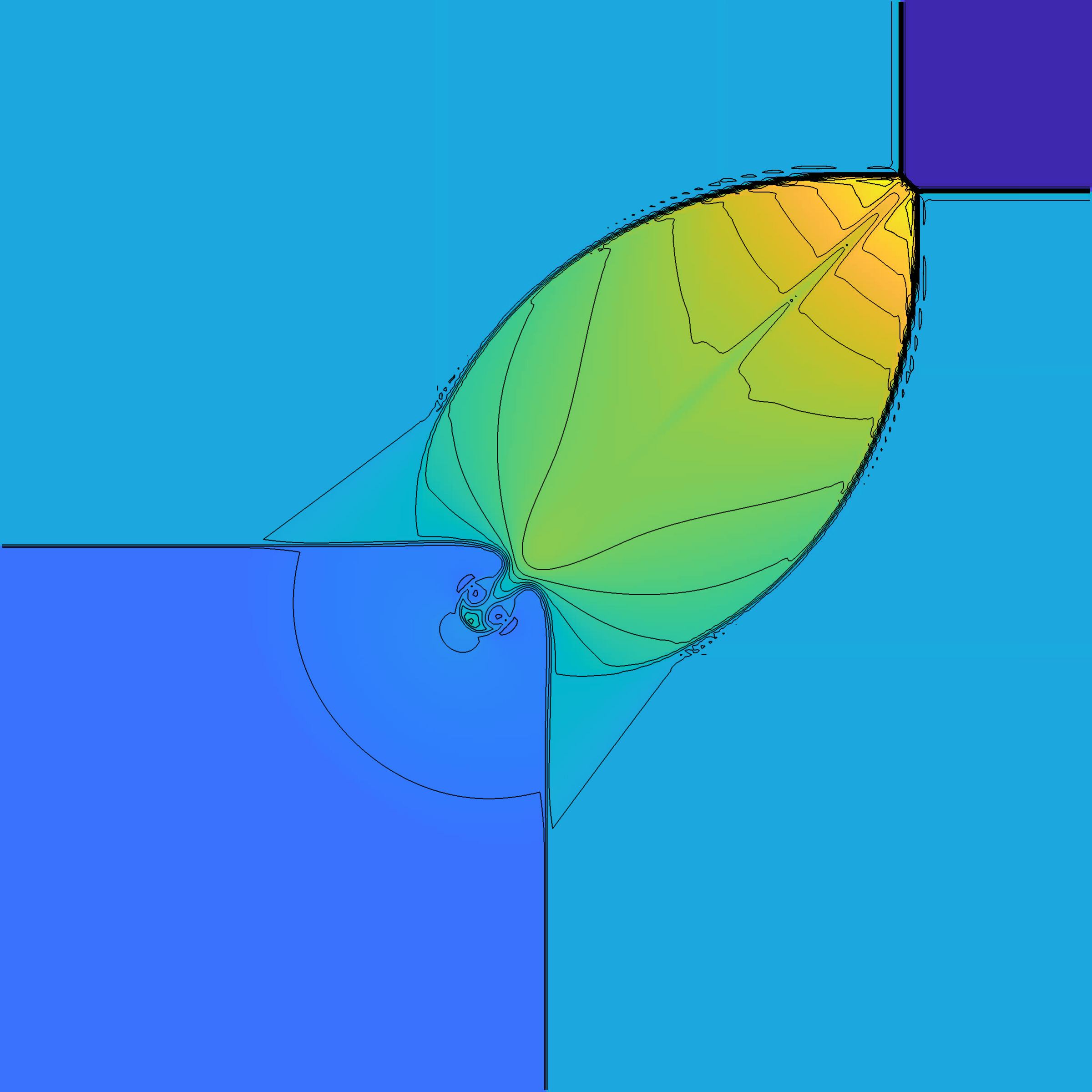}
  \includegraphics[width=0.32\linewidth]{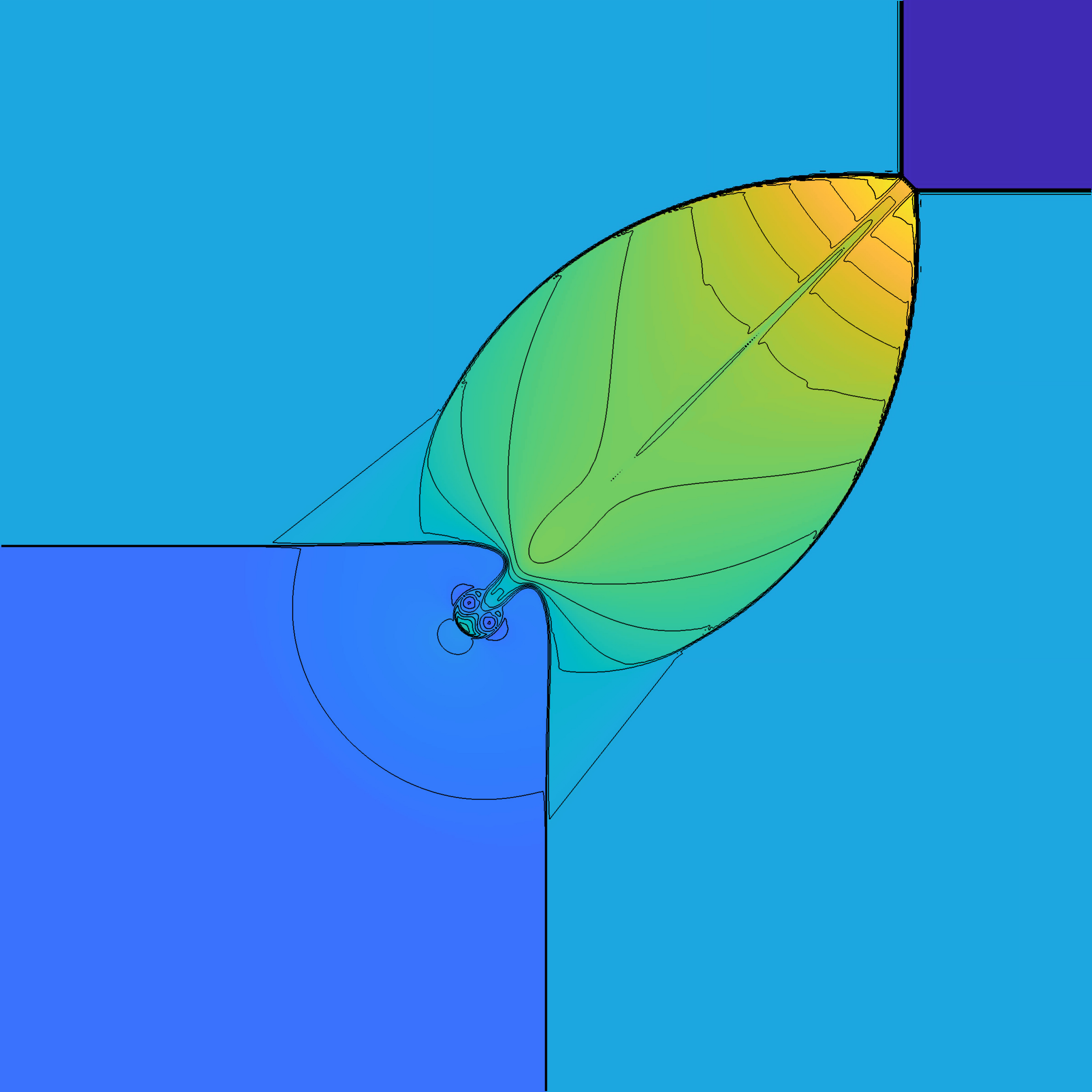}\\[0.15cm]
  \includegraphics[width=0.32\linewidth]{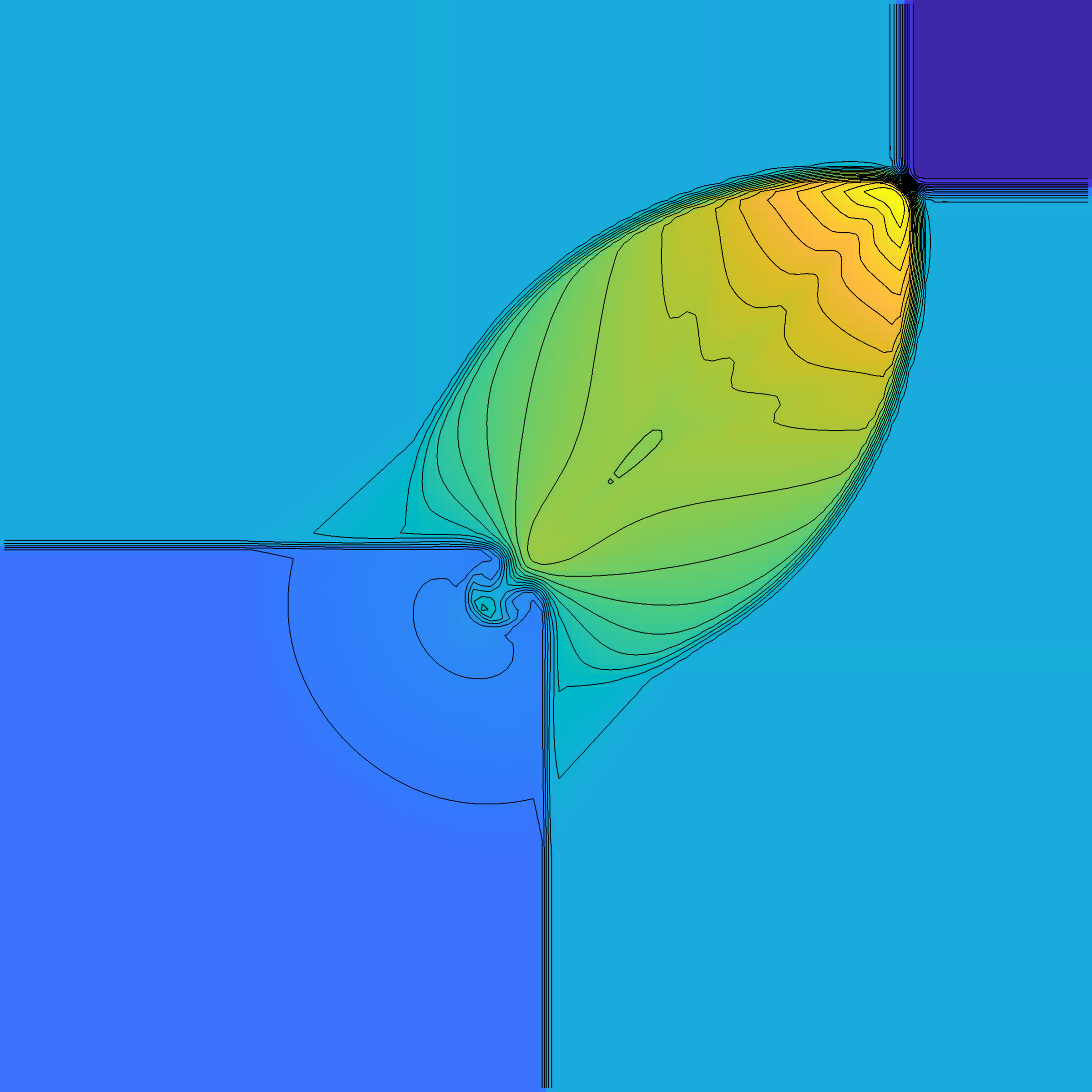}
  \includegraphics[width=0.32\linewidth]{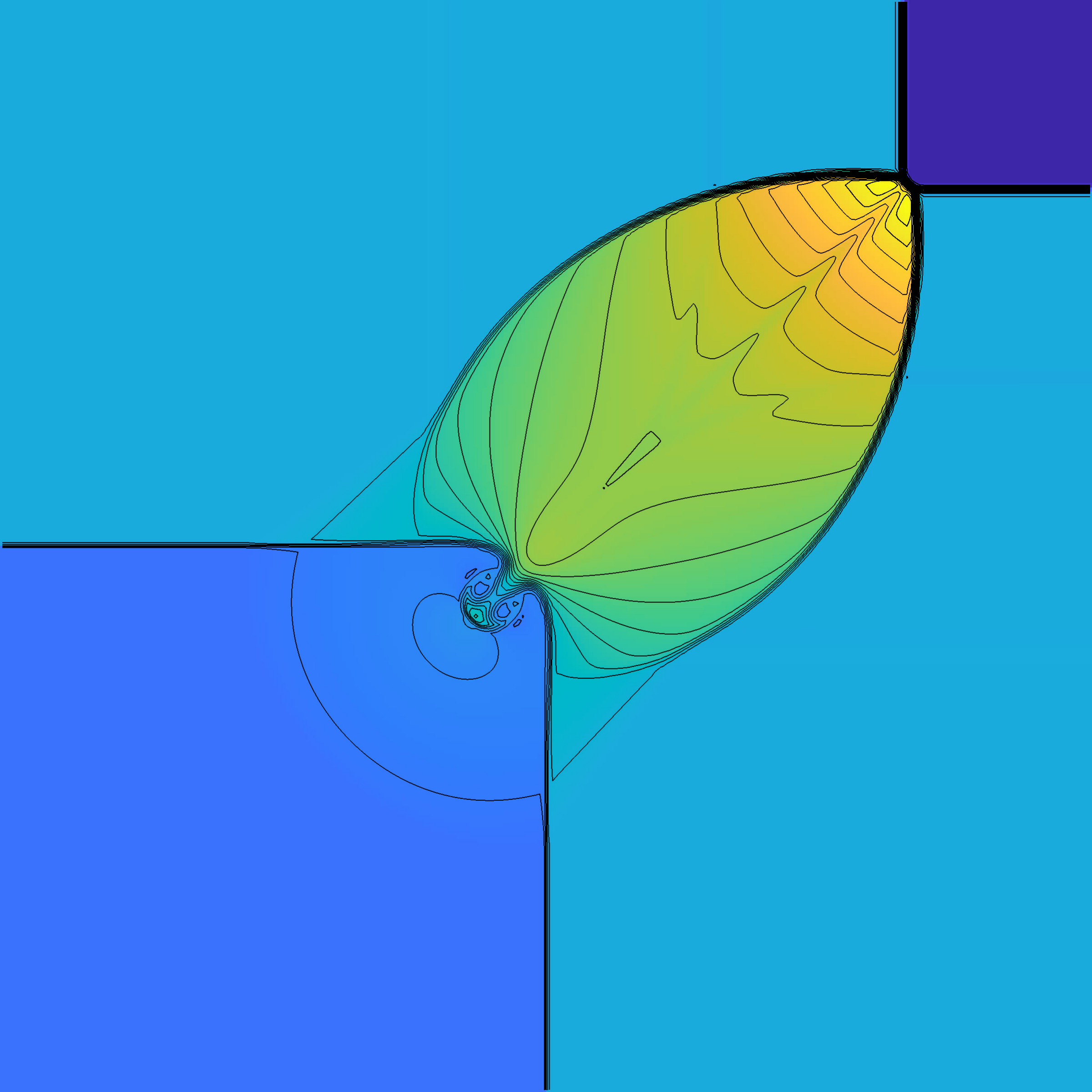}
    \includegraphics[width=0.32\linewidth]{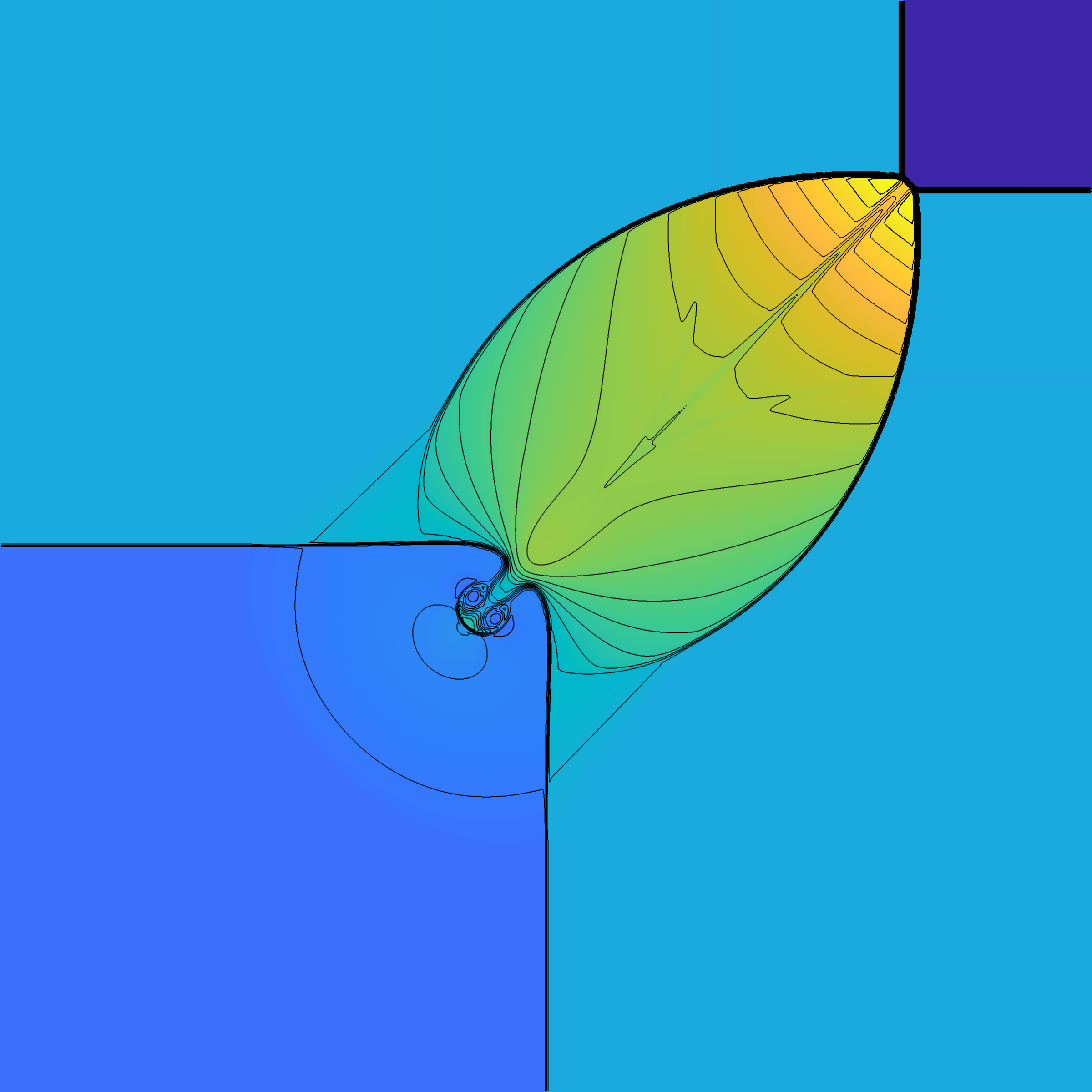}
  \caption{\label{fig:RP12}Density at time $t=0.21$
    for the two-dimensional Riemann problem described in
    \cite[Configuration F]{article:RCG1993}. From left to right we
    used grids with $128^2$, $256^2$ and $512^2$ grid cells. In the top row the third order accurate Active
    Flux method was used without any limiting. In the bottom row the
    point values were limited to avoid oscillations near shocks as
    described in \cref{sec:shockIndicator}.}
\end{figure}
The solution structure is well approximated even on the coarsest
grid. Note that the same problem was also studied in
\cite{article:ABK2025} using a semi-discrete form of the Active Flux
method. Our results on the $128^2$ grid seem to
lead to a more accurate approximation of the solution structure than the semi-discrete method on a grid with  $200^2$ cells. 

\begin{figure}
\includegraphics[width=0.32\linewidth]{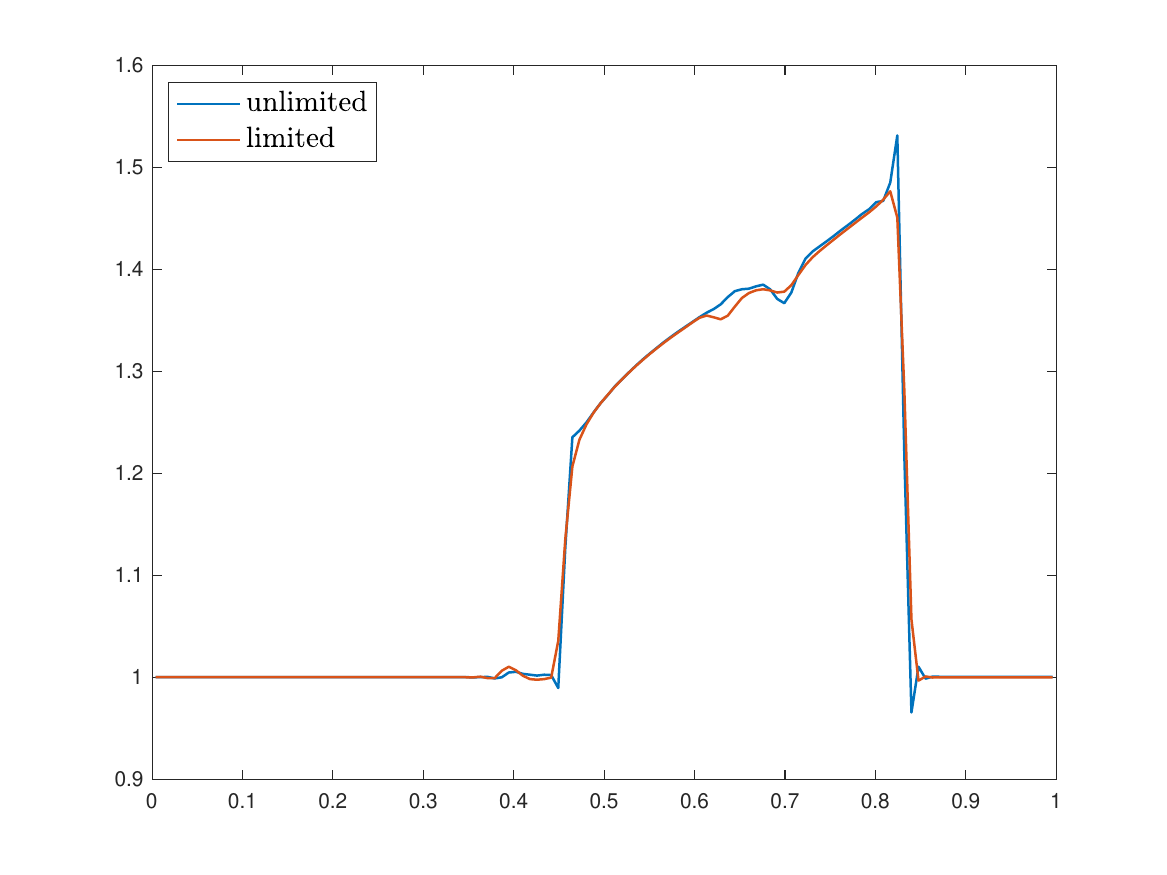} \hfill
\includegraphics[width=0.32\linewidth]{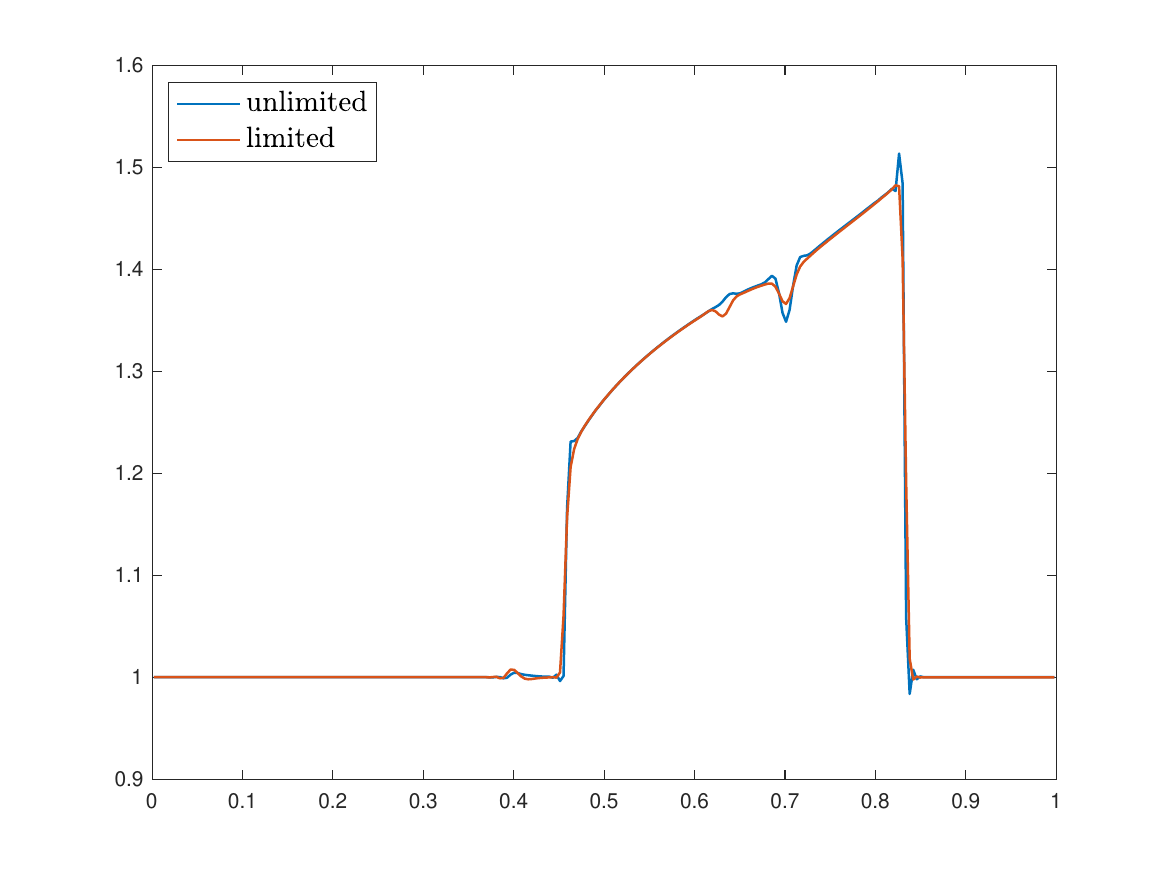} \hfill
\includegraphics[width=0.32\linewidth]{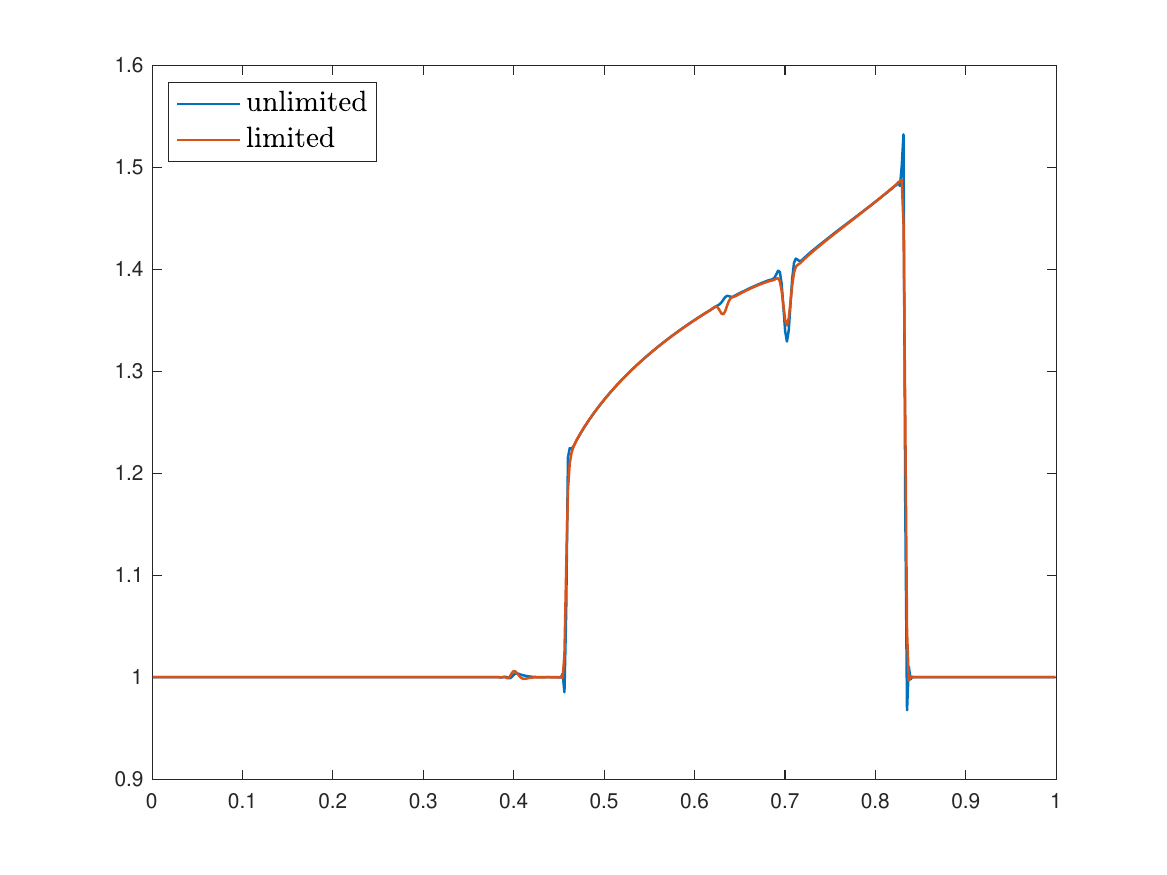}
\caption{  \label{fig:Config12_yfix}Comparison of the limited and unlimited solution 
structure. Density for Configuration F along a slice with $y=0.7012$ using  $128^2$ cells (left), $256^2$ cells  (middle) and  $512^2$ cells (right).}
\end{figure}

Slightly more challenging are the Configurations 4, E and J from
\cite{article:RCG1993}. For Configurations 4 and J some point values
could not be evolved using the third order accurate evolution operator
and thus the bound preserving limiter for point values, described in
\cref{sec:boundPreservingPointValues}, was needed.  In \cref{fig:RPs} we give some results.
\begin{figure}
\includegraphics[width=0.32\linewidth]{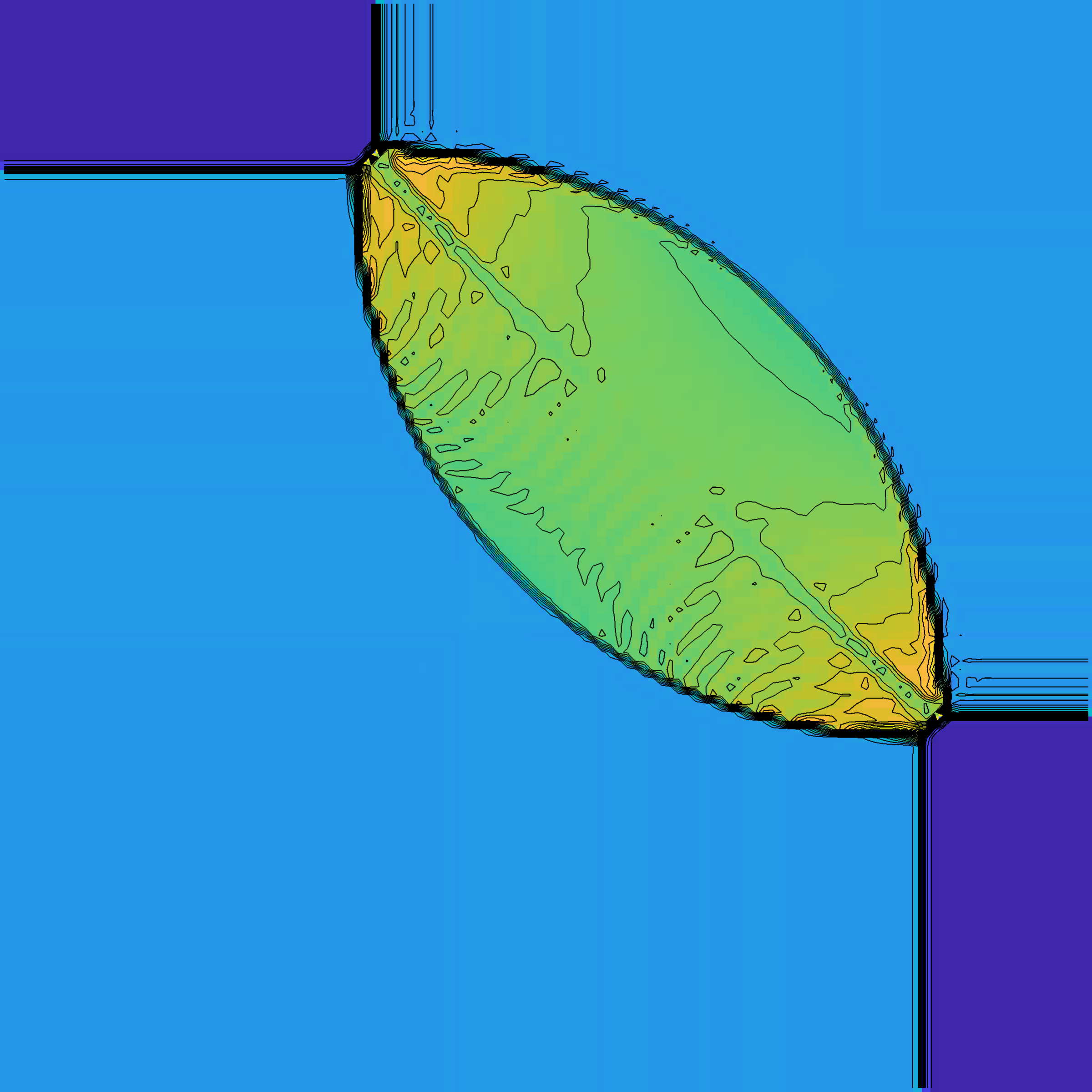}\hfill
\includegraphics[width=0.32\linewidth]{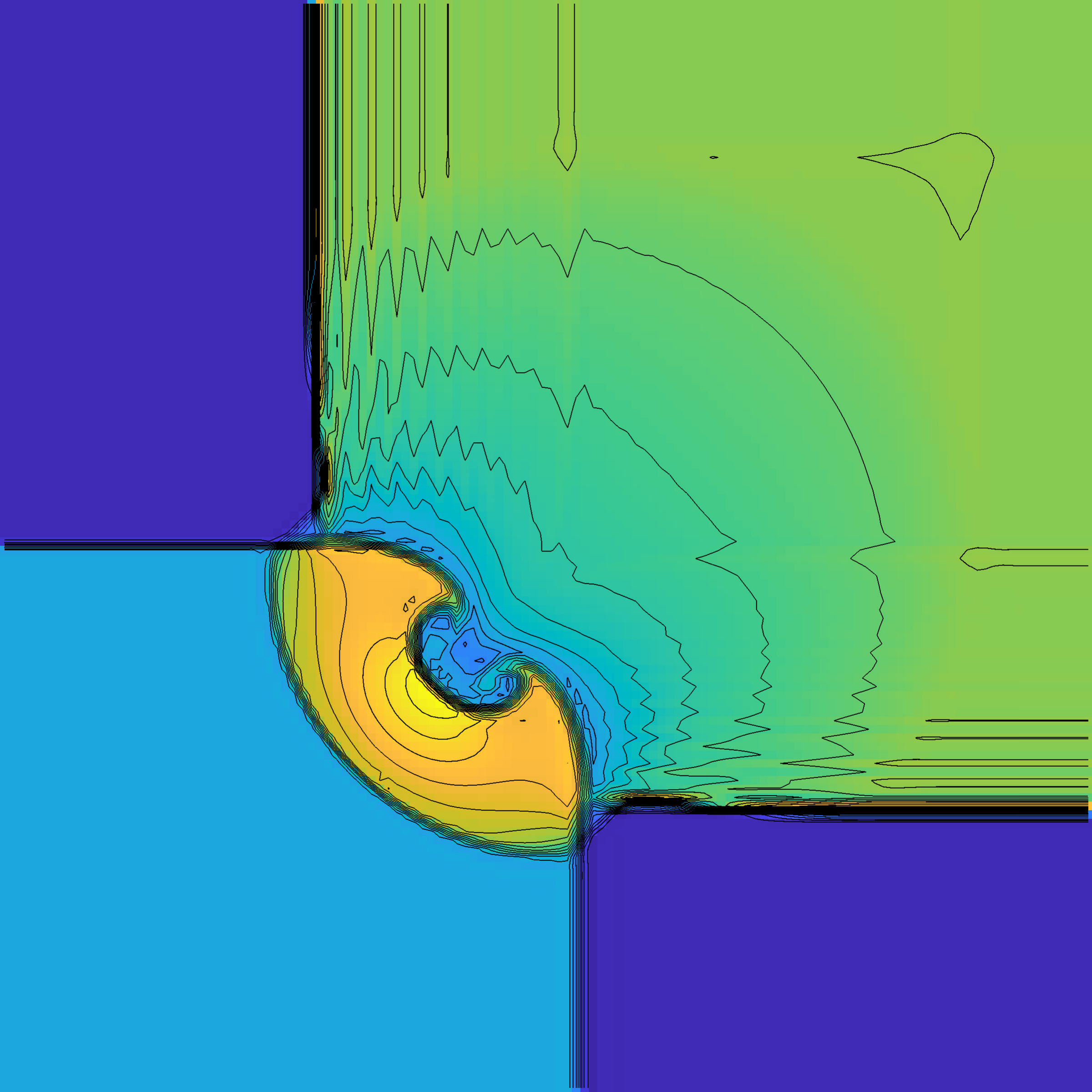}\hfill
\includegraphics[width=0.32\linewidth]{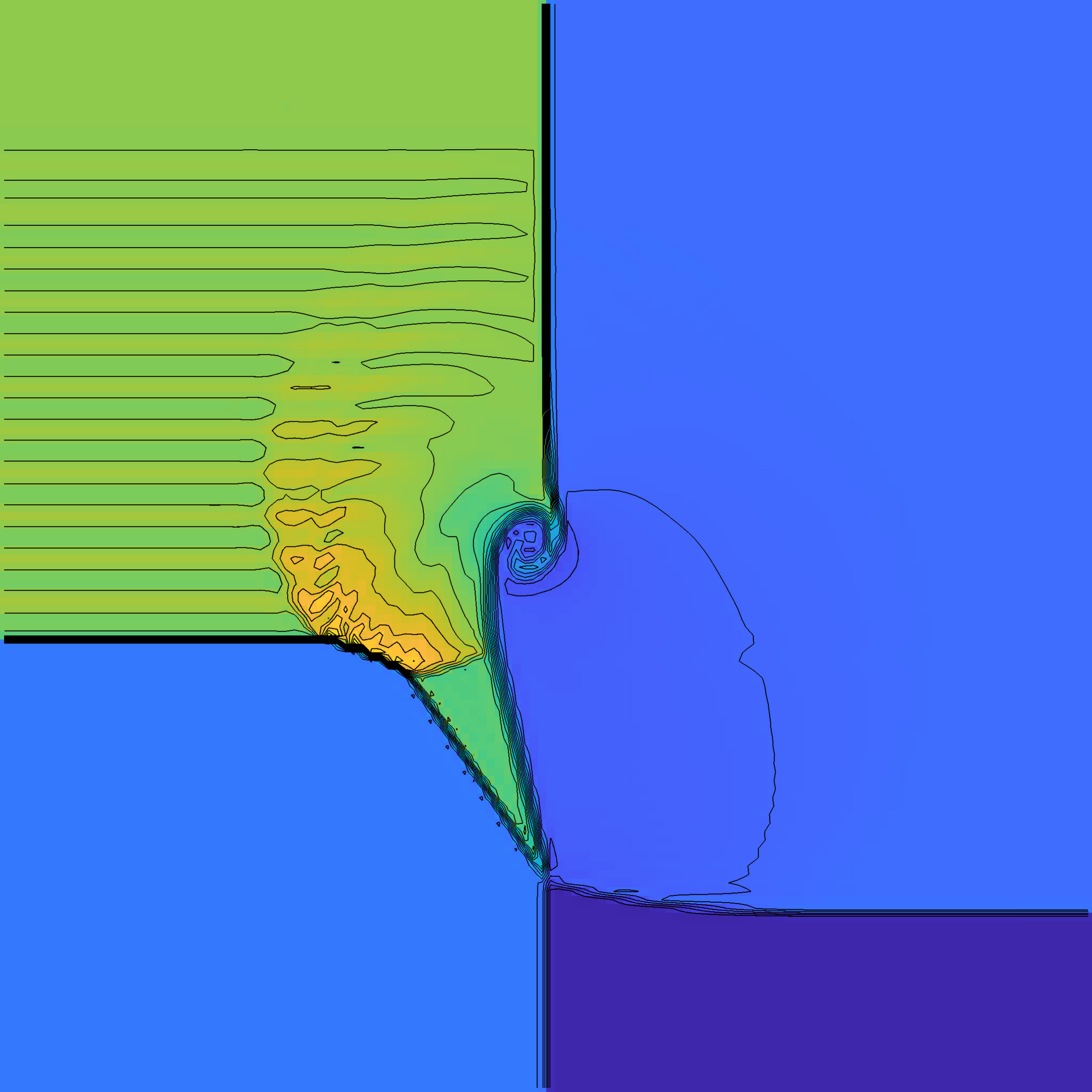}\\[0.15cm]
\includegraphics[width=0.32\linewidth]{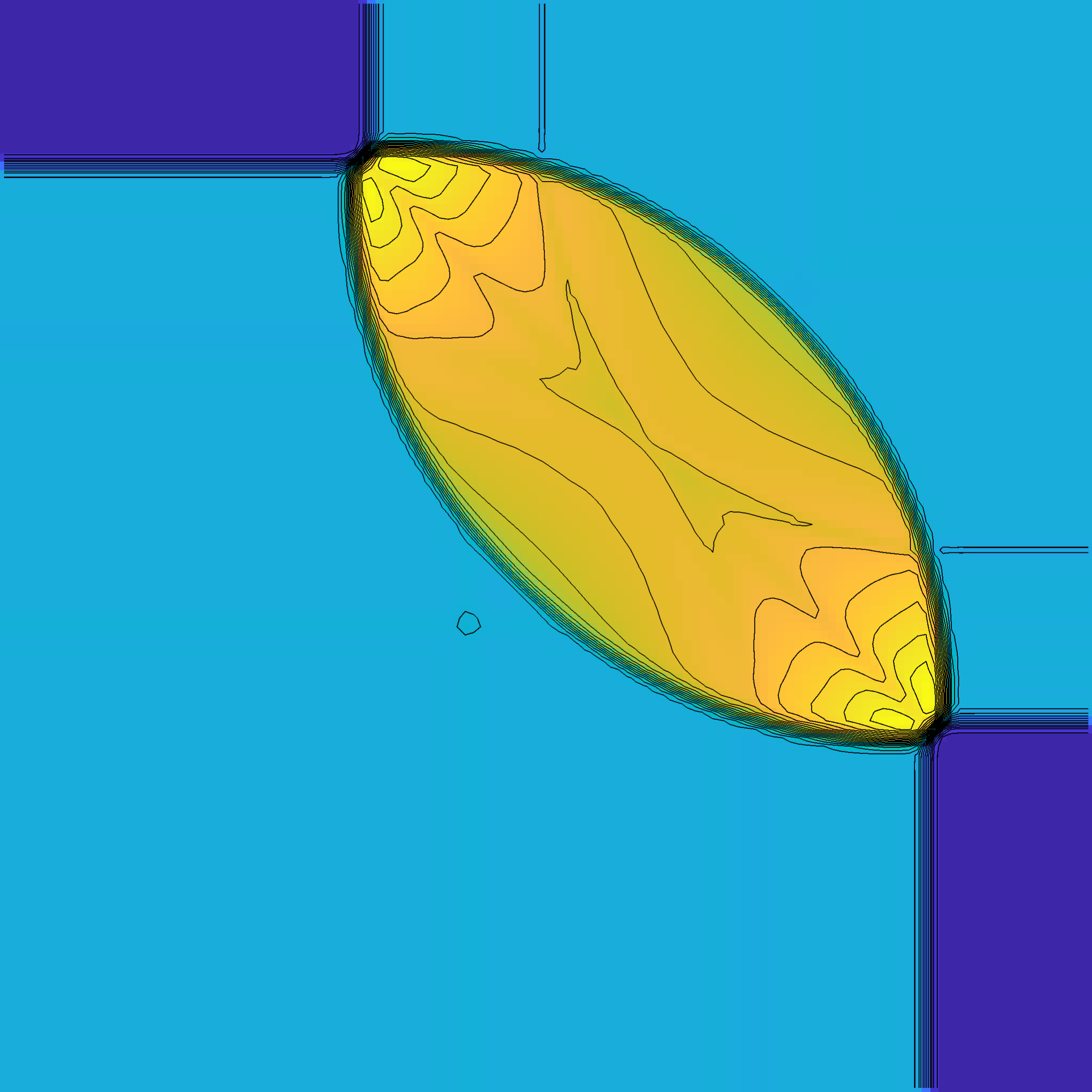}\hfill
\includegraphics[width=0.32\linewidth]{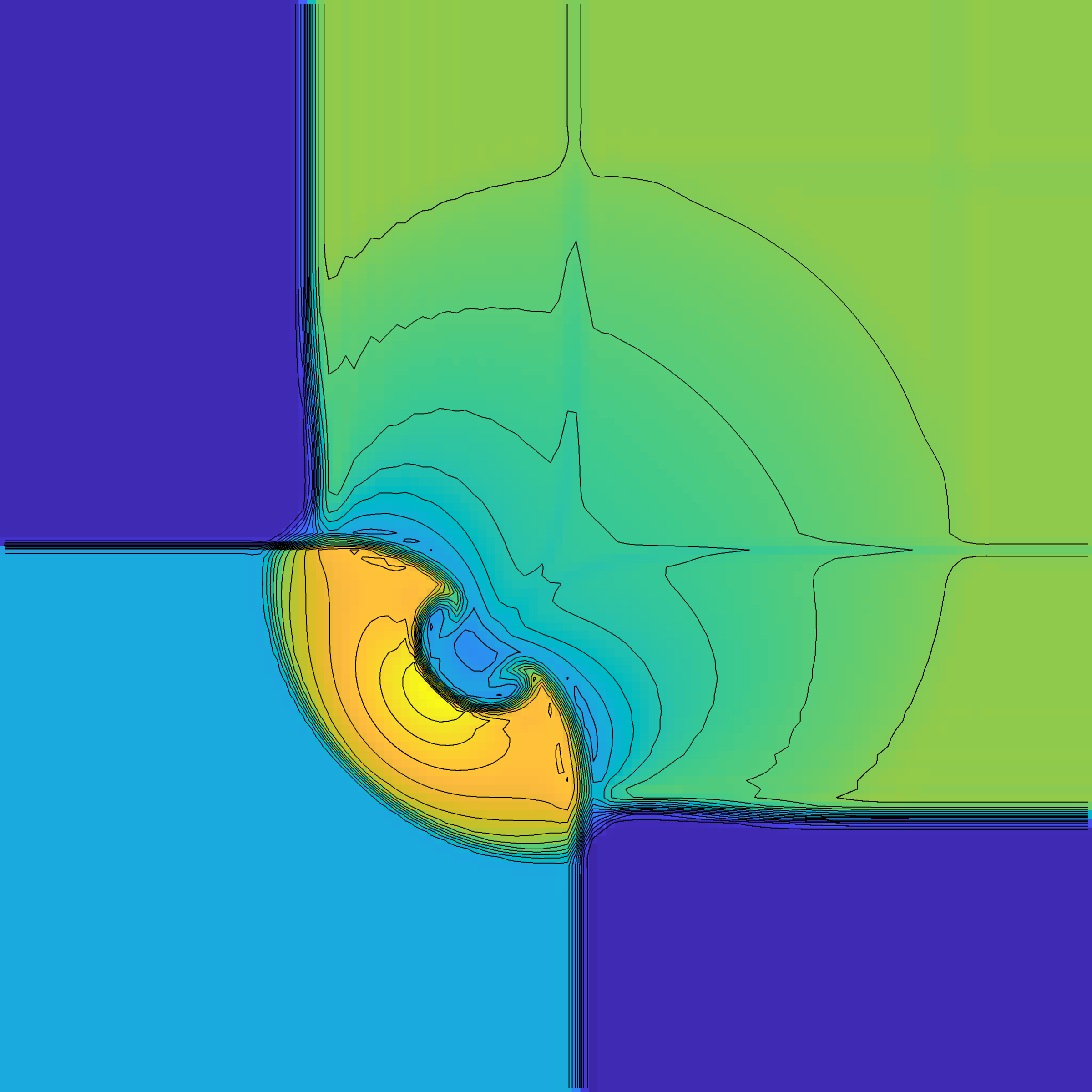}\hfill
\includegraphics[width=0.32\linewidth]{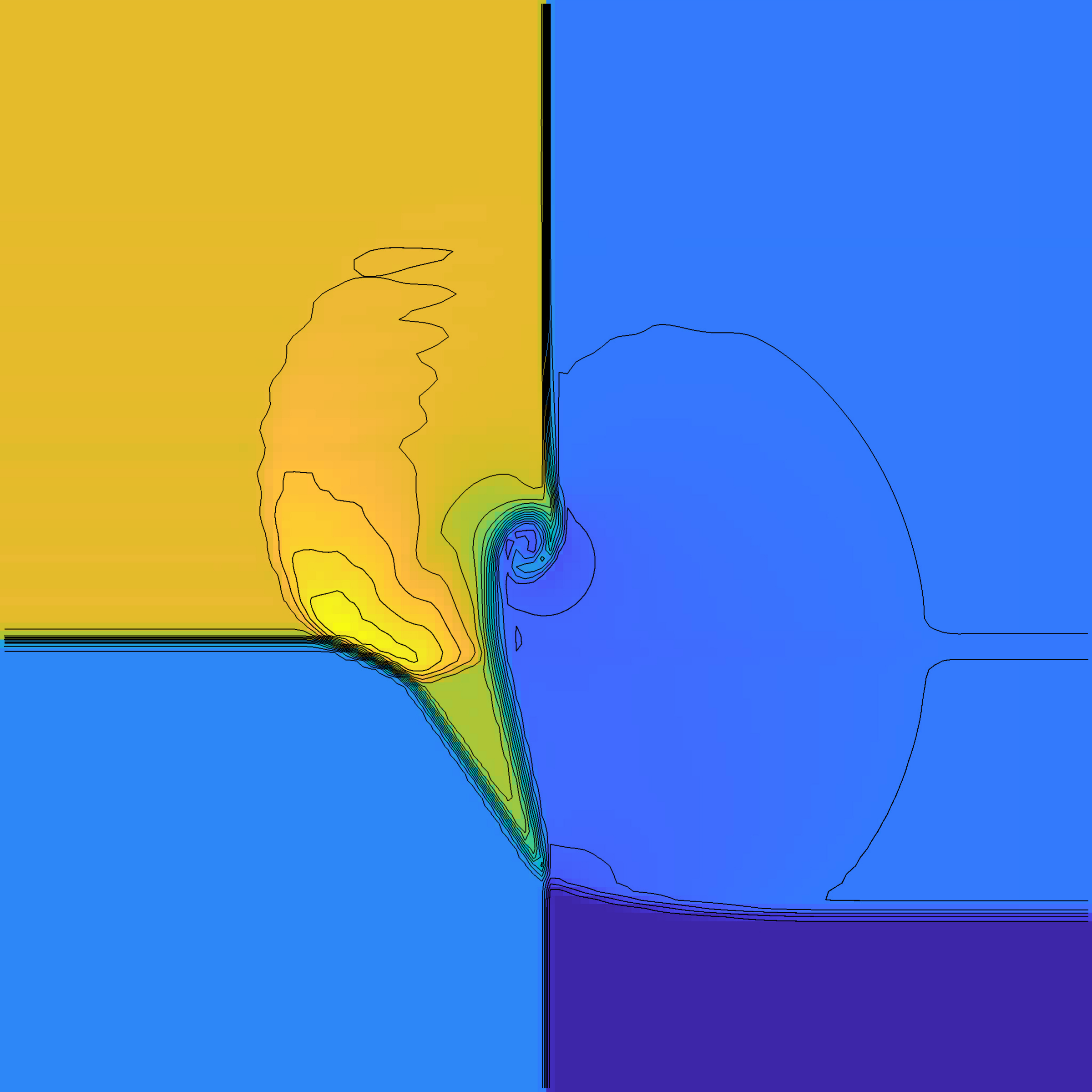}\\[0.15cm]
\includegraphics[width=0.32\linewidth]{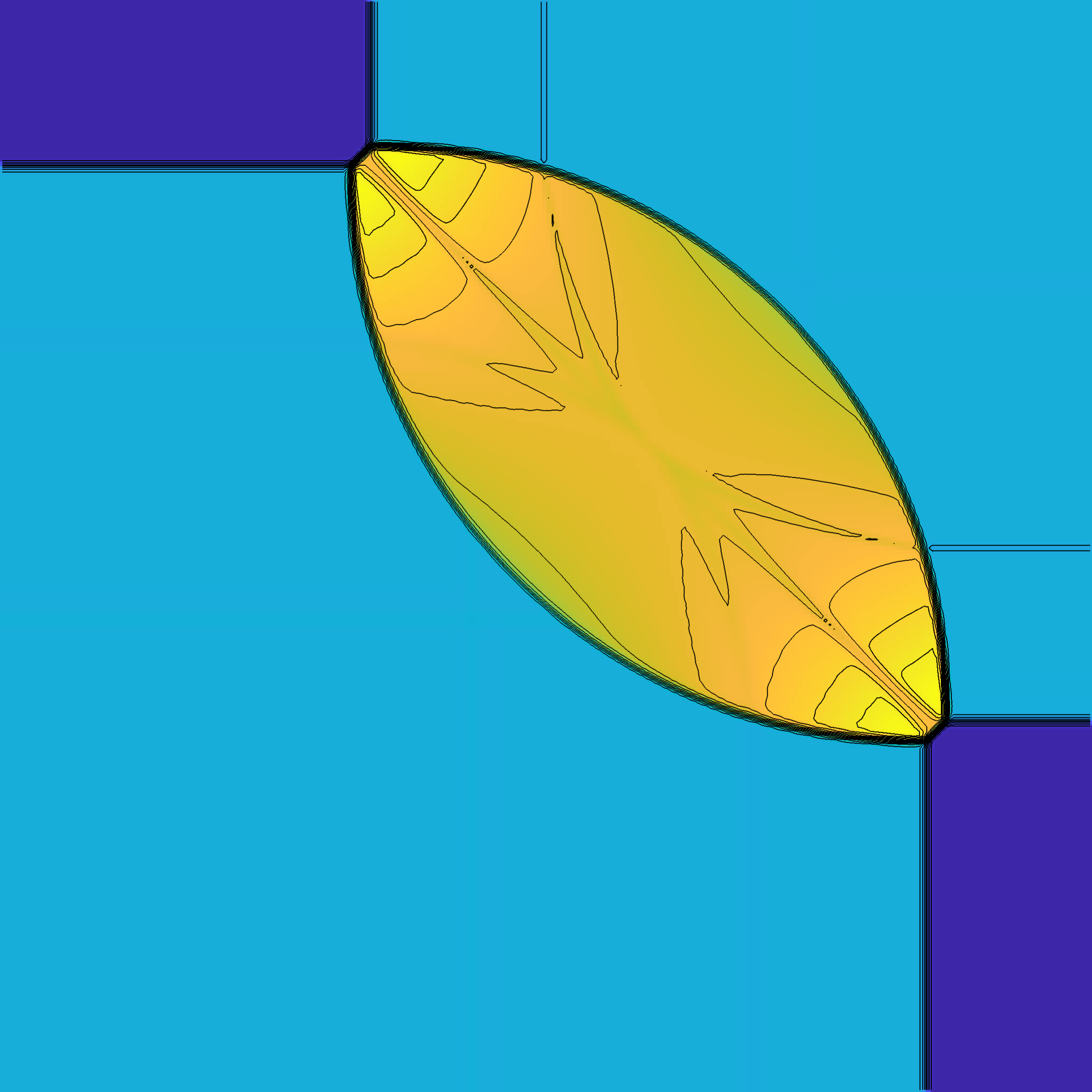}\hfill
\includegraphics[width=0.32\linewidth]{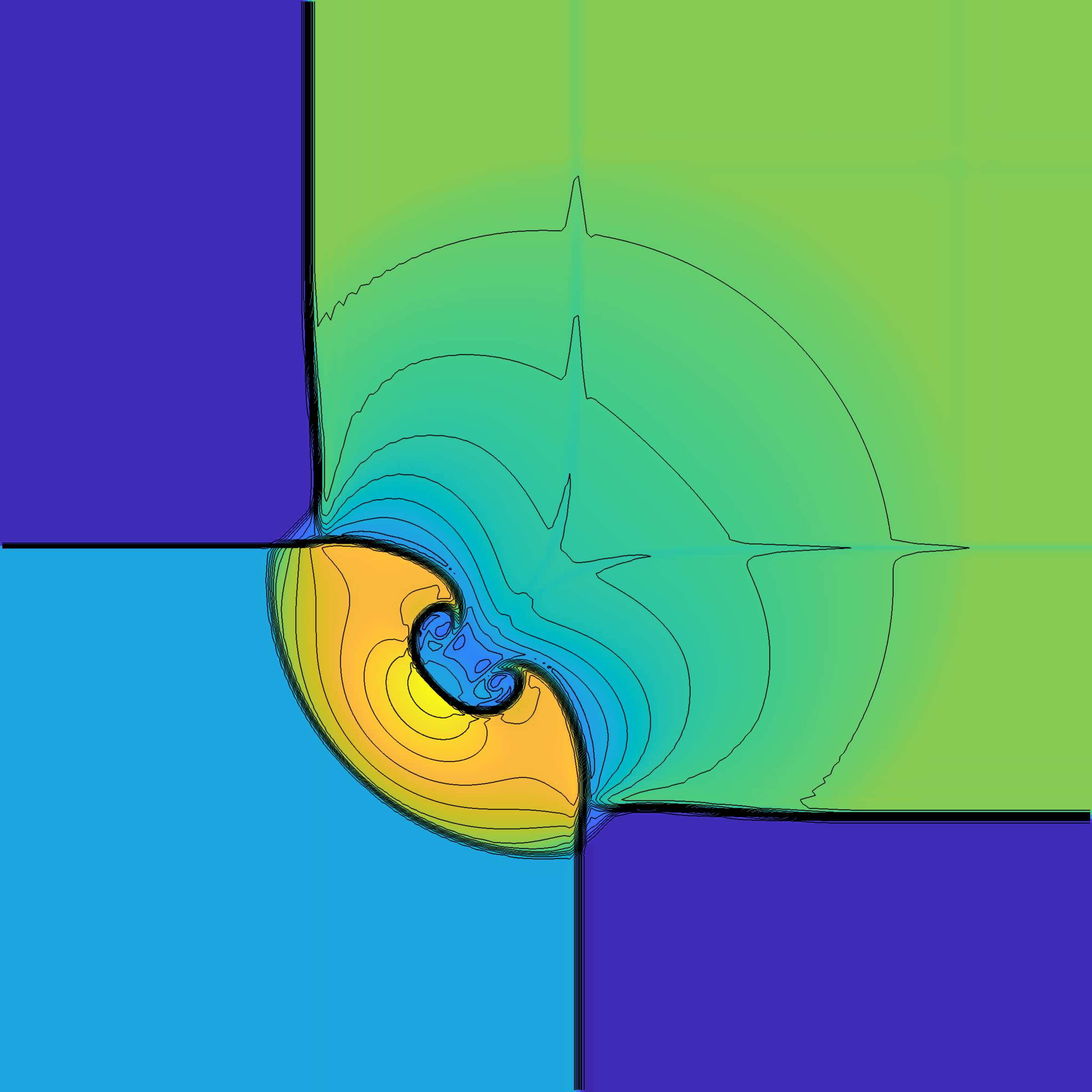}\hfill
\includegraphics[width=0.32\linewidth]{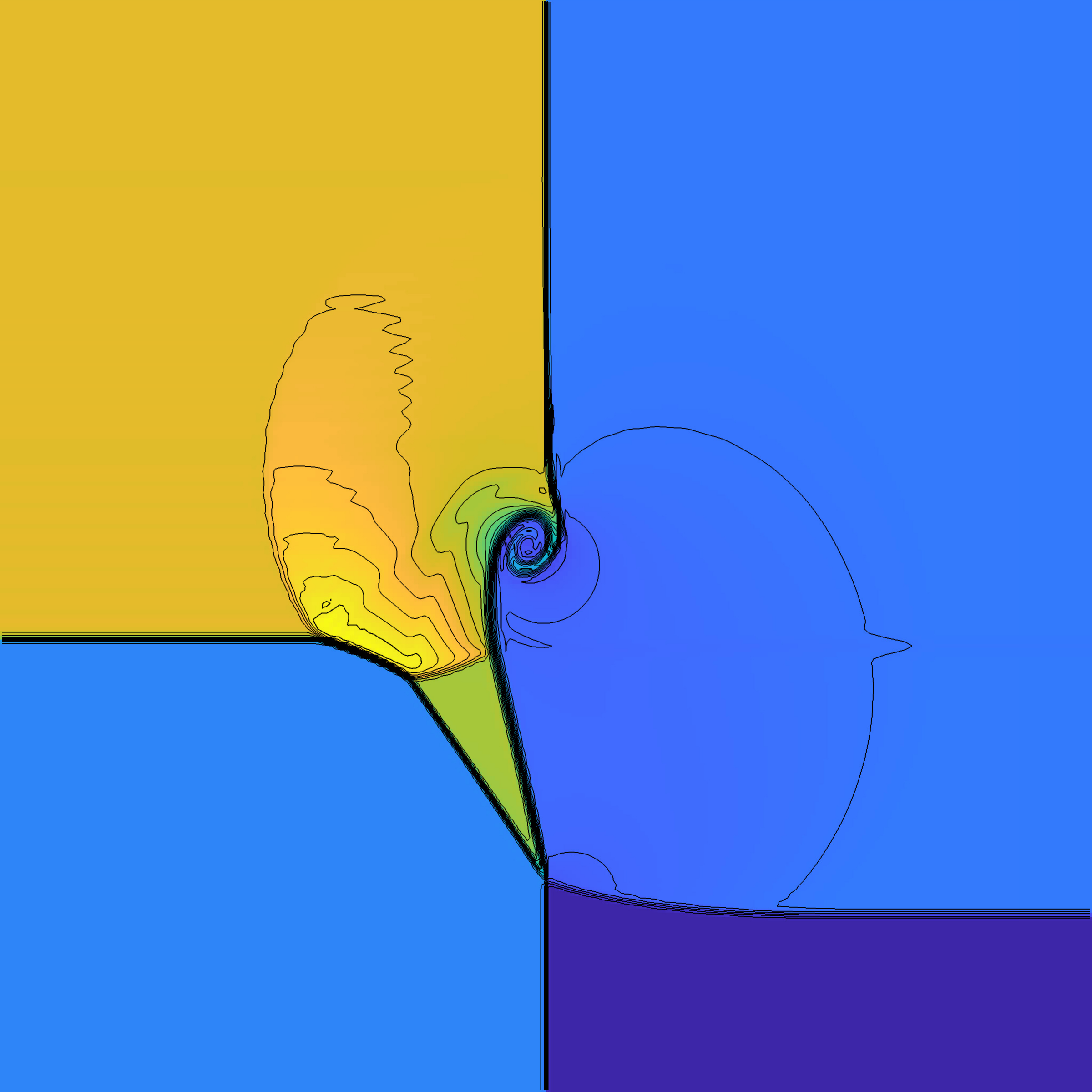}
\caption{Solution structure for Configuration 4 at time $t=0.21$
  (left), Configuration E at time $t=0.3$ (middle) and Configuration J
at time $t=0.3$ (right).  In the first two rows $128^2$  cells and  in the last row $256^2$ cells were used. In the first row only the bound preserving update was used. 
In the last two rows the shock indicated limiting of point values, as described in \cref{sec:shockIndicator}, was used in addition. \label{fig:RPs}}
\end{figure}
In the first row we show the solution on a grid with $128^2$
grid cells using only the bound preserving point value limiter where
needed.
Oscillations, caused by the shock waves, are clearly seen and disturb
the overall solution structure. In the second row we show results on
the same coarse grid but using in addition the limiting of point
values as described in \cref{sec:shockIndicator}. In the bottom row we
show results on a grid with $256^2$ grid cells using the same parameter free
limiting of point values. Note that Configuration J contains a
transonic shock wave between the left top and the left bottom
quadrants. The
modification described in \cref{alg:AFEuTransonicShocks} was needed to
approximate this problem.
Compared with other published results, see for
example \cite{article:KT2002,article:RCG1993},
we obtain accurate results on coarser grids.

Finally, we consider Configuration 3 from \cite{article:RCG1993},
which shows the most interesting solution structure. Several authors,
see for example
\cite{article:WWD2019}, used this problem to study the performance of high order
methods. The vortices along the contact discontinuities, which
develop due to a Kelvin-Helmholtz instability, are seen as indicator for the
quality of the numerical method.
In  \cref{fig:C3} (top), we show the density at time $t=0.8$
computed using  the Active Flux method
on grids with $256^2$ (left), $512^2$ (middle) and
$1024^2$ (right) cells.
This test computation used the limiters described in 
\cref{sec:boundPreservingPointValues,sec:shockIndicator}. 
Furthermore, the correct treatment of transonic shock waves, as
described in \cref{alg:AFEuTransonicShocks}, was needed, since all
four of the initial configuration's Riemann problems lead to transonic shock waves.
\begin{figure}	
\includegraphics[width=0.32\linewidth]{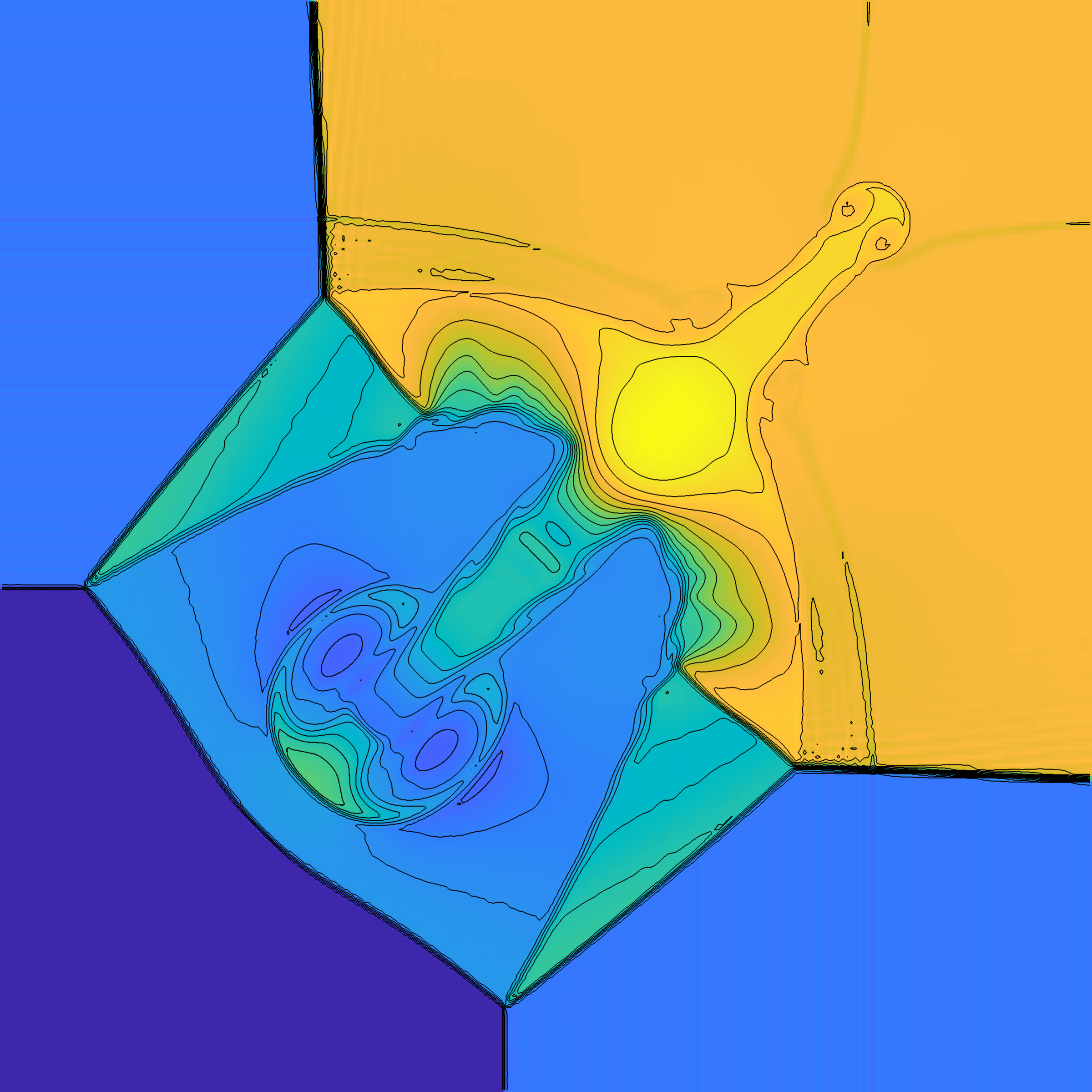}\hfill
\includegraphics[width=0.32\linewidth]{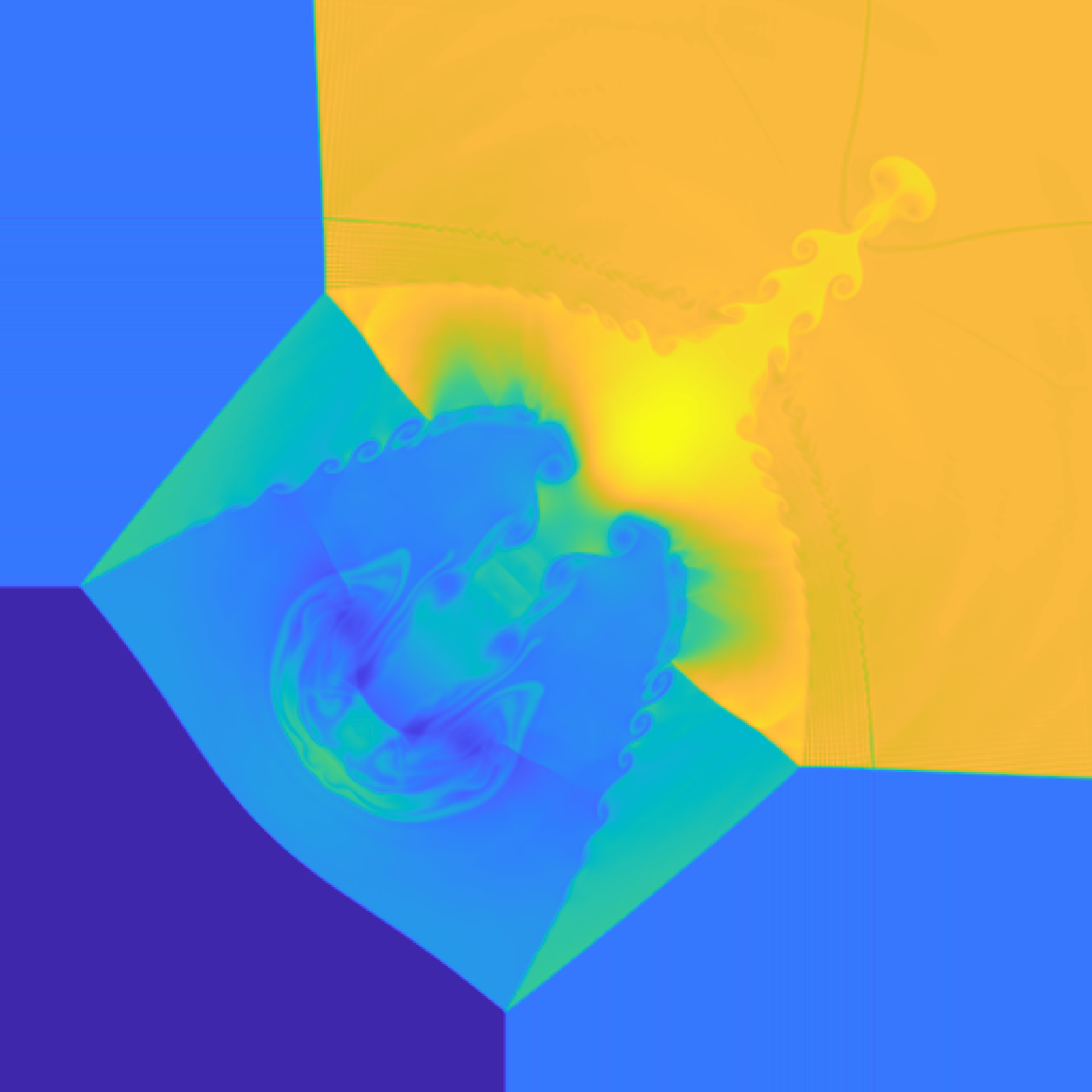}\hfill
\includegraphics[width=0.32\linewidth]{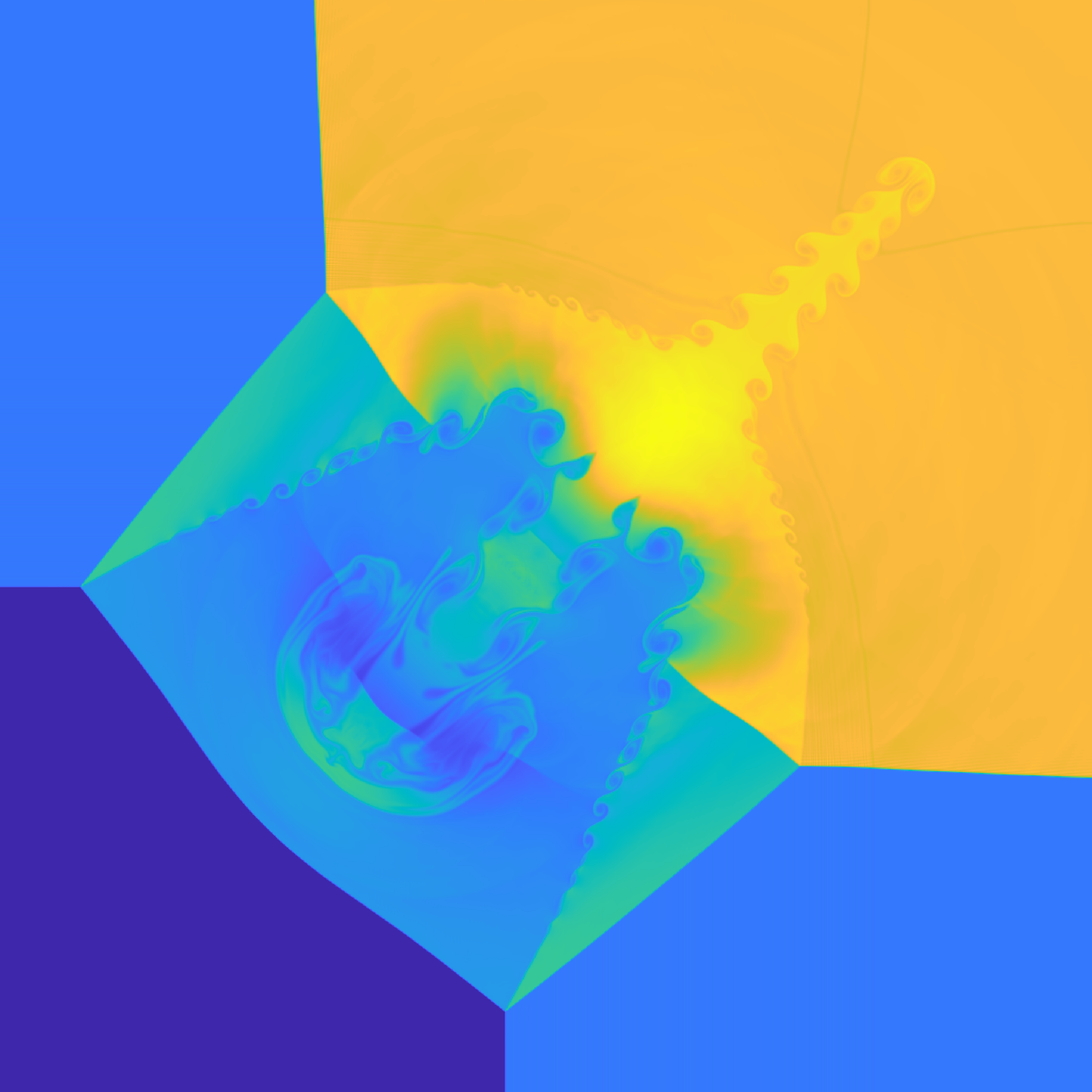} \\[0.15cm]
\includegraphics[width=0.32\linewidth]{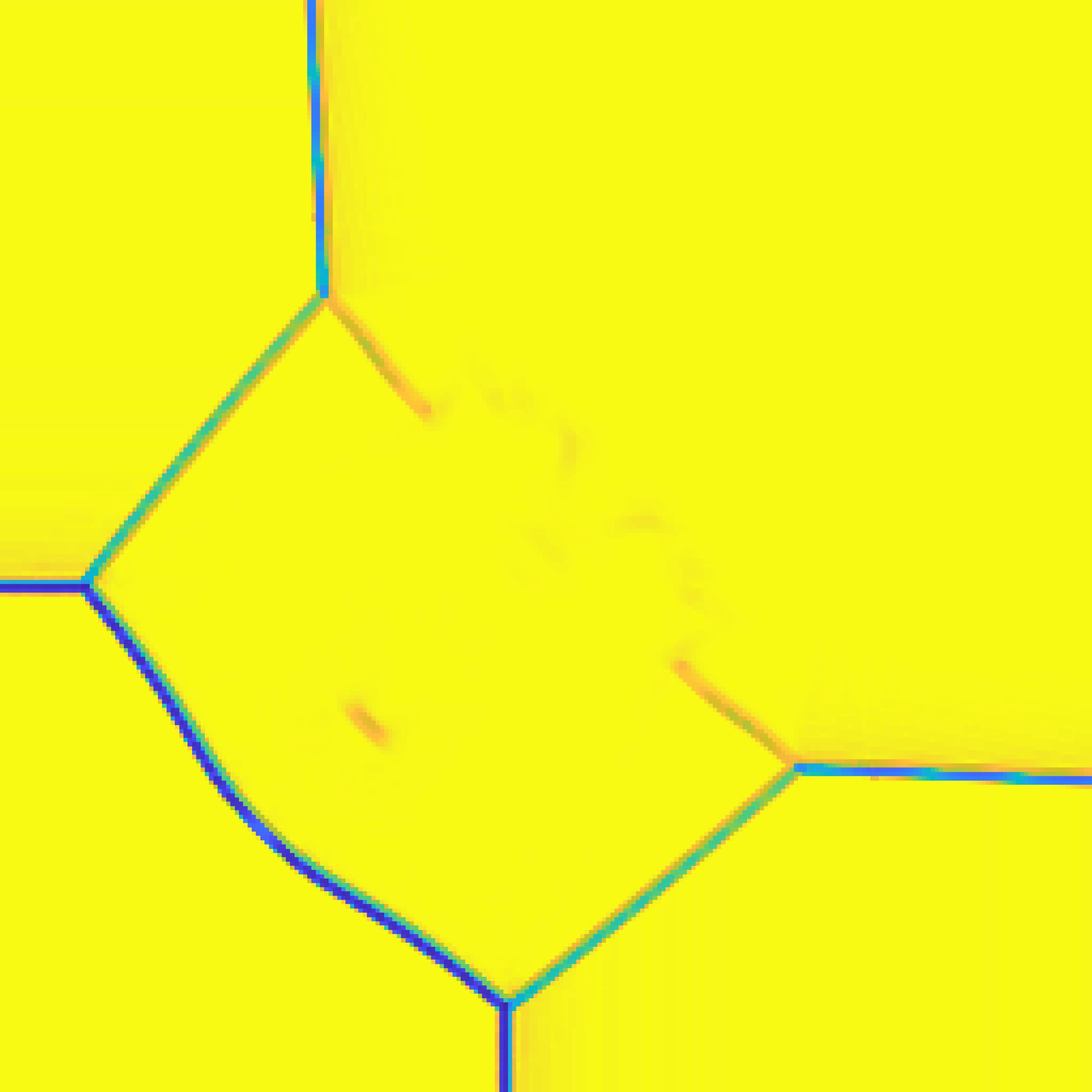}\hfill
\includegraphics[width=0.32\linewidth]{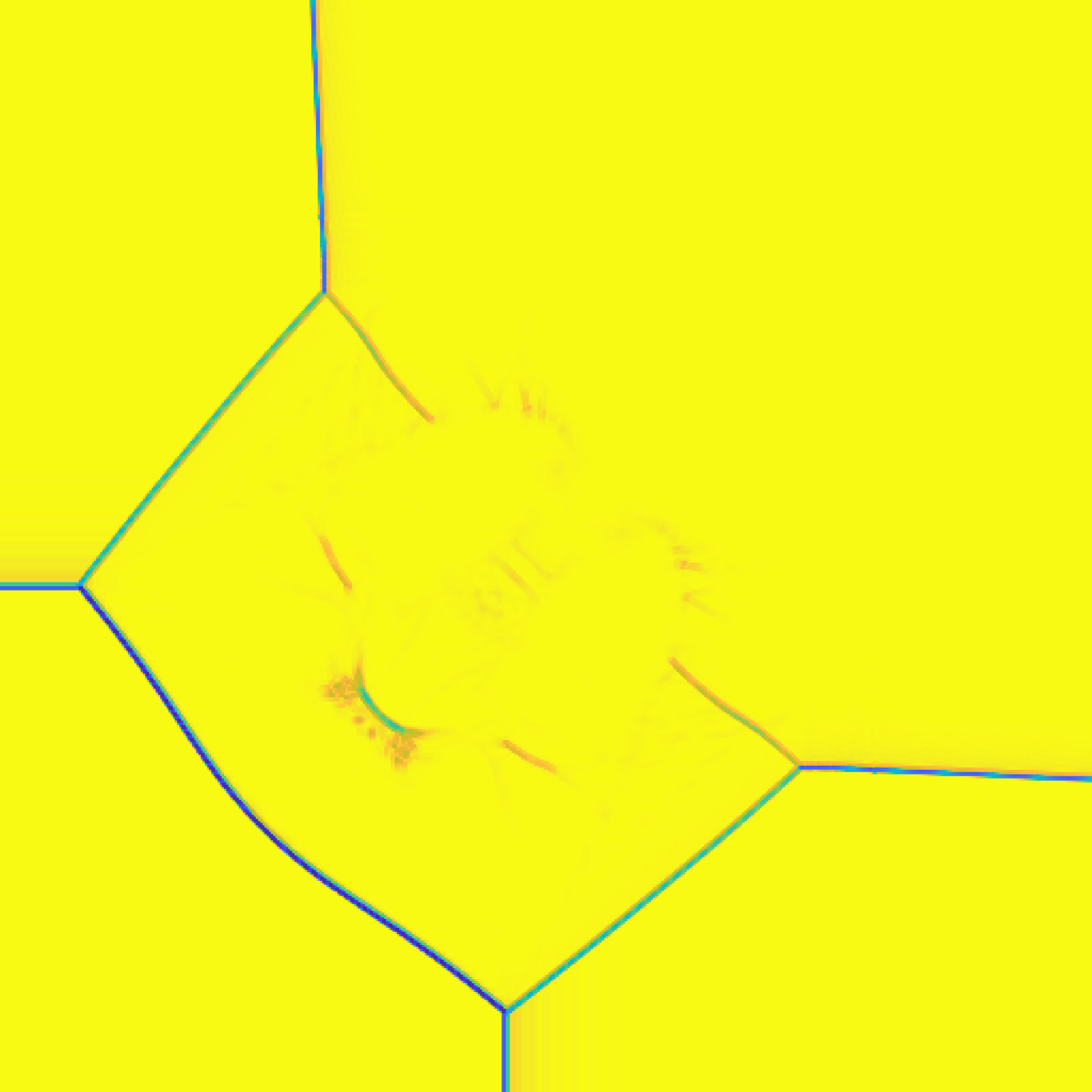}\hfill
\includegraphics[width=0.32\linewidth]{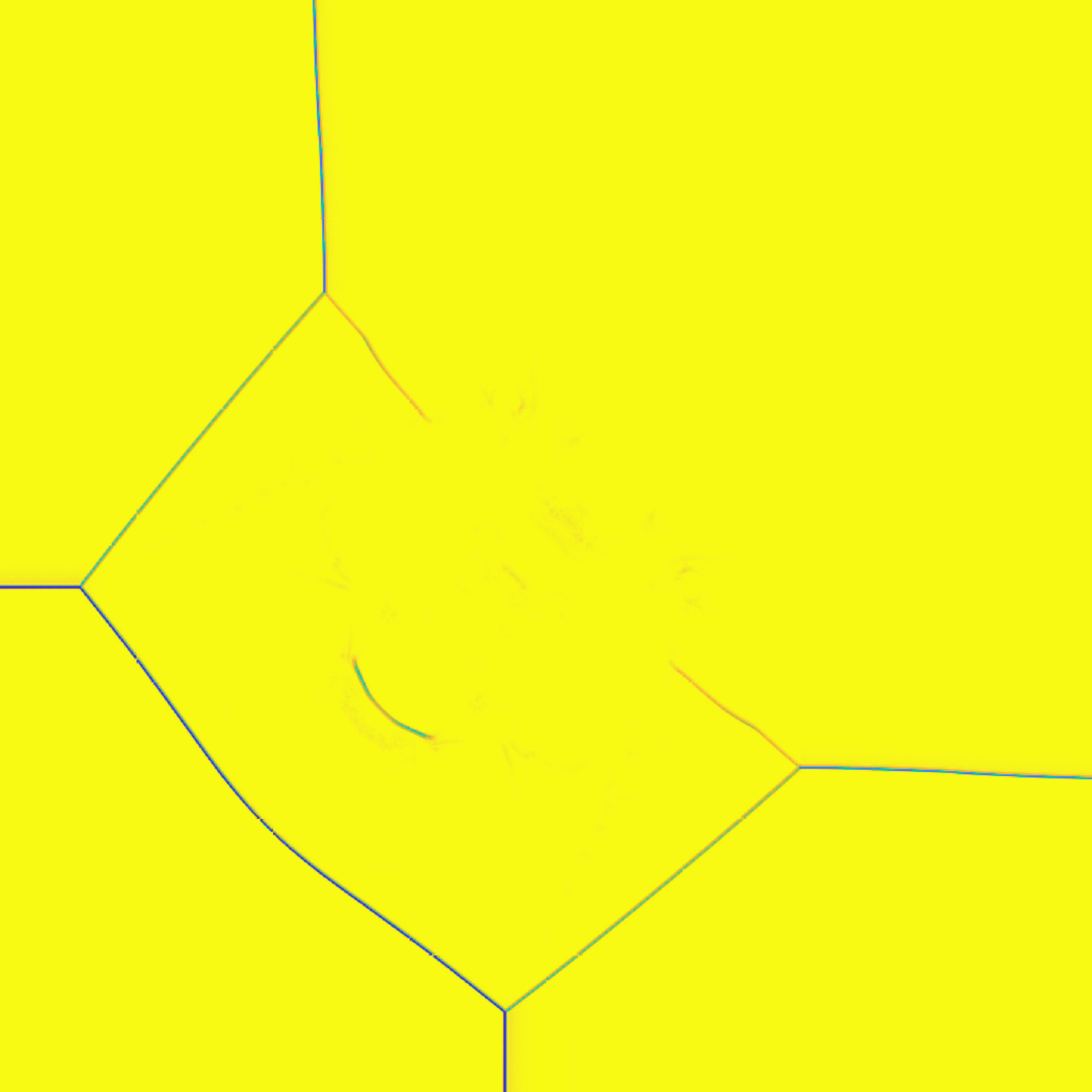}
\caption{\label{fig:C3} Numerical results for
    Configuration 3 on grids with $256 \times 256 $ (left), $512 \times
    512$ (middle) and $1024 \times 1024$ (right) grid cells. Plot of density at time $t=0.8$ (top), 
    corresponding shock indicator $\theta$ (bottom).}
\end{figure}
Even on the $256^2$ grid the vortices are visible and
we recover more structure than, for example,  the fifth order WENO method on a grid
with $400^2$ cells used in
\cite{article:WWD2019}. The second row shows the
  values of the shock indicator $\theta$ at the grid cell vertices 
  computed at the final time $t=0.8$.

\subsection{Kelvin-Helmholtz instability}
Now we consider approximations of the Kelvin-Helmholtz instability
test problem from \cite{article:SanKara}.
\begin{example}\label{ex:KH}
We consider the two-dimensional Euler equations with initial values of
the form
\begin{equation*}
\left(\rho_0,u_0,v_0,p_0\right)  = 
  \left\{ \begin{array}{lcl}
            (1,0.5,10^{-2} \sin(2 \pi x),2.5) & : & y>0.25,\\
            (2,-0.5,10^{-2}\sin(2 \pi x),2.5) & : & -0.25 \le y \le
                                                    0.25,\\
            (1,0.5,10^{-2}\sin(2\pi x),2.5) & : & y<-0.25\end{array}\right.
\end{equation*}
on the domain $[-0.5,0.5]\times [-0.5,0.5]$ with double periodic
boundary conditions.       
\end{example}
The one-dimensional, single mode perturbation imposed in the
$y$-directions evolves into a two-dimensional turbulent
structure, evoking  small-scale vortical structures.
These structures can be observed in numerical simulations
of high-order numerical methods and by refining the grid.
San and Kara \cite{article:SanKara} provide
comparisons of several WENO methods of order 3, 5 and 7 using different
numerical fluxes and different grid resolutions. In \cref{fig:KH} we
show numerical results obtained with our Active Flux method
without and with limiting of point values. As expected, the two
versions lead to different approximations. For both methods, the
Kelvin-Helmholtz instability is clearly seen even on the very coarse
grid.
\begin{figure}	
	\includegraphics[width=0.32\linewidth]{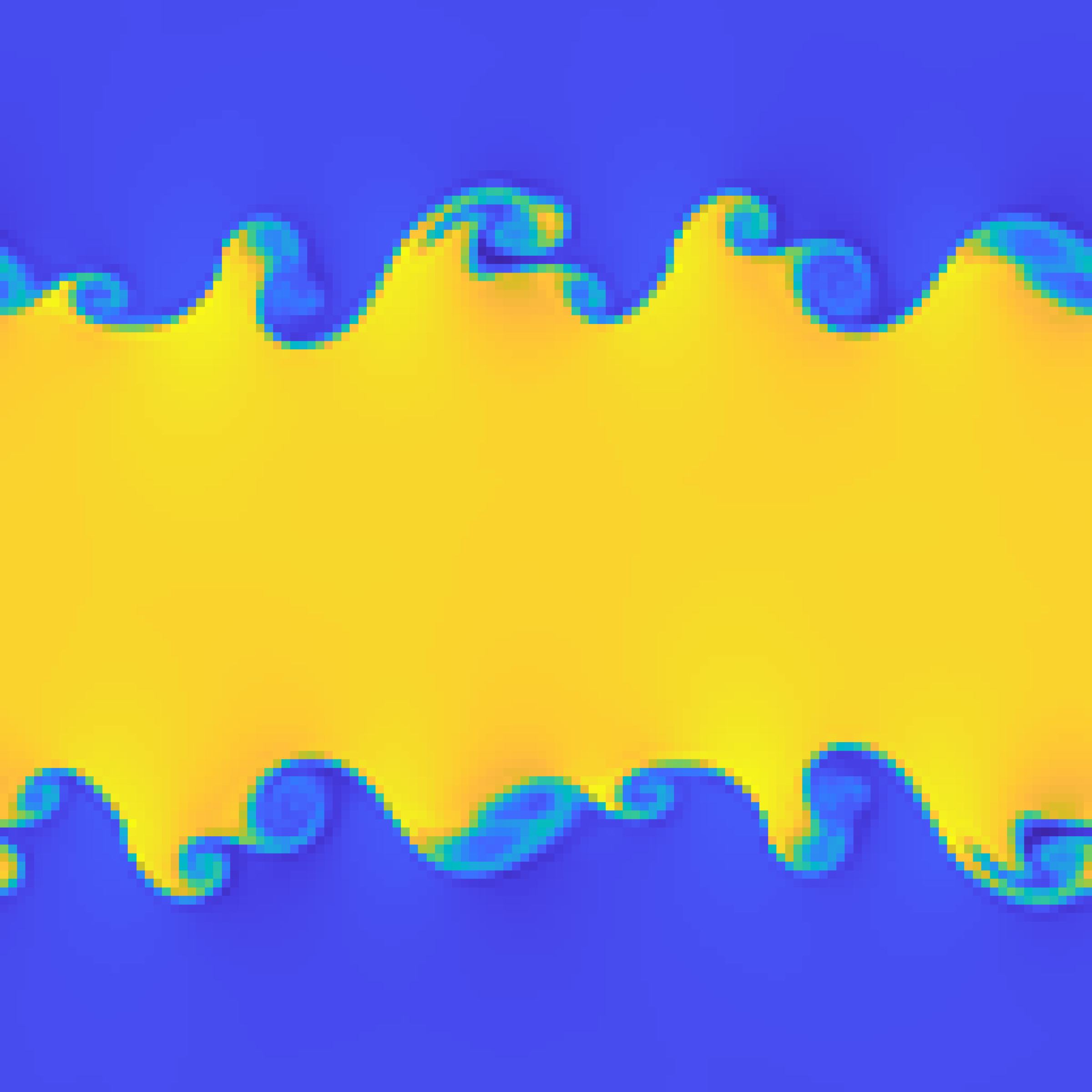}\hfill
	\includegraphics[width=0.32\linewidth]{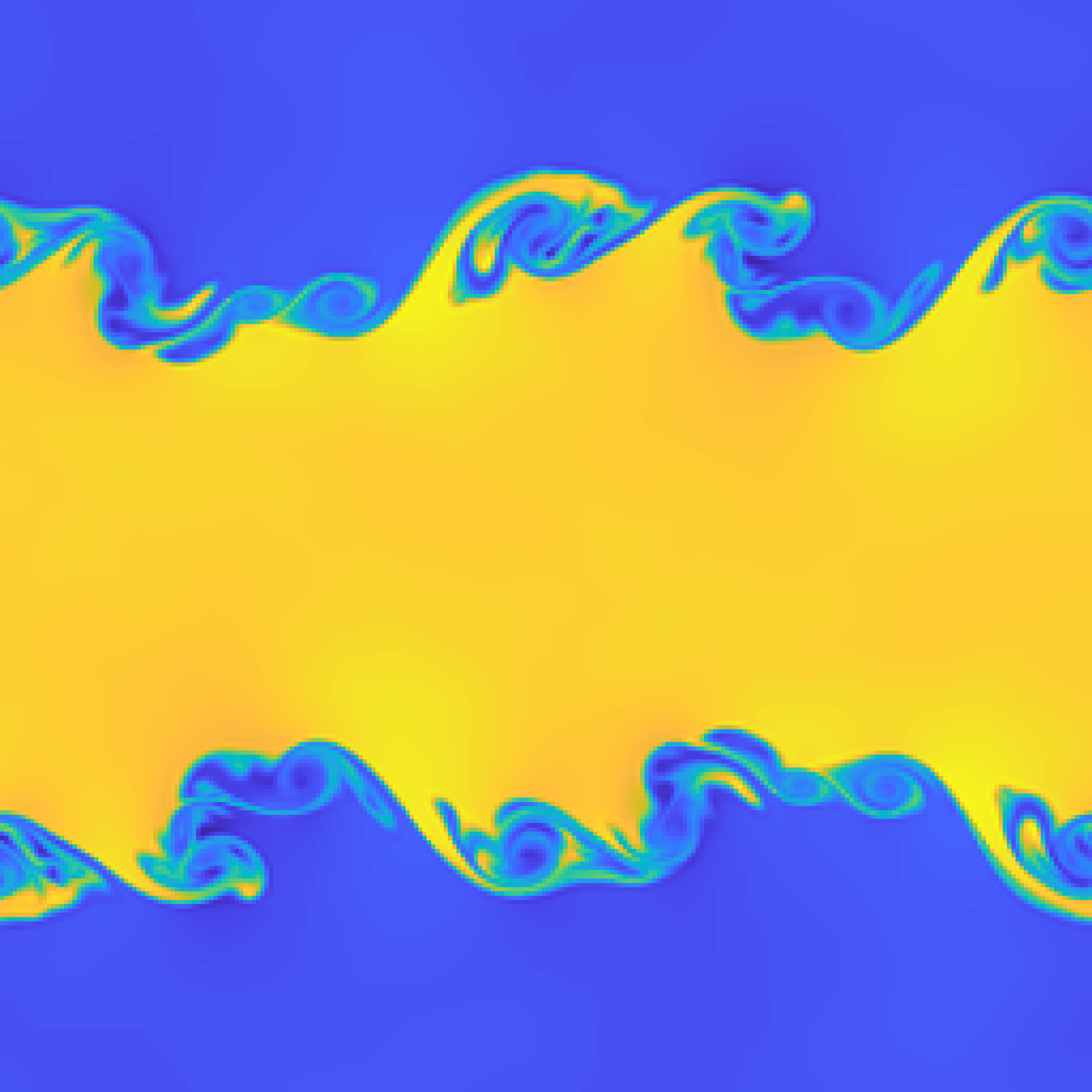}\hfill
	\includegraphics[width=0.32\linewidth]{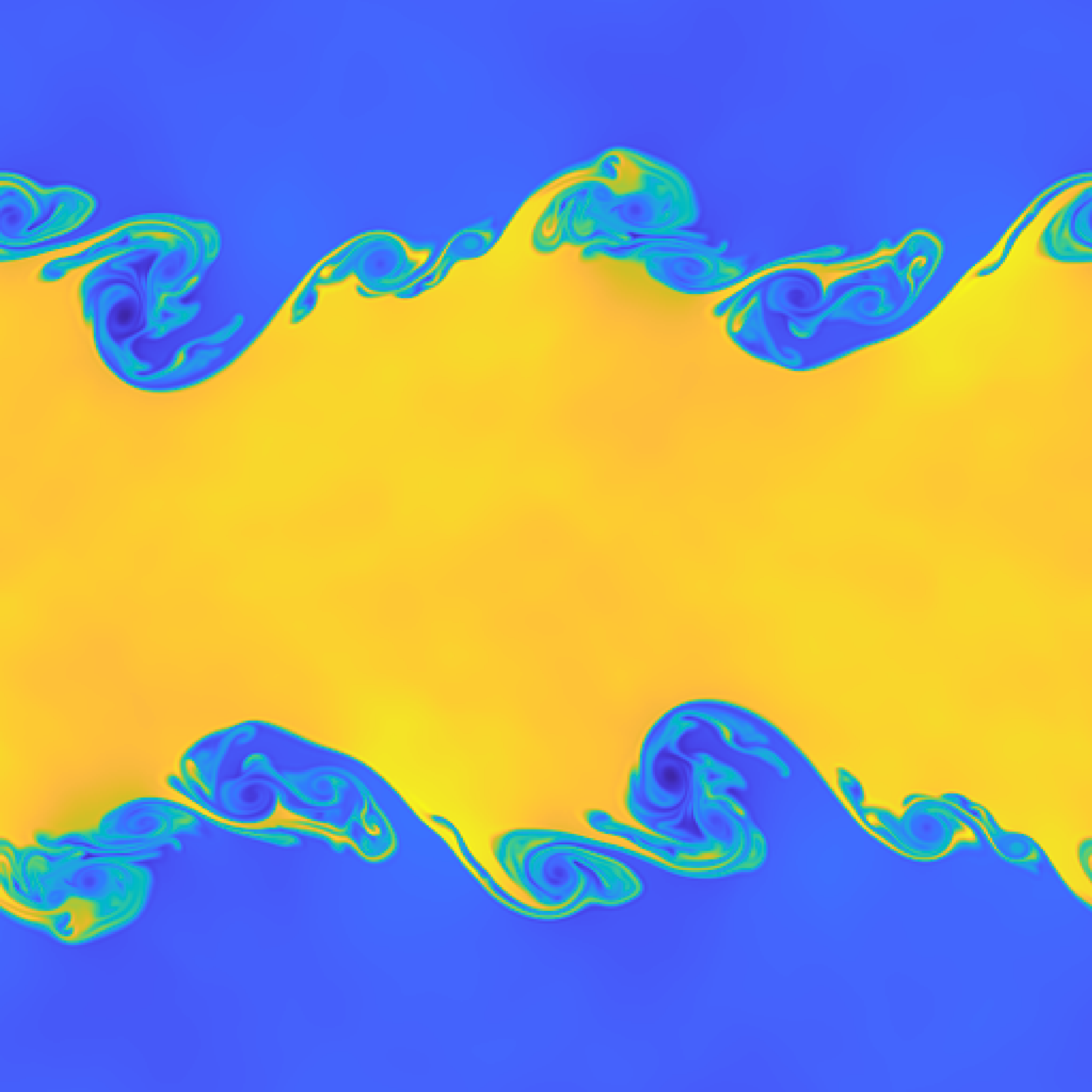} \\[0.15cm]
		\includegraphics[width=0.32\linewidth]{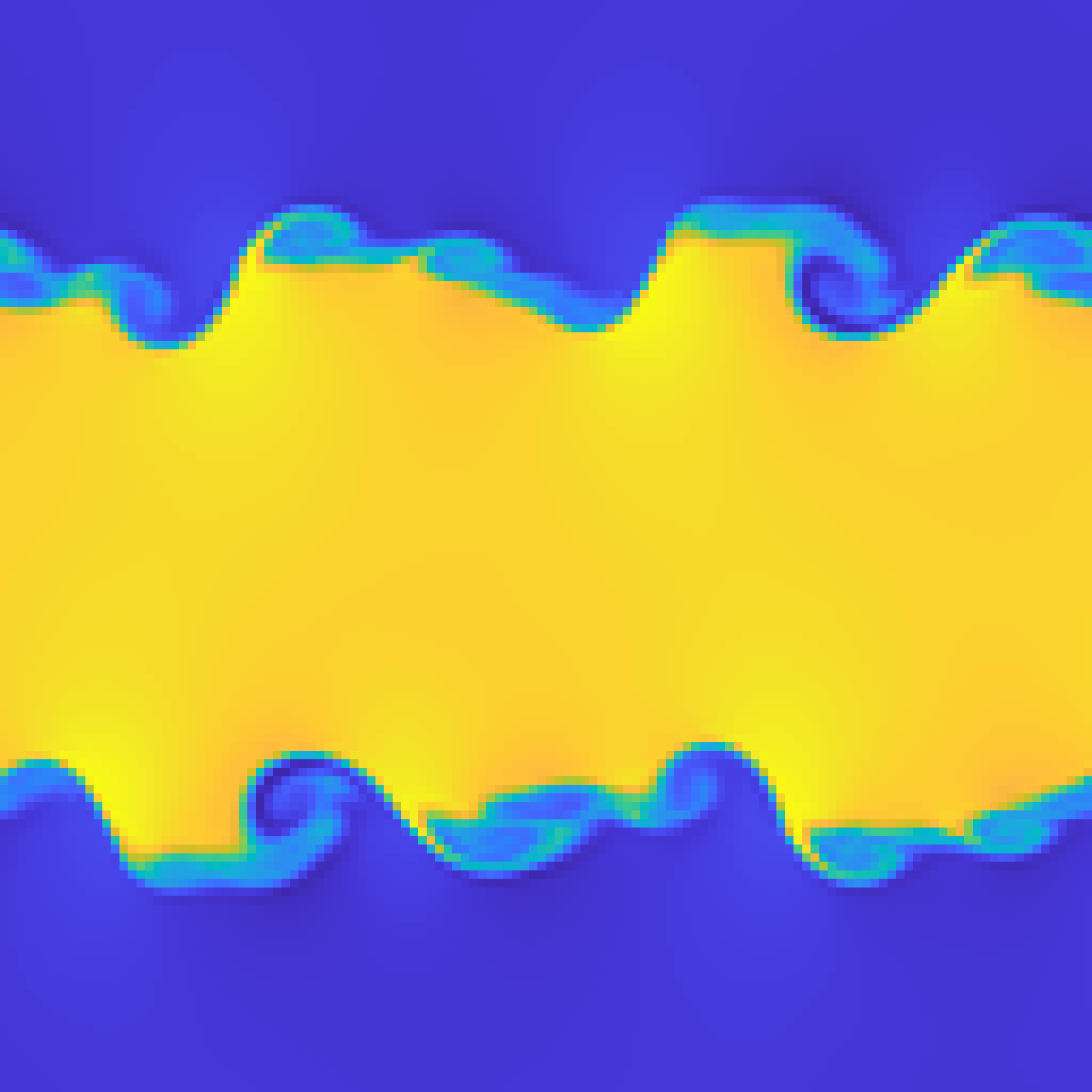}\hfill
	\includegraphics[width=0.32\linewidth]{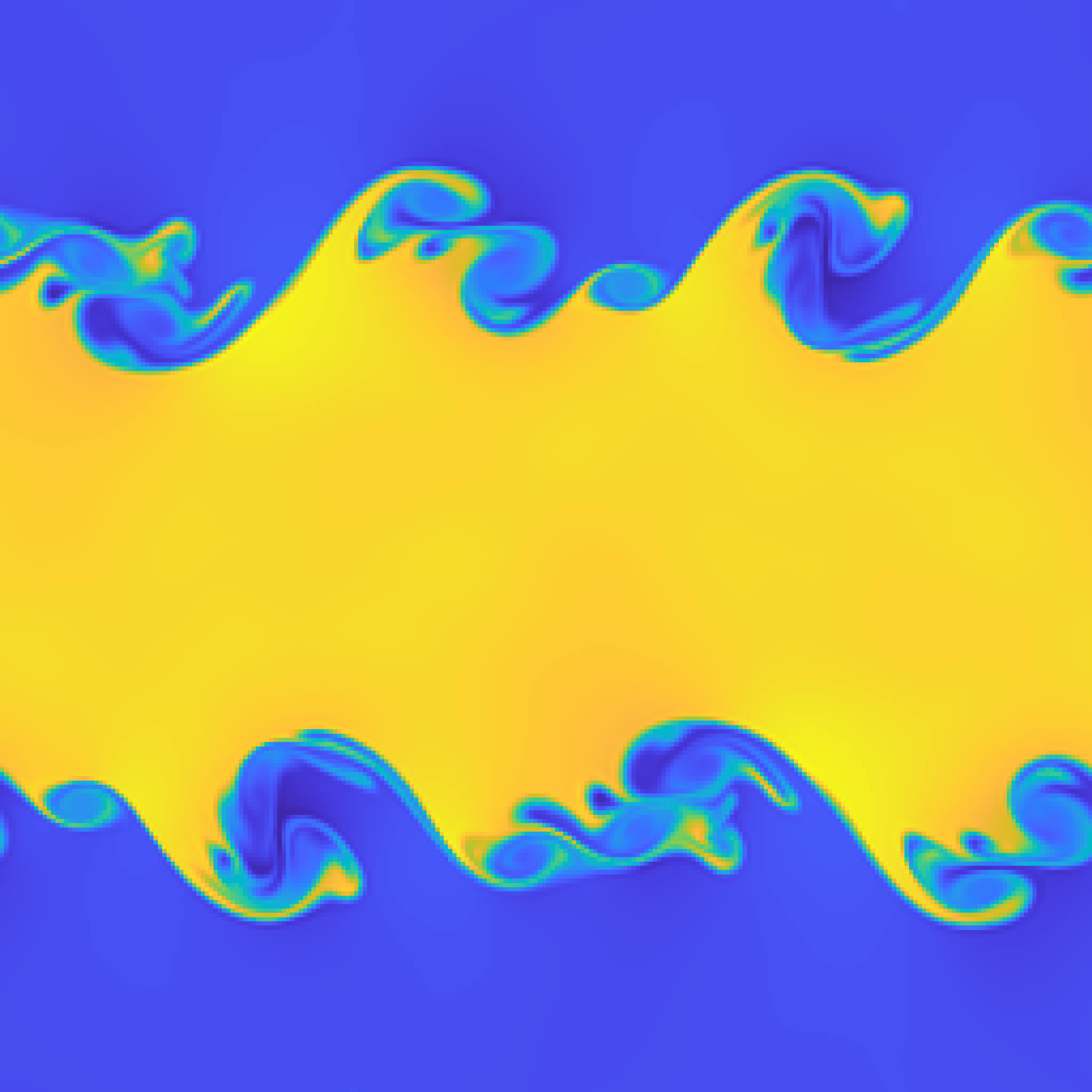}\hfill
	\includegraphics[width=0.32\linewidth]{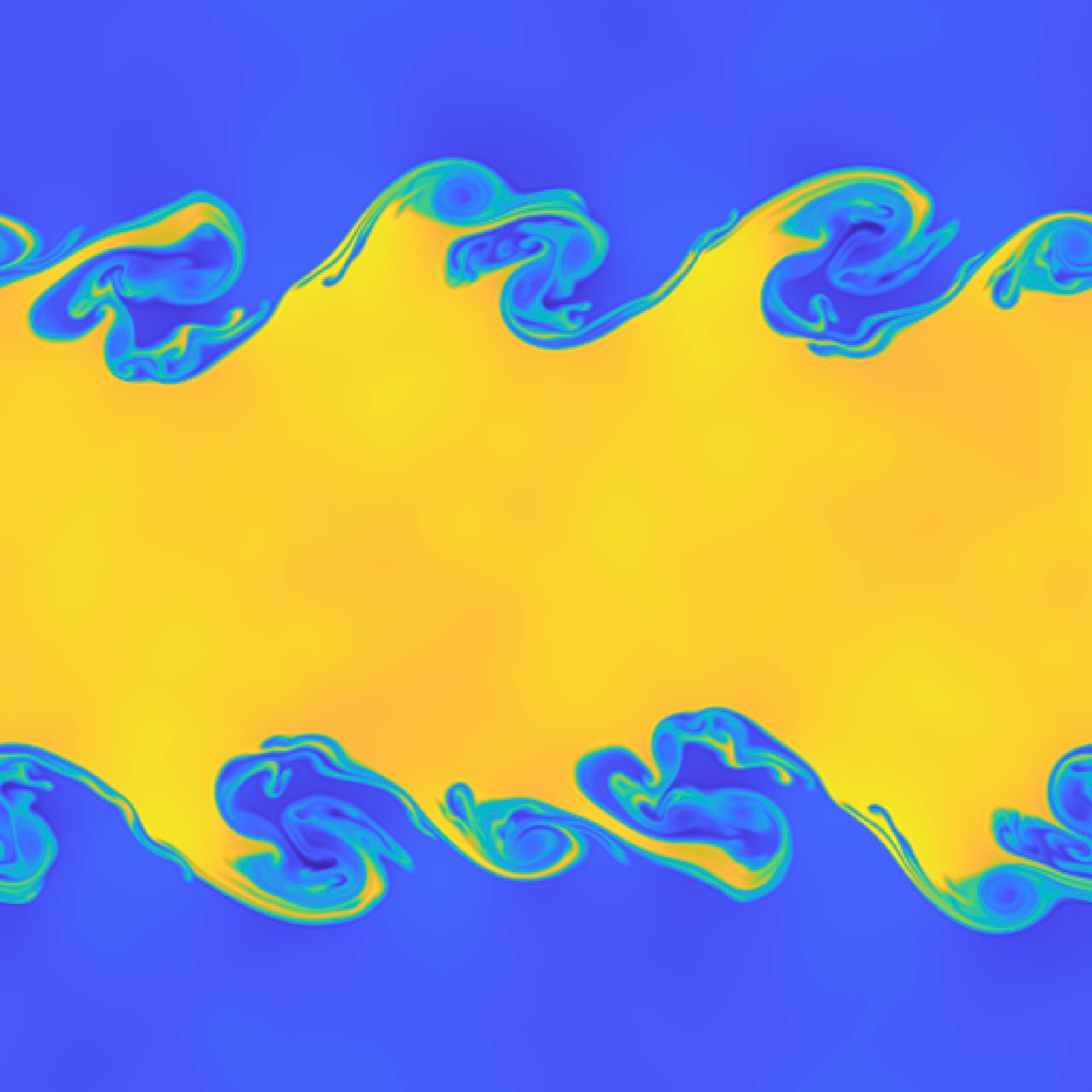}
\caption{\label{fig:KH}Approximations of the Kelvin-Helmholtz
  instability at time $t=1$ using the
  unlimited (top) and the limited (bottom) third order accurate Active
  Flux method on grids with $128^2$ (left), $256^2$ (middle) and
  $512^2$ (right) grid cells.}
\end{figure}
The observed structures compare with
structures seen in WENO methods of higher order and on finer grids
showing again that the fully discrete Active Flux method performs
well even on coarse grids.

\subsection{Implementation of boundary conditions}

Samani and Roe \cite{article:SR2023} pointed out that fully discrete
Active Flux methods allow an accurate implementation of boundary
conditions for acoustics.
So far our computations used either periodic boundary
conditions or outflow conditions, which were implemented using a
straight forward piecewise constant extrapolation of the data to ghost cells as in Figure \ref{fig:bc} (right).
In our final test problem,
we need in addition inflow as well as reflecting boundary conditions.

\begin{figure}	
	\includegraphics[width=1.\linewidth]{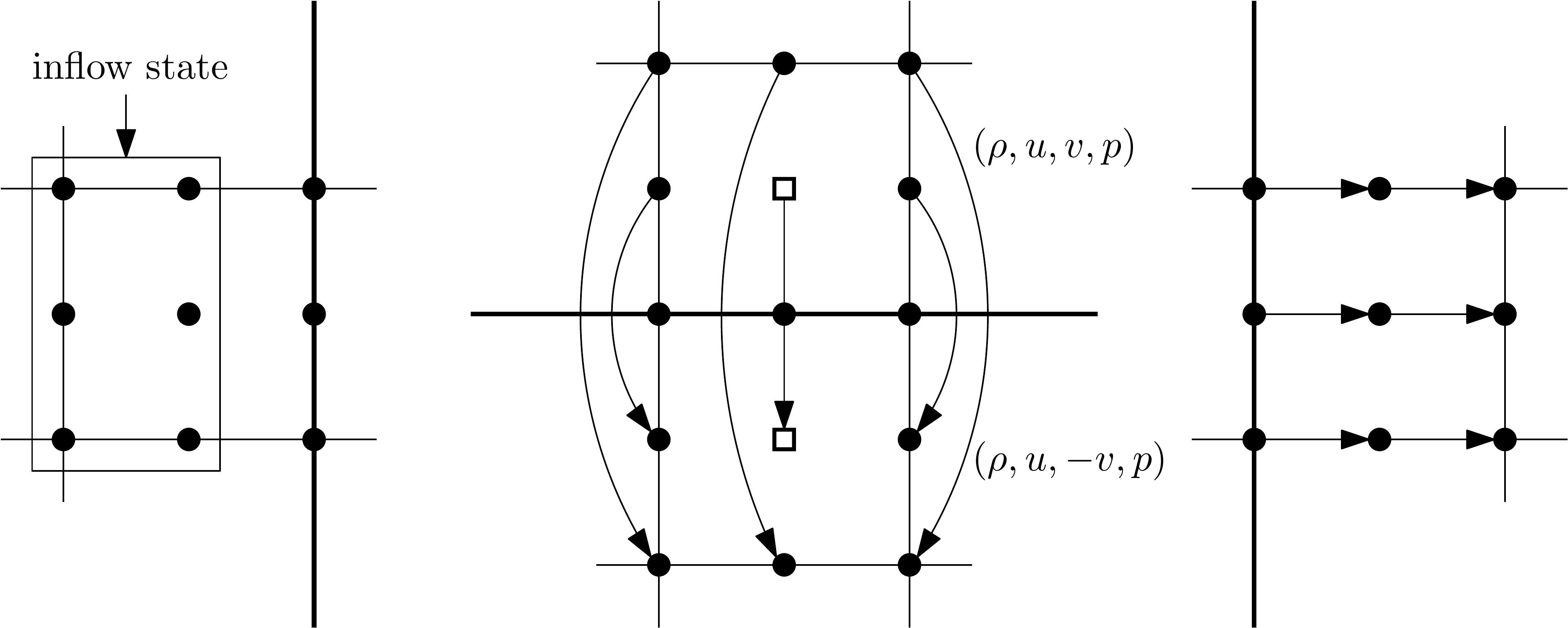}\hfill
	\caption{\label{fig:bc} Ghost cell filling for inflow (left), solid wall (middle) and outflow boundary conditions (right). }
\end{figure}

At an inflow boundary we define
point values in ghost cells and compute cell averages during each time step.
The boundary flux can then be computed by using the same
method at the boundary as inside the domain, Figure \ref{fig:bc} (left).

At a solid wall boundary we define cell
average values in ghost cells by copying the primitive variables from
the first grid cell to the ghost cell and negating the normal
velocity. Point values in ghost cells are defined by reflecting
the corresponding point values in direction normal to the boundary as
illustrated in  Figure \ref{fig:bc} (middle).
At a reflecting boundary simple averaging of neighboring grid
  cell values may no longer  provide an accurate enough linearisation
  that is needed
  by the evolution operator. Instead a linearisation around the point values at
  the boundary should be used.

\begin{example}
We now consider the shock reflection problem for the two-\\ dimensional Euler equations in the domain $[0,4]\times [0,1]$.
The initial values are 
\begin{equation*}
(\rho,u,v,p) = 	
\left\{ \begin{array}{lcl} 
(1.69997,2.61934,-0.50632,1.52819) & : & y=1,\\	
(1,2.9,0,1/1.4) & : & \text{otherwise}.
\end{array}\right.
\end{equation*}
These
values are also used as inflow conditions at the left and top
boundary. At the bottom a reflecting boundary is imposed and on the right boundary outflow conditions. 
\end{example}
Numerical results for the density at time $t=6$, using grids with
$120\times 30$ and $240 \times 60$ cells, are shown in Figure \ref{fig:SRP}. The computation was performed without any limiter.
\begin{figure}	
	\includegraphics[width=0.49\linewidth]{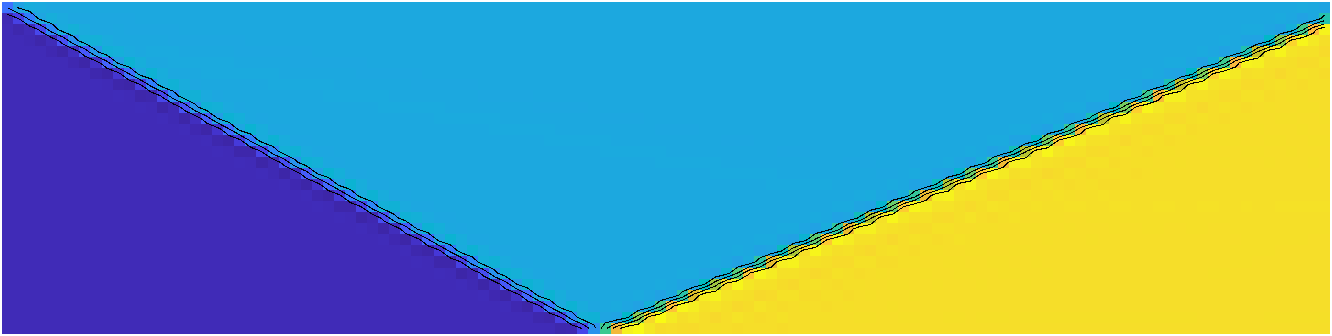}\hfill
	\includegraphics[width=0.49\linewidth]{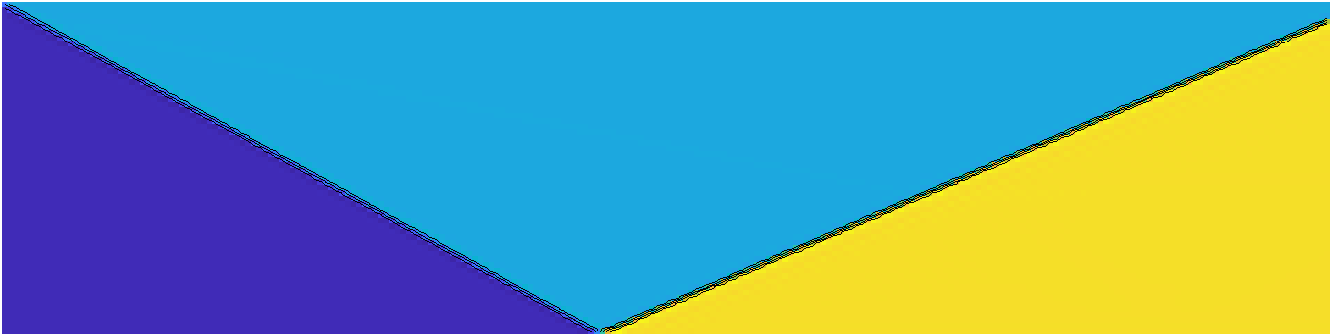}\hfill
	\caption{\label{fig:SRP} Approximations of the shock reflection problem at time $t=6$ with $120\times 30$ (left) and $240 \times 60$ cells (right). }
\end{figure}

\section*{Conclusions}
We introduced new Active Flux methods for the two-dimensional Euler
equations on Cartesian grids. Our method uses the conservative form of
the equations for the update of the cell averages of the conserved
quantities and a local
linearisation of the characteristic form to evolve the point
values using the method of bicharacteristics.
Third order accuracy can be obtained by using a correction term which
eliminates the linearisation error.

We showed that the quality of the solution depends on the choice of
the local
linearisation, in particular when approximating
transonic shock waves or transonic rarefaction waves.

A parameter free limiting  was introduced which guarantees that
the point values have positive density and positive pressure. 
Our numerical simulations confirm that the carefully designed fully
discrete me\-thod provides accurate results even on relatively coarse grids.
For our test problems no additional flux limiter was needed. However,
there are approaches available in the recent literature which
can be combined with our fully discrete approach
to guarantee that density and pressure computed from the cell average
values remain positive.
\bibliographystyle{siamplain}
\bibliography{references}
\end{document}